\documentclass[a4paper,10pt,reqno]{amsart}
\usepackage[utf8]{inputenc}
\usepackage{fullpage}
\usepackage{csquotes}
\usepackage{graphicx}
\usepackage[english]{babel}
\usepackage[T1]{fontenc}
\usepackage[utf8]{inputenc}
\usepackage{amsthm}
\usepackage{amssymb}
\usepackage{amsfonts}
\usepackage{dsfont}
\usepackage{mathrsfs}
\usepackage{xurl}
\usepackage{enumitem}
\usepackage{prettyref}
\usepackage{tikz}
\usepackage{mathtools}
\usepackage{orcidlink}
\usepackage{todonotes}

\usepackage[backend=biber,style=alphabetic,maxbibnames=20,maxalphanames=4,giveninits=true]{biblatex}
\newrefformat{defi}{Definition \ref{#1}}
\newrefformat{rem}{Remark \ref{#1}}
\newrefformat{sect}{Section \ref{#1}}
\newrefformat{app}{Appendix \ref{#1}}
\newrefformat{prop}{Proposition \ref{#1}}
\newrefformat{thm}{Theorem \ref{#1}}
\newrefformat{cor}{Corollary \ref{#1}}
\newrefformat{ex}{Example \ref{#1}}
\newrefformat{prob}{Problem \ref{#1}}

\newcommand{\vertiii}[1]{{\left\vert\kern-0.25ex\left\vert\kern-0.25ex\left\vert #1 
    \right\vert\kern-0.25ex\right\vert\kern-0.25ex\right\vert}}

\newcommand{\N}{\mathds{N}}
\newcommand{\Z}{\mathds{Z}}

\newcommand{\R}{\mathds{R}}
\newcommand{\C}{\mathds{C}}

\newcommand{\topo}{\tau}

\DeclareMathOperator{\re}{Re}

\newcommand{\calF}{\mathcal{F}}
\newcommand{\calL}{\mathcal{L}}
\newcommand{\calS}{\mathcal{S}}

\newcommand{\calK}{\mathcal{K}}
\newcommand{\calN}{\mathcal{N}}
\newcommand{\calQ}{\mathcal{Q}}

\newcommand{\semis}{\mathcal{P}}

\newcommand{\id}{\operatorname{id}}

\providecommand{\norm}[1]{\left\lVert#1\right\rVert}
\providecommand{\differential}{\mathrm{d}}
\renewcommand{\d}{\differential}

\DeclareMathOperator{\TextRe}{Re}

\renewcommand{\d}{\differential}

\newcommand{\euler}{\mathrm{e}}

\newcommand\rlim{
\mathchoice{\vcenter{\hbox{${\scriptstyle{+}}$}}}
{\vcenter{\hbox{$\scriptstyle{+}$}}}
{\vcenter{\hbox{$\scriptscriptstyle{+}$}}}
{\vcenter{\hbox{$\scriptscriptstyle{+}$}}}}

\newtheorem{lemma}{Lemma}[section]
\newtheorem{proposition}[lemma]{Proposition}
\newtheorem{theorem}[lemma]{Theorem}
\newtheorem{corollary}[lemma]{Corollary}

\theoremstyle{definition}
\newtheorem{definition}[lemma]{Definition}
\newtheorem{remark}[lemma]{Remark}

\begin{document}

%opening
\title[Locally equicontinuous families and bi-continuity]{Strongly continuous and locally equicontinuous families of operators and their relation to bi-continuity}
\author[K.\,Kruse]{Karsten Kruse\,\orcidlink{0000-0003-1864-4915}}
\address[Karsten Kruse]{University of Twente \\ Department of Applied Mathematics \\ P.O.~Box 217 \\ 7500 AE Enschede \\ The Netherlands
}
\email{k.kruse@utwente.nl}
\author[C.\,Seifert]{Christian Seifert\,\orcidlink{0000-0001-9182-8687}}
\address[Christian Seifert]{Hamburg University of Technology\\ 
Institute of Mathematics \\ 
Am Schwarzenberg-Campus~3 \\
21073 Hamburg \\
Germany\\
and
University of the Free State\\
Mathematics and Applied Mathematics,
Bloemfontein 9300\\
Republic of South Africa}
\email{christian.seifert@tuhh.de}

%\thanks{\ldots}



\date{\today}
\begin{abstract}
We study strongly continuous and locally equicontinuous families of operators on sequentially complete Hausdorff locally convex spaces. In case of Saks spaces, we relate the general notions to bi-continuity as well as equitightness. In this way, we recover and also generalise known results for special classes of operator families such as bi-continuous ($C$-)semigroups and ($C$-)cosine families by well-known results for the corresponding families in Hausdorff locally convex spaces.
\end{abstract}

\maketitle

\section{Introduction}

The aim of this article is to investigate strongly continuous and (locally) equicontinuous families of operators on Hausdorff locally convex spaces and to relate many different variants---particular classes of operators on particular spaces---present in the literature to this situation. In this way, on the one hand, we analyse and structure the various notions available. On the other hand, we are able to recover existing, as well as to obtain new, results in special situations by means of the corresponding results for operator families in Hausdorff locally convex spaces.

Strongly continuous and (locally) equicontinuous families of operators on Hausdorff locally convex spaces are classical objects in functional analysis and operator theory. Such families appear, for example, as fundamental solutions for abstract Cauchy problems (ACPs). Let us mention Stone's theorem for strongly continuous groups of unitary operators on Hilbert spaces \cite{Stone1930, vonNeumann1932} and the celebrated Hille--Yosida theorem for strongly continuous semigroups on Banach spaces \cite{Hille1948, Yosida1948}, for first order ACPs. Furthermore, we also mention the corresponding generation theorems for strongly continuous cosine families in Banach spaces \cite{Sova1966, Kisynski1972} for second order ACPs, as well as the analogs for higher order ACPs in Banach spaces \cite{Neubrander1986}. Standard monographs for semigroups and/or cosine families, also including applications, are e.g.~\cite{HillePhillips1957, Pazy1983, Goldstein1985, EngelNagel2000}.
Since the late 1950s, various generalisations have appeared in the literature. The theory for special classes of operator families on Banach or Hilbert spaces has been generalised to operator families in Hausdorff locally convex spaces, see \cite{Schwartz1958, Komatsu1964, Komura1968, Yosida1968} for strongly continuous and (locally) equicontinuous semigroups on sequentially complete, even locally complete \cite{Teichmann2002}, Hausdorff locally convex spaces, \cite{Fattorini1969a, Fattorini1969b, Konishi1972} for strongly continuous and (locally) equicontinuous cosine families on sequentially complete Hausdorff locally convex spaces. Another generalisation arises from so-called $C$-semigroups and $C$-cosine families, both on Banach spaces \cite{DaviesPang1987, KuoShaw1997a, KuoShaw1997b} as well as on sequentially complete Hausdorff locally convex spaces \cite{XiaoLiang1998}. Further, the adjoints of semigroups \cite{vanNeerven1992} and cosine families \cite{Shaw1986, Bochenek2000, Bochenek2002} on Banach spaces have also been studied.
Moreover, in the 1980s, integrated semigroups and integrated cosine families on Banach spaces have been introduced and studied \cite{Arendt1984, Neubrander1988, KellermanHieber1989} (to name a few). Starting from the 2000s, two further variants have been studied, namely bi-continuous semigroups and cosine families, where the operator families defined on Banach spaces are only strongly continuous with respect to a weaker Hausdorff locally convex topology, see \cite{Kuehnemund2001, Kuehnemund2003, Farkas2003, AlbaneseMangino2004, Farkas2011, BuddeFarkas2019} for semigroups and \cite{DuanSun2006, CangChenSong2014, Budde2024, Budde2025} for cosine families. Moreover, so-called integrable semigroups on norming dual pairs have been introduced by Kunze \cite{Kunze2009}. Recently, also locally equicontinuous semigroup on Hausdorff locally convex spaces which can also be equipped with an auxiliary norm (so the concept is closely related to Saks spaces and bi-continuous semigroups) \cite{Kraaij2016} as well as sequentially locally equicontinuous semigroups \cite{FedericoRosestolato2020} have been considered.

From the point of view of the underlying spaces, the most general situation treated is the (sequentially complete) Hausdorff locally convex case. This gives rise to the question of whether the above-mentioned variants are special situations of this general setup. 
Our aim is to actually show that this is indeed the case. More precisely, we will relate bi-continuous and locally (exponentially) bounded operator families in a one-to-one fashion with strongly continuous and locally (exponentially) equicontinuous families with respect to the mixed topology, as long as the mixed topology is C-sequential. Note that the assumption that a space is \emph{C-sequential} (i.e.~all convex sequentially open subsets are open) is very mild and is fulfilled in all relevant applications treated in the literature. In this way, we can recover as well as generalise (large parts of) the theory of bi-continuous operator families, say, by means of the well-established theory of operator families on locally convex spaces. 
Further, it refutes \cite[Remark 5 (b)]{Kuehnemund2003} (cf.~\cite[p.~1 \& 54]{Kuehnemund2001}) 
where it is said that it seems that the theory of locally equicontinuous semigroups on sequentially complete Hausdorff 
locally convex spaces has only few applications in contrast to the theory of locally bi-continuous semigroups.

We first review equicontinuity and sequential equicontinuity of operator families in \prettyref{sect:seq_equicont}. Indeed, on C-sequential spaces these two notions are equivalent. We then turn to special classes of operator families such as $C$-semigroups or $C$-cosine families as well as strongly continuous and (locally) equicontinuous families of operators in \prettyref{sect:strongly_cont_equicont}. In \prettyref{sect:bi-continuity}, we focus on bi-continuous operator families, in particular semigroups and cosine families. We then have all the necessary objects and tools available to relate the various notions with each other in \prettyref{sect:relations_bi-continuity_equicontinuity}; thus, this section contains our main theoretical results.
The final \prettyref{sect:applications} is devoted to applications where we show that we can recover existing results for, say, bi-continuous operator semigroups by means of their counterparts in Hausdorff locally convex spaces in a standard manner, as well as obtain new results by the same strategy. Here, we focus on (already existing) generation theorems, (existing as well as new) approximation theorems, and (new) representation formulas.

\section{Sequential equicontinuity}
\label{sect:seq_equicont}

For a Hausdorff locally convex space $(X,\topo)$ we denote by $\semis$ a system of 
seminorms generating the topology $\topo$. W.l.o.g.~$\semis$ is directed. If no confusion about the topology $\topo$ seems 
to be likely, we just write $X$ instead of $(X,\topo)$. 
On the other hand, if we want to emphasise the dependency of $\semis$ on $X$ or $\topo$, we write $\semis_{X}$ or $\semis_{\topo}$ 
instead of just $\semis$, respectively. 
We denote by $\calL(X,\topo)$ the space of continuous linear maps from $(X,\topo)$ to $(X,\topo)$. 
If $(X,\norm{\cdot})$ is a normed space, then $\topo_{\norm{\cdot}}$ denotes the topology induced by 
$\norm{\cdot}$ and we write $\calL(X)\coloneq\calL(X,\topo_{\norm{\cdot}})$ which we equip with operator norm 
$\norm{\cdot}_{\calL(X)}$. Further, we denote by $(X,\topo)'$ the topological dual space of $(X,\topo)$.

\begin{definition}[{\cite[pp.~1135, 1166]{FedericoRosestolato2020}}]
    Let $X$ and $Y$ be Hausdorff locally convex spaces and $\calS$ be a set of linear maps 
    from $X$ to $Y$. Then $\calS$ is called
    \begin{enumerate}[label=\upshape(\alph*), leftmargin=*]
        \item 
        \emph{equicontinuous} if for all $p\in\semis_{Y}$ there exist $q\in\semis_{X}$ and $K\geq 0$ such that 
        $\sup_{T\in\calS} p(Tx) \leq Kq(x)$ for all $x\in X$.
        \item
        \emph{sequentially equicontinuous} if for all $p\in\semis_{Y}$ and all sequences $(x_n)$ in $X$ such that $x_n\to 0$ 
        we have that $\sup_{T\in\calS} p(Tx_n)\to 0$.
    \end{enumerate}
\end{definition}

We note that the definition of (sequential) equicontinuity 
does not depend on the choice of the systems 
of seminorms $\semis_{X}$ and $\semis_{Y}$ that generate the Hausdorff locally convex topologies on $X$ and $Y$, respectively.
Clearly, if $\calS$ is equicontinuous, then it is also sequentially equicontinuous. 
We note the following characterisation of sequential equicontinuity whose proof we include for the sake of completeness. 
We recall that a subset $U$ of a topological space $X$ is called \emph{sequentially open} 
if every sequence $(x_{n})$ in $X$ converging to a point $x\in U$ is eventually in $U$, see \cite[p.~108]{Franklin1965}. 
In particular, open sets are sequentially open but the converse might not hold.

\begin{proposition}\label{prop:seq_equicont_equiv}
Let $X$ and $Y$ be Hausdorff locally convex spaces and $\calS$ a set of linear maps from $X$ to $Y$. 
Then the following assertions are equivalent.
\begin{enumerate}[label=\upshape(\alph*), leftmargin=*]
\item\label{it:seq_equicont_1} $\calS$ is sequentially equicontinuous. 
\item\label{it:seq_equicont_2} For all absolutely convex open neighbourhoods $V$ of zero in $Y$ 
there is an absolutely convex sequentially open neighbourhood $U$ of zero in $X$ such that $T(U)\subseteq V$ 
for all $T\in\calS$.
\item\label{it:seq_equicont_3} For all $p\in\semis_{Y}$ there are a sequentially continuous seminorm $q$ on $X$ and $K\geq 0$ such that $\sup_{T\in\calS}p(T(x))\leq Kq(x)$ for all $x\in X$.
\end{enumerate}
\end{proposition}

\begin{proof}
``\ref{it:seq_equicont_1}$\Rightarrow$\ref{it:seq_equicont_2}'' Let $V$ be an absolutely convex open neighbourhood 
of zero in $Y$. Then the Minkowski functional $p_{V}\colon Y\to\R$, $p_{V}(y)\coloneq\inf\{\lambda>0\;|\;y\in\lambda V\}$, 
is continuous and $V=\{y\in Y\;|\;p_{V}(y)<1\}$ by \cite[Lemma 22.3]{MeiseVogt1997}. It follows that 
$q(x)\coloneq \sup_{T\in\calS}p_{V}(Tx)$, $x\in X$, defines a seminorm on $X$ which is sequentially continuous since 
$\calS$ is sequentially equicontinuous by \ref{it:seq_equicont_1}. The set $U\coloneq\{x\in X\;|\;q(x)<1\}$ is an absolutely convex sequentially open neighbourhood of zero in $X$ because $q$ is a sequentially continuous seminorm. 
Furthermore, for $T\in\calS$ and $x\in U$, we have $p_{V}(Tx)\leq q(x)< 1$ and so $T(x)\in V$.

``\ref{it:seq_equicont_2}$\Rightarrow$\ref{it:seq_equicont_3}'' Let $p\in\semis_{Y}$ 
and define the set $V\coloneq\{y\in Y\;|\;p(y)<1\}$, which is an absolutely convex open neighbourhood 
of zero in $Y$. Then there is an absolutely convex sequentially open neighbourhood $U$ of zero in $X$ such that $T(U)\subseteq V$ 
for all $T\in\calS$ by \ref{it:seq_equicont_2}. We define the Minkowski functional $q_{U}\colon X\to\R$ 
by $q_{U}(x)\coloneq\inf\{\lambda>0\;|\;x\in\lambda U\}$ for $x\in X$. By the proof of 
\cite[Theorem 1, (3)$\Rightarrow$(4)]{Snipes1973}, $q_{U}$ is a sequentially continuous seminorm on $X$. 
Let $x\in X$ and $T\in\calS$. Let $\lambda>0$ such that $x\in\lambda U$. Then we have 
\[
Tx\in\lambda T(U)\subseteq \lambda V
\]
and thus $p(Tx)<\lambda$, which implies $p(Tx)\leq q_{U}(x)$ by the definition of $q_{U}$. 

``\ref{it:seq_equicont_3}$\Rightarrow$\ref{it:seq_equicont_1}'' Let $p\in\semis_Y$ and $(x_n)$ in $X$ 
such that $x_n\to 0$. By \ref{it:seq_equicont_3}, there are a sequentially continuous seminorm $q$ on $X$ and $K\geq 0$ such that $\sup_{T\in\calS}p(Tx)\leq Kq(x)$ for all $x\in X$. Since $q$ is sequentially continuous, 
we obtain that $q(x_n)\to 0$, implying $\sup_{T\in\calS}p(Tx_n)\to 0$. 
Therefore, $\calS$ is sequentially equicontinuous. 
\end{proof}

The equivalence ``\ref{it:seq_equicont_1}$\Leftrightarrow$\ref{it:seq_equicont_2}'' of 
\prettyref{prop:seq_equicont_equiv} is mentioned in \cite[p.~419]{Hsiang1988} without a proof. 
Next, we show that sequential equicontinuity and equicontinuity are equivalent if $X$ is C-sequential. 
Recall that a Hausdorff locally convex space $X$ is called \emph{C-sequential} if every convex 
sequentially open subset of $X$ is already open, see \cite[p.~273]{Snipes1973}. 
For instance, every bornological Hausdorff locally convex space is C-sequential by \cite[Theorem 8]{Snipes1973}. 
However, the spaces we are mainly interested in, namely Saks spaces $(X,\norm{\cdot},\topo)$ equipped with their mixed topology 
$\gamma$, see  \prettyref{defi:saks_mixed} and \prettyref{defi:seq_compl_C-seq_Saks}, are not bornological 
by \cite[I.1.15 Proposition]{Cooper1978} unless the mixed topology already coincides with the norm-topology. 
Nevertheless, many Saks spaces are still C-sequential (w.r.t.~the mixed topology), 
see \cite[Remarks 3.19, 3.20]{KruseSchwenninger2022} and \cite[3.5 Corollary \& Section 4]{Kruse2024a}.

\begin{theorem}\label{thm:seq_equicont_equicont_equiv}
Let $X$ and $Y$ be Hausdorff locally convex spaces, $X$ C-sequential and $\calS$ a set of linear maps from $X$ to $Y$. 
Then the following assertions are equivalent.
\begin{enumerate}[label=\upshape(\alph*), leftmargin=*]
\item\label{it:equicont_1} $\calS$ is sequentially equicontinuous. 
\item\label{it:equicont_2} $\calS$ is equicontinuous. 
\end{enumerate}
\end{theorem}

\begin{proof}
We only need to prove the implication ``\ref{it:equicont_1}$\Rightarrow$\ref{it:equicont_2}''. 
Let $\calS$ be sequentially equicontinuous and $p\in\semis_{Y}$. Due to \prettyref{prop:seq_equicont_equiv} 
there are a sequentially continuous seminorm $q$ on $X$ and $K\geq 0$ such that $\sup_{T\in\calS}p(T(x))\leq Kq(x)$ 
for all $x\in X$. Since $X$ is C-sequential, $q$ is continuous by \cite[Theorem 1]{Snipes1973}. 
Hence, $\calS$ is equicontinuous.
\end{proof}

\section{Strongly continuous and equicontinuous families}
\label{sect:strongly_cont_equicont}

We start this section by recalling some examples of families $\calS\coloneq (S(t))_{t\in I}$ of linear maps to 
which our results from \prettyref{sect:seq_equicont} can be applied to. 

\begin{definition}\label{defi:general_semigroups_cosine}
Let $X$ be a Hausdorff locally convex space, $(S(t))_{t\geq 0}$ a family of linear maps from $X$ to $X$ and 
$C\colon X\to X$ linear.
\begin{enumerate}[label=\upshape(\alph*), leftmargin=*]
\item\label{it:Csg} $(S(t))_{t\geq 0}$ is called a \emph{$C$-semigroup} if $S(0) = C$ and $S(t)S(s) = S(t+s)C$ 
for all $t,s\geq 0$. A $C$-semigroup is called a \emph{semigroup} if $C = \id$ where $\id$ denotes the identity on $X$.
\item\label{it:Ccf} $(S(t))_{t\geq 0}$ is called a \emph{$C$-cosine family} if $S(0) = C$ and $2S(t)S(s) = (S(t+s)+S(t-s))C$
for all $t\geq s\geq 0$. A $C$-cosine family is called a \emph{cosine family} if $C = \id$.
\end{enumerate}
Now, let $X$ be in addition sequentially complete and $(S(t))_{t\geq 0}$ strongly continuous, 
see \prettyref{defi:equicontinouscf} \ref{it:strong_cont0}. 
Further, let $\Gamma$ denote the Gamma function.
\begin{enumerate}[label=\upshape(\alph*), leftmargin=*]\setcounter{enumi}{2}
\item\label{it:alpha_semi} $(S(t))_{t\geq 0}$ is called an \emph{$\alpha$-times integrated $C$-semigroup} for $\alpha>0$ if 
$S(0) = 0$, $CS(t) = S(t)C$ and 
\[
S(t)S(s)x=\frac{1}{\Gamma(\alpha)}\left[\int_{0}^{s+t}-\int_{0}^{t}-\int_{0}^{s}\right](t+s-r)^{\alpha-1}S(r)Cx\,\d r
\]
for all $t,s\geq 0$ and $x\in X$ where the integrals are Riemann integrals w.r.t.~the topology of $X$. 
Further, a $C$-semigroup is also called a \emph{$0$-times integrated $C$-semigroup}.
\item\label{it:alpha_cosine} $(S(t))_{t\geq 0}$ is called an \emph{$\alpha$-times integrated $C$-cosine family} for $\alpha>0$ if 
$S(0) = 0$, $CS(t) = S(t)C$ and 
\begin{align*}
 S(t)S(s)x
&=\frac{1}{\Gamma(\alpha)}\left(\left[\int_{0}^{s+t}-\int_{0}^{t}-\int_{0}^{s}\right](t+s-r)^{\alpha-1}S(r)Cx\,\d r\right.\\
&\phantom{=}+\int_{|t-s|}^{t}(s-t+r)^{\alpha-1}S(r)Cx\,\d r + \int_{|t-s|}^{s}(t-s+r)^{\alpha-1}S(r)Cx\,\d r\\
&\phantom{=}+\left.\int_{0}^{|t-s|}(|t-s|+r)^{\alpha-1}S(r)Cx\,\d r\right)
\end{align*}
for all $t,s\geq 0$ and $x\in X$ where the integrals are (improper) Riemann integrals w.r.t.~the topology of $X$. 
Further, a $C$-cosine family is also called a \emph{$0$-times integrated $C$-cosine family}.
%\cite[Definition 1.1]{Kuo2006}
\end{enumerate}
\end{definition}

We note that the additional assumptions on $X$ and $(S(t))_{t\geq 0}$ in \prettyref{defi:general_semigroups_cosine} 
\ref{it:alpha_semi} and \ref{it:alpha_cosine} guarantee that the Riemann integrals there exist. 
Further, we remark that it is usually assumed in the literature that the operators $S(t)$ and $C$ 
in \prettyref{defi:general_semigroups_cosine} are continuous on $X$. 
Let us comment on the literature for different families of operators as in \prettyref{defi:general_semigroups_cosine}. 
For semigroups on Banach spaces we refer the reader to the monographs \cite{HillePhillips1957,Pazy1983,Goldstein1985,EngelNagel2000} 
and the references therein. 
Semigroups on Hausdorff locally convex spaces are for instance treated in 
\cite{Miyadera1959,Komatsu1964,Komura1968,Yosida1968,Ouchi1973,Babalola1974,Dembart1974,Choe1985,FedericoRosestolato2020,
Kraaij2016,KruseSchwenninger2022} and we also refer to the references therein.
Cosine families on Banach spaces appeared for example in \cite{Sova1966,DaPratoGiusti1967,TravisWebb1978,Goldstein1985,ArendtBHN2011} 
and on locally convex spaces in \cite{Fattorini1969a,Fattorini1969b,Konishi1972}. 
$C$-semigroups on Banach spaces are studied in \cite{DaPrato1966,DaviesPang1987,deLaubenfels1990,deLaubenfels1993}, 
while $C$-cosine families on Banach spaces as well as Hausdorff locally convex spaces are investigated in
\cite{Mandic1992,KuoShaw1997a,KuoShaw1997b}. 
In \cite{KellermanHieber1989,Neubrander1989}, $1$-times and $\alpha$-times integrated semigroups on Banach spaces for $\alpha\in\N$ 
are studied, and \cite{Hieber1991,MiyaderaOkuboTanaka1993} considered $\alpha$-times integrated semigroups for $\alpha\geq 0$, 
again on Banach spaces. 
Moreover, $\alpha$-times integrated $C$-semigroups and $C$-cosine families on Banach spaces for $\alpha\in\N$ are studied 
in \cite{LiShaw2000}, and generalised to families on Hausdorff locally convex spaces in \cite{LiShaw1993,LiShaw1997,ShawLi1995,WangHuang1997}. 
In \cite{KuoShaw2000,LiShaw2003} and \cite{Kuo2006}, $\alpha$-times integrated $C$-semigroups and $C$-cosine families, respectively, 
on Banach spaces for $\alpha\geq 0$ are investigated, and again generalised to families on Hausdorff locally convex spaces 
in \cite{XiaoLiang1998}.
Lastly, \cite{Kostic2012,Kostic2015} considers the case of even more general families on Hausdorff locally convex spaces. 
Furthermore, if $I\coloneq [0,\infty)$ in \prettyref{defi:general_semigroups_cosine} \ref{it:Csg} is replaced by $\R$, 
then $(S(t))_{t\in\R}$ is called a \emph{$C$-group}, see \cite[p.~37]{EngelNagel2000} in the case $C\coloneq\id$. 
We also note that one can always extend the definition of 
a cosine family in \prettyref{defi:general_semigroups_cosine} \ref{it:Ccf} from $I\coloneq [0,\infty)$ to $\R$ 
by setting $S(-t)\coloneq S(t)$ for $t<0$ if the cosine family is strongly continuous on $[0,\infty)$, 
see \prettyref{rem:commutative_cosine}. 
By replacing $I\coloneq [0,\infty)$ in \prettyref{defi:general_semigroups_cosine} \ref{it:Csg} and \ref{it:Ccf} by $\N_{0}$ 
and $\Z$, respectively, we get so-called \emph{discrete $C$-semigroups} and \emph{discrete $C$-cosine families}, respectively, 
see \cite[p.~762]{Gibson1972} and \cite[p.~567]{SchwenningerZwart2015} in the case that $X$ is a Banach space and $C\coloneq\id$.
There are also \emph{local and convoluted} versions of \prettyref{defi:general_semigroups_cosine} \ref{it:alpha_semi} and 
\ref{it:alpha_cosine} where one replaces in the local version $I\coloneq [0,\infty)$ by $[0,t_{0})$ for some $t_{0}>0$ 
and considers $0\leq s,t\leq s+t<t_{0}$ for the defining functional equation, 
and in the convoluted version one replaces the integral kernel $\zeta\mapsto \zeta^{\alpha-1}/\Gamma(\alpha)$ 
by some other locally integrable function, see \cite{LiShaw2007,Kostic2018,Kuo2015}.

\begin{definition}\label{defi:evol_family}
Let $X$ be a Hausdorff locally convex space, $I$ a set and 
$(S(t))_{t\in I}$ a family of linear maps from $X$ to $X$.
\begin{enumerate}[label=\upshape(\alph*), leftmargin=*]
\item\label{it:evol_fam} If $I=\{(t,s)\in[0,\infty)^2\;|\;t\geq s\}$, then $(S(t,s))_{(t,s)\in I}$ is called an 
\emph{evolution family} if $U(t,t)=\id$ for all $t\in [0,\infty)$ and $U(t,s)U(s,r)=U(t,r)$ for all $(t,s), (s,r)\in I$.  
\item\label{it:resolvent} If $I\subseteq\C$, then $(S(t))_{t\in I}$ is called a \emph{pseudo-resolvent} 
if $(t-s)S(t)S(s)=S(t)-S(s)$ for all $t,s\in I$. 
\end{enumerate}
\end{definition}

We note again that it is usually assumed in the literature that the operators $S(t)$ in \prettyref{defi:evol_family} 
are continuous on $X$. Sometimes $[0,\infty)$ in \prettyref{defi:evol_family} \ref{it:evol_fam} is replaced by $\R$, 
see \cite[Definition 1.1]{LatushkinRandolphSchnaubelt1998}.
There is also a \emph{local} version of it where $[0,\infty)$ is replaced by $[0,t_{0}]$ for some $t_{0}>0$, see 
\cite[Definition 2.4]{BuddeSeifert2024}, \cite[Theorem 2]{Kato1953} and \cite[Equation (3.5)]{Wenzel1985}. 
The definition of a pseudo-resolvent with continuous $S(t)$ is for example given in \cite[Definition, p.~215]{Yosida1968}. 
Next, we introduce strong continuity and variations of (sequential) equicontinuity of a family of linear maps 
$(S(t))_{t\in I}$ that take the set $I$ into account.

\begin{definition}
  \label{defi:equicontinouscf}
  Let $(X,\topo)$ be a Hausdorff locally convex space, $I$ a set, $\calF$ a set of functions 
  from $I$ to $[0,\infty)$, and $(S(t))_{t\in I}$ 
  a family of linear maps from $X$ to $X$. Then $(S(t))_{t\in I}$ is called 
  \begin{enumerate}[label=\upshape(\alph*), leftmargin=*]
    \item\label{it:strong_cont0}
      \emph{strongly continuous} if $I$ is a Hausdorff topological space and for all $x\in X$ the map $I\ni t\mapsto S(t)x\in X$ is continuous.
    \item\label{it:equicontinousfam1}
      \emph{(sequentially) locally equicontinuous} if $I$ is a Hausdorff topological space and 
      $\{S(t)\;|\; t\in M\}$ is (sequentially) equicontinuous for all compact sets $M\subseteq I$.
    \item\label{it:equicontinousfam2}
      \emph{(sequentially) $\calF$-equicontinuous} if there is $k\in\calF$ such that $\{k(t)S(t)\;|\; t\in I\}$ 
      is (sequentially) equicontinuous.
  \end{enumerate}
  If we want to emphasise the dependency of strong continuity and equicontinuity on $\topo$, 
  we write $\topo$-strongly continuous and (locally, sequentially, $\calF$-) $\topo$-equicontinuous
  instead of just strongly continuous and (locally, sequentially, $\calF$-) equicontinuous, respectively. 
%  Further, we drop the word ``locally'' if property \ref{it:equicontinoussg1} holds for the set 
%  $\{S(t)\;|\; t\in I\}$.
\end{definition}

We point out that we do not assume that the maps $S(t)$ of a strongly continuous family belong to 
$\calL(X,\topo)$ in contrast to other definitions of, for instance, semigroups and cosine families in the literature, 
see e.g.~\cite[Definition 1.1]{Komura1968}, \cite[p.~76]{Fattorini1969a} and \cite[Definition 2.1]{Konishi1972}. 
For semigroups and cosine families, our definition of (sequential) local equicontinuity coincides with the one given in 
\cite[Definition 1.1]{Komura1968}, \cite[Definition 3.3]{FedericoRosestolato2020} and \cite[p.~446]{Konishi1972}.
Clearly, $(S(t))_{t\in I}$ is (sequentially) equicontinuous if and only if it is (sequentially)
$\calF_{\operatorname{const}}$-equicontinuous where $\calF_{\operatorname{const}}$ denotes the set of all constant strictly positive 
functions on $I$. Further, $(S(t))_{t\in I}$ is (sequentially) locally equicontinuous if and only if 
$(\chi_{M}(t)S(t))_{t\in I}$ is equicontinuous for all compact sets $M\subseteq I$ where $\chi_{M}$ denotes the 
characteristic function of $M$.

\begin{definition}\label{defi:exp_pol_equicont}
Let $X$ be a Hausdorff locally convex space, $I\subseteq\R$ or $I\subseteq\R^2$ and 
$(S(t))_{t\in I}$ a family of linear maps from $X$ to $X$.  
\begin{enumerate}[label=\upshape(\alph*), leftmargin=*]
\item\label{it:exp_pol_equicont_1} If $I=[0,\infty)$, then $(S(t))_{t\in I}$ is called 
\emph{(sequentially) exponentially equicontinuous} if it is (sequentially) $\calF_{\exp}$-equicontinuous for 
\[
\calF_{\exp}\coloneq\{[0,\infty)\ni t\mapsto \euler^{\alpha t}\in\R\;|\;\alpha\in\R\}.
\]
\item\label{it:exp_pol_equicont_2} If $I=[0,\infty)$, then $(S(t))_{t\in I}$ is called 
\emph{(sequentially) polynomially equicontinuous} if it is (sequentially) $\calF_{\operatorname{pol}}$-equicontinuous for
\[ 
\calF_{\operatorname{pol}}\coloneq\{[0,\infty)\ni t\mapsto (1+t^\alpha)^{-1}\in\R\;|\;\alpha>0\}. 
\]
\item\label{it:exp_pol_equicont_3} If $I=\{(t,s)\in[0,\infty)^2\;|\;t\geq s\}$, then $(S(t))_{t\in I}$ is called 
\emph{(sequentially) exponentially equicontinuous} if it is (sequentially) $\calF_{\exp ,2}$-equicontinuous for 
\[
\calF_{\exp ,2}\coloneq\{ I \ni (t,s)\mapsto \euler^{\alpha (t-s)}\in\R\;|\;\alpha\in\R\}.
\]
\end{enumerate}
\end{definition}

Our definitions of exponential equicontinuity coincide for $C$-semigroups, $C$-cosine families and evolution families 
with the ones given in \cite[p.~29--30]{LiShaw1993} and \cite[p.~495]{LatushkinRandolphSchnaubelt1998}, 
in the latter case see \prettyref{defi:bound_lin_norm} as well. 
Exponential equicontinuity is also called \emph{quasi-equicontinuity}, see e.g.~\cite[p.~294]{Choe1985}. 
We note that $\alpha>0$ in the definition of $\calF_{\operatorname{pol}}$ in \prettyref{defi:exp_pol_equicont} \ref{it:exp_pol_equicont_2} need not be in $\N$ even though the word ``polynomial'' would suggest it. 
Moreover, our definition of (sequential) polynomial equicontinuity coincides for strongly continuous semigroups on Banach spaces with 
\emph{polynomial boundedness}, see \cite[p.~438]{EisnerZwart2007} (cf.~\cite[Theorem 1.1]{RozendaalVeraar2018}) 
and \prettyref{defi:bound_lin_norm}.

We make the following observation about sequential exponential equicontinuity of $C$-cosine families which generalises 
\cite[Proposition 3.14.6]{ArendtBHN2011} where cosine families on Banach spaces are considered. 

\begin{proposition}\label{prop:cosine_type}
Let $X$ be a Hausdorff locally convex space, $C\colon X\to X$ linear and $(S(t))_{t\geq 0}$ a $C$-cosine family on $X$. 
\begin{enumerate}[label=\upshape(\alph*), leftmargin=*]
\item\label{it:cosine_type_1} If $x\in X$ is such that $\lim_{t\to\infty}S(t)x=0$ and $\lim_{t\to\infty}S(t)Cx=0$ 
and $(S(t))_{t\geq 0}$ is sequentially equicontinuous, then $C^2x=0$ where $C^2x\coloneq CCx$.
\item\label{it:cosine_type_2} If there is $k\colon [0,\infty)\to (0,\infty)$ such that $\lim_{t\to\infty}k(t)=\infty$, 
$\inf_{t\geq 0} k(t)>0$ and $(k(t)S(t))_{t\geq 0}$ is sequentially equicontinuous, then $C^2x=0$ for all $x\in X$. 
\item\label{it:cosine_type_3} If $X\neq\{0\}$, $C^2$ is injective and there is $\alpha\in\R$ such $(\euler^{-\alpha t}S(t))_{t\geq 0}$ 
is sequentially equicontinuous, then $\alpha\geq 0$.
\end{enumerate}
\end{proposition}
\begin{proof}
    \ref{it:cosine_type_1} We have $2S(t)^2x=S(2t)Cx+C^2x$ for all $t\geq 0$ where $S(t)^2x\coloneq S(t)S(t)x$. 
    Let $p\in\semis_X$. Since $(S(t))_{t\geq 0}$ is sequentially equicontinuous, there are a sequentially continuous seminorm $q$ on $X$ and 
    $K\geq 0$ such that $p(S(t)^2x)\leq Kq(S(t)x)$ for all $t\geq 0$ by \prettyref{prop:seq_equicont_equiv}, 
    which implies that $\lim_{t\to\infty}S(t)^2 x=0$ as $\lim_{t\to\infty}S(t)x=0$.
    Hence, we get 
    \[
    C^2x=\lim_{t\to\infty}S(2t)Cx+C^2x=2\lim_{t\to\infty}S(t)^2 x=0
    \]
    since $\lim_{t\to\infty}S(t)Cx=0$.
    
    \ref{it:cosine_type_2} Since $\lim_{t\to\infty}k(t)=\infty$, there is $t_0\geq 0$ such that $k(t)\geq 1$ for all $t\geq t_0$. 
    Further, we note that $K_1\coloneq\inf_{t\in [0,t_0]}k(t)\geq \inf_{t\geq 0} k(t)>0$. Let $p\in\semis_X$. Then we have 
    \[
    p(S(t)x)= \frac{1}{k(t)}p(k(t)S(t)x)\leq \frac{1}{K_1}p(k(t)S(t)x) \quad (t\in [0,t_0],\,x\in X)
    \]
    and $p(S(t)x)\leq p(k(t)S(t)x)$ for all $t\geq t_0$ and $x\in X$. Thus, the sequential equicontinuity of $(k(t)S(t))_{t\geq 0}$ implies 
    that $(S(t))_{t\geq 0}$ is sequentially equicontinuous. Moreover, by \prettyref{prop:seq_equicont_equiv} and the sequential equicontinuity 
    of $(k(t)S(t))_{t\geq 0}$, there are also a sequentially continuous seminorm $q$ on $X$ and $K\geq 0$ such that 
    \[
    p(S(t)x)\leq \frac{K}{k(t)}q(x)\quad (t\geq 0,\,x\in X),
    \]
    yielding that $\lim_{t\to\infty}S(t)x=0$ for all $x\in X$ because $\lim_{t\to\infty}k(t)=\infty$. 
    In particular, $\lim_{t\to\infty}S(t)Cx=0$ for all $x\in X$ as well. We conclude that $C^2x=0$ for all $x\in X$ by part \ref{it:cosine_type_1}.
    
    \ref{it:cosine_type_3} Suppose there is $\alpha\in\R$ with $\alpha<0$ such that $(\euler^{-\alpha t}S(t))_{t\geq 0}$ is 
    sequentially equicontinuous. Then it follows from part \ref{it:cosine_type_2} applied to $k(t)\coloneq \euler^{-\alpha t}$ for $t\geq 0$ 
    that $C^2x=0$ for all $x\in X$. Since $C^2$ is injective, we obtain that $X=\{0\}$, which is a contradiction.
\end{proof}

\begin{proposition}\label{prop:loc_equicont_sg_cf}
Let $X$ be a Hausdorff locally convex space and $(S(t))_{t\geq 0}$ a semigroup or cosine family on $X$. 
If $(S(t))_{t\geq 0}$ is a cosine family, then we assume in addition that $S(t)S(s)=S(s)S(t)$ for all $t,s\geq 0$.
Then $(S(t))_{t\geq 0}$ is (sequentially) locally equicontinuous if and only if there is $t_0>0$ such that $(S(t))_{t\in [0,t_0]}$ 
is (sequentially) equicontinuous.
\end{proposition}

\begin{proof}
We only need to prove the implication ``$\Leftarrow$''. Let there be $t_0>0$ such that $(S(t))_{t\in [0,t_0]}$ is 
(sequentially) equicontinuous. 
We claim that $(S(t))_{t\in [0, nt_{0}]}$ is (sequentially) equicontinuous for all $n\in\N$, which we prove by induction. 
For $n=1$ this is clear.

(i) Let $(S(t))_{t\geq 0}$ be a semigroup. Assume that $n\in\N$ is such that $(S(t))_{t\in [0,n t_{0}]}$ is 
(sequentially) equicontinuous. 
Let $p\in\semis_X$. By the (sequential) equicontinuity of 
$(S(t))_{t\in [0,n t_{0}]}$ and \prettyref{prop:seq_equicont_equiv}, 
there are a (sequentially) continuous seminorm $q$ on $X$ 
and $K\geq 0$ such that for all $t\in ((n-1)t_0,nt_{0}]$ and $x\in X$ it holds that
\[
     p(S(t+t_{0})x)
=    p(S(t)S(t_0)x)
\leq Kq(S(t_0)x),
\]
which proves our claim by the (sequential) continuity of $S(t_0)$.

(ii) Let $(S(t))_{t\geq 0}$ be a cosine family. Let $p\in\semis_X$. 
First, we show that $(S(t))_{t\in [0,2t_{0}]}$ is (sequentially) equicontinuous. 
By the (sequential) equicontinuity of $(S(t))_{t\in [0,t_{0}]}$ and \prettyref{prop:seq_equicont_equiv}, 
there are a (sequentially) continuous seminorm $q$ on $X$ 
and $K\geq 0$ such that for all $t\in (0,t_{0}]$ and $x\in X$ it holds that 
\[
      p(S(t+t_{0})x)
  =   p(2S(t_0)S(t)x-S(t_0-t)x)
  =   p(2S(t)S(t_0)x-S(t_0-t)x)
 \leq 2Kq(S(t_0)x)+Kq(x)
\]
where we used the commutativity assumption in the second equation. 
This proves our claim for $n=2$ by the (sequential) continuity of $S(t_0)$.
Assume that $n\in\N$, $n\geq 2$, is such that $(S(t))_{t\in [0,n t_{0}]}$ is (sequentially) equicontinuous. 
The proof is now quite similar to case (i). 
Again, by the (sequential) equicontinuity of $(S(t))_{t\in [0,n t_{0}]}$ 
and \prettyref{prop:seq_equicont_equiv}, there are a (sequentially) continuous seminorm $q$ on $X$ 
and $K\geq 0$ such that for all $t\in ((n-1)t_0,nt_{0}]$ and $x\in X$ it holds that
\[
      p(S(t+t_{0})x)
  =   p(2S(t)S(t_0)-S(t-t_0)x)
 \leq 2Kq(S(t_0)x)+Kq(x),
\]
which proves our claim by the (sequential) continuity of $S(t_0)$.
\end{proof}

\begin{remark}\label{rem:commutative_cosine}
    Let $X$ be a Hausdorff locally convex space and $(S(t))_{t\geq 0}$ a cosine family on $X$. 
    The commutativity condition on the cosine family $(S(t))_{t\ge 0}$ in \prettyref{prop:loc_equicont_sg_cf} 
    is for instance fulfilled if $(S(t))_{t\ge 0}$ is strongly continuous, see the proof of \cite[Lemma 3.14.3 b)]{ArendtBHN2011} 
    which also works for Hausdorff locally convex spaces instead of just Banach spaces. 
\end{remark}

In the case that $(X,\topo)$ is equentially complete and $(S(t))_{t\geq 0}$ a semigroup in 
$\calL(X,\topo)$ \prettyref{prop:loc_equicont_sg_cf} is stated in \cite[Remarks, p.~128]{Dembart1974} without a proof for 
the locally equicontinuous case. In some cases one can also get local equicontinuity from strong continuity, which we explore next.

We recall that a Hausdorff locally convex space $(X,\topo)$ is said to have the \emph{convex compactness property (ccp)} 
if the closure of the absolutely convex hull of every compact set is compact, see \cite[9-2-8 Definition]{Wilansky1978}. 
By \cite[9-2-10 Example]{Wilansky1978}, every quasi-complete Hausdorff locally convex space has ccp but the converse is not true. 
Using the convex compactness property, we have the following special version of \cite[Lemma 3.8]{Kunze2009} 
for dual pairs $\langle X,Y\rangle$ of linear spaces $X$ and $Y$ (over $\R$ or $\C$). 
For dual pairs, we refer to e.g.~\cite[Chapter 8]{Jarchow1981} and \cite[Section 10.3 \& Chapters 20, 21]{Koethe1969}. 
We denote by $\sigma(X,Y)$ and $\sigma(Y,X)$ the weak topology on $X$ and $Y$, respectively, and by 
$\mu(X,Y)$ and $\mu(Y,X)$ the Mackey topology on $X$ and $Y$, respectively. 
Analysing the proof of \cite[Lemma 3.8]{Kunze2009}, we have the following result, which was in the case of semigroups already observed 
in \cite[Lemma 3.2]{Kraaij2016} and generalises \cite[Proposition 1.1]{Komura1968}.

\begin{proposition}\label{prop:str_M_str_cont_loc_equicont}
Let $\langle X,Y\rangle$ be a dual pair of linear spaces, $I$ a Hausdorff topological space 
and $(S(t))_{t\in I}$ a $\mu(X,Y)$-strongly continuous family in $\calL(X,\sigma(X,Y))$. 
If $(Y,\sigma(Y,X))$ has ccp, then $(S(t))_{t\in I}$ is locally $\mu(X,Y)$-equicontinuous.
\end{proposition}

\begin{proof}
Let $M\subseteq I$ be compact and $\calK\subseteq Y$ be $\sigma(Y,X)$-compact. We set 
$\calQ\coloneq \{S(t)'y\;|\;t\in M,y\in \calK\}$ where $S(t)'\colon Y\to Y$ for $t\in M$ is the dual map of 
$S(t)$ defined by $(S(t)'y)(x)\coloneq \langle S(t)x,y\rangle=y(S(t)x)$ for $y\in Y$ and $x\in X$. 
We point out that $S(t)'y\in Y$ due to the assumption $S(t)\in\calL(X,\sigma(X,Y))$, the Mackey--Arens theorem 
and $y\in Y$. By the proof of \cite[Lemma 3.8]{Kunze2009}\footnote{\cite[Lemma 3.8]{Kunze2009} is phrased for norming dual pairs 
$\langle X,Y\rangle$ but that they are norming is not relevant for the proof.}, the set $\calQ$ is $\sigma(Y,X)$-compact. 
Let $\overline{\operatorname{acx}}(\calQ )$ denote the $\sigma(Y,X)$-closure of 
the absolutely convex hull $\operatorname{acx}(\calQ )$ of $\calQ$. 
If $(Y,\sigma(Y,X))$ has ccp, then the absolutely convex set $\overline{\operatorname{acx}}(\calQ )$ 
is $\sigma(Y,X)$-compact. Further, we observe that 
\[
  \sup_{y\in\calK}|\langle S(t)x,y\rangle|
= \sup_{y\in\calK}|\langle x,S(t)'y\rangle|
\leq \sup_{z\in\calQ}|\langle x,z\rangle|
\leq \sup_{z\in\overline{\operatorname{acx}}(\calQ )}|\langle x,z\rangle|
\]
for all $t\in M$ and $x\in X$. Hence, $(S(t))_{t\in M}$ is $\mu(X,Y)$-equicontinuous.
\end{proof}

Apart from the requirement of $(Y,\sigma(Y,X))$ having ccp, we note that the restriction to the Mackey topology and 
the a-priori condition that all $S(t)$ should be continuous w.r.t.~$\sigma(X,Y)$ are also not so desirable for us.

\begin{remark}\label{rem:conditions_str_M_str_cont_loc_equicont}
Let $\langle X,Y\rangle$ be a dual pair of linear spaces. 
\begin{enumerate}[label=\upshape(\alph*), leftmargin=*]
\item\label{it:weak_Mackey_cont} Then $\calL(X,\sigma(X,Y))=\calL(X,\mu(X,Y))$ as linear spaces by 
\cite[Lemmas 23.28, 23.29]{MeiseVogt1997} and the Mackey--Arens theorem. 
If $\topo$ is a Hausdorff locally convex topology on $X$ and $S\in\calL(X,\topo)$, then $S\in\calL(X,\sigma(X,(X,\topo)'))$ 
by \cite[Proposition 23.30 (a)]{MeiseVogt1997}.
\item\label{it:weak_ccp} The following assertions are equivalent.
\begin{enumerate}[label=\upshape(\roman*), leftmargin=*, widest=iii]
\item\label{it:weak_ccp_1} $(Y,\sigma(Y,X))$ has ccp.
\item\label{it:weak_ccp_2} The system of seminorms given by 
\[
p_{\calK}(x)\coloneq\sup_{y\in\calK}|\langle x,y\rangle| \quad(x\in X)
\] 
for $\sigma(Y,X)$-compact $\calK\subseteq Y$ generates the Mackey topology $\mu(X,Y)$.
\item\label{it:weak_ccp_3} Every $\sigma(Y,X)$-compact subset of $Y$ is $\mu(X,Y)$-equicontinuous.
\end{enumerate}
Indeed, the equivalence ``\ref{it:weak_ccp_1}$\Leftrightarrow$\ref{it:weak_ccp_2}'' is \cite[Problems 9-2-4]{Wilansky1978}. 
The implication ``\ref{it:weak_ccp_1}$\Rightarrow$\ref{it:weak_ccp_3}'' follows from the fact that the $\sigma(Y,X)$-closure of 
the absolutely convex hull $\overline{\operatorname{acx}}(\calK)$ is $\sigma(Y,X)$-compact for every $\sigma(Y,X)$-compact 
subset $\calK\subseteq Y$ if $(Y,\sigma(Y,X))$ has ccp, and the estimate 
$|\langle x,y\rangle|\leq p_{\overline{\operatorname{acx}}(\calK)}(x)$ for all $x\in X$ and $y\in \calK$. 
Let us turn to the implication ``\ref{it:weak_ccp_3}$\Rightarrow$\ref{it:weak_ccp_2}''. 
Let $\calK\subseteq Y$ be $\sigma(Y,X)$-compact. By \ref{it:weak_ccp_3}, $\calK$ is $\mu(X,Y)$-equicontinuous, 
implying that $\operatorname{acx}(\calK)$ is $\mu(X,Y)$-equicontinuous. By the Alaoglu--Bourbaki theorem, 
see \cite[Theorem 9-1-10]{Wilansky1978}, $\overline{\operatorname{acx}}(\calK)$ is $\sigma(Y,X)$-compact, 
which implies the statement due to the estimate $p_{\calK}(x)\leq  p_{\overline{\operatorname{acx}}(\calK)}(x)$ for all $x\in X$.
\item\label{it:strong_Mackey} Let $\topo$ be a Hausdorff locally convex topology on $X$ and set $X'\coloneq(X,\topo)'$. 
Then $(X,\topo)$ is called a \emph{strong Mackey space} if all $\sigma(X',X)$-compact subsets 
of $X'$ are $\topo$-equicontinuous, see \cite[p.~317]{Sentilles1972}. 
In particular, strong Mackey spaces are \emph{Mackey spaces}, 
i.e.~their topology $\topo$ coincides with the Mackey topology $\mu(X,X')$, see \cite[\S 21, 4.(1)]{Koethe1969}. 
So, a strong Mackey space $(X,\topo)$ has property \ref{it:weak_ccp_3} of 
\prettyref{rem:conditions_str_M_str_cont_loc_equicont} \ref{it:weak_ccp}
with $Y\coloneq X'=(X,\topo)'$. 
\item\label{it:sufficient_strong_Mackey_1} Let $\topo$ be a Hausdorff locally convex topology on $X$. 
By \cite[Chap.~IV, 3.4]{Schaefer1971}, $(X,\topo)$ is a Mackey space if it is barrelled. 
By \cite[11.1.5 Proposition]{Jarchow1981}, $(X',\sigma(X',X))$ is quasi-complete if and only if $(X,\mu(X,X'))$ is barrelled. 
So, barrelled Hausdorff locally convex spaces are strong Mackey spaces by 
\prettyref{rem:conditions_str_M_str_cont_loc_equicont} \ref{it:weak_ccp}, cf.~\cite[Proposition 3.3]{Kraaij2016}. 

By \cite[24, 5.(4')]{Koethe1969} and \cite[8.2.5 Proposition]{Jarchow1981}, the space
$(X',\sigma(X',X))$ has ccp if and only if the $\sigma(X',X)$-closure (equivalently, $\mu(X',X)$-closure) of the absolutely convex hull 
of any $\sigma(X',X)$-compact subset of $X'$ is $\mu(X',X)$-complete. 
So, a Mackey space $(X,\topo)$ is a strong Mackey space if $(X',\mu(X',X))$ is quasi-complete, see \cite[20, 6.(3)]{Koethe1969}. 
In particular, a sequentially complete \emph{Mackey--Mazur space} $(X,\topo)$, i.e.~$(X,\topo)$ is a Mackey and a Mazur space, 
is a strong Mackey space since $(X',\mu(X',X))$ is complete in this case, see (the proof of) \cite[Proposition 3.3]{Kraaij2016}.
We recall that $(X,\topo)$ is called a \emph{Mazur space} if all linear sequentially 
$\topo$-continuous functionals on $X$ are already $\topo$-continuous. 
Examples of Mazur spaces are C-sequential spaces by \cite[Theorem 7.4]{Wilansky1981}.
\item\label{it:sufficient_strong_Mackey_2} Let $\topo$ be a Hausdorff locally convex topology on $X$ such that $(X,\topo)$ is 
quasi-complete. Then $(X,\mu(X,X'))$ is quasi-complete as $\topo\subseteq\mu(X,X')$ by the Mackey--Arens theorem and as a subset of $X$ 
is $\mu(X,X')$-bounded if and only if it is $\topo$-bounded by \cite[8.3.4 Theorem]{Jarchow1981}. 
Thus, $(X,\sigma(X,X'))$ has ccp by \prettyref{rem:conditions_str_M_str_cont_loc_equicont} \ref{it:sufficient_strong_Mackey_1} 
since $X=(X',\sigma(X',X))'$ as linear spaces. 
Hence, for a Hausdorff topological space $I$, any $\mu(X',X)$-strongly continuous family $(S(t))_{t\in I}$ 
in $\calL(X',\sigma(X',X))$ is locally $\mu(X',X)$-equicontinuous by \prettyref{prop:str_M_str_cont_loc_equicont}, 
cf.~\cite[Corollary 3.10]{Kunze2009}.
\end{enumerate}
\end{remark}

Again (see the comment above \prettyref{thm:seq_equicont_equicont_equiv}), the spaces we are mainly interested in, 
namely Saks spaces equipped with their mixed topology, 
are not barrelled by \cite[I.1.15 Proposition]{Cooper1978} unless the mixed topology already coincides with the norm-topology. 
However, among them are some that are sequentially complete Mackey--Mazur spaces 
(w.r.t.~the mixed topology), see \cite[Remark 3.20 \& Section 4]{KruseSchwenninger2022}.
Modifying the proof of \cite[Theorem 3.9]{Kunze2009}, which is concerned with exponentially bounded semigroups, 
we obtain the following result regarding $\calF$-equiconti\-nuity.

\begin{proposition}\label{prop:equicont_from_loc_eqicont}
Let $\langle X,Y\rangle$ be a dual pair of linear spaces, $I$ a Hausdorff topological space, 
$\calF$ a set of functions from $I$ to $[0,\infty)$ 
and $(S(t))_{t\in I}$ a $\mu(X,Y)$-strongly continuous family in $\calL(X,\sigma(X,Y))$. 
If $(Y,\sigma(Y,X))$ has ccp, $I$ is locally compact and noncompact, and there exists a continuous 
function $k\in\calF$ such that $S_{k}x\coloneq\mu(X,Y)$-$\lim_{t\to\infty}k(t)S(t)x$ exists for all $x\in X$ 
and $S_{k}\in\calL(X,\sigma(X,Y))$, then $(k(t)S(t))_{t\in I}$ is $\mu(X,Y)$-strongly continuous and $\mu(X,Y)$-equicontinuous. 
In particular, $(S(t))_{t\in I}$ is $\calF$-$\mu(X,Y)$-equicontinuous in this case.
\end{proposition}

\begin{proof}
We denote by $I_{\infty}\coloneq I\cup\{\infty\}$ the one-point compactification of the locally compact, 
noncompact Hausdorff topological space $I$. Further, we define for each $x\in X$ the map 
$S_{\infty}(\cdot)x\colon I_{\infty} \to X$ by 
\[
S_{\infty}(t)x
\coloneq 
\begin{cases}
k(t)S(t)x &,\; t\in I,\\
S_{k}x &,\; t=\infty .
\end{cases}
\]
We note that $S_{\infty}(t)\in \calL(X,\sigma(X,Y))$ for all $t\in I_{\infty}$ since $(S(t))_{t\in I}$ and $S_{k}$ 
are in $\calL(X,\sigma(X,Y))$ and $k(t)S(t)=S(t)k(t)$. Moreover, $(S_{\infty}(t))_{t\in I_{\infty}}$ 
is $\mu(X,Y)$-strongly continuous by definition of $S_{k}$ and since $(S(t))_{t\in I}$ is $\mu(X,Y)$-strongly 
continuous and $k$ continuous. 
Finally, our statement follows from \prettyref{prop:str_M_str_cont_loc_equicont} applied 
to $(S_{\infty}(t))_{t\in I_{\infty}}$ and the compactness of $I_{\infty}$. 
\end{proof}

Next, we turn to the relation between (sequential) $\calF$-equicontinuity and (sequential) local equicontinuity. 
Let $I$ be a Hausdorff topological space and $\calF$ a set of functions from $I$ to $[0,\infty)$. 
In some cases (sequential) $\calF$-equicontinuity already implies (sequential) local equicontinuity. 
We say that a function $k\in\calF$ is \emph{locally bounded away from zero} if $\inf_{t\in M}k(t)>0$ 
for all compact sets $M\subseteq I$, see e.g.~\cite[p.~26]{Kruse2023}. 
In particular, if $\calF$ consists of strictly positive continuous functions, 
then all its elements are locally bounded away from zero.

\begin{proposition}\label{prop:loc_bound_away_zero_equicont}
Let $X$ be a Hausdorff locally convex space, $I$ a Hausdorff topological space, $\calF$ a set of functions 
from $I$ to $[0,\infty)$ that are locally bounded away from zero and $(S(t))_{t\in I}$ a family of linear maps from $X$ to $X$. 
If $(S(t))_{t\in I}$ is (sequentially) $\calF$-equicontinuous, then it is (sequentially) locally equicontinuous.
\end{proposition}

\begin{proof}
Let $(S(t))_{t\in I}$ be $\calF$-equicontinuous. Then there is $k\in\calF$ such that 
the family $(k(t)S(t))_{t\in I}$ is equicontinuous. 
Let $M\subseteq I$ compact and $p\in\semis$. Our statement directly follows from $k$ being locally bounded away from zero 
and the estimate
\[
 p(S(t)x)
\leq \frac{1}{\inf_{s\in M}k(s)} p(k(t)S(t)x)
\] 
for all $t\in M$ and $x\in X$.  
\end{proof}

Hence, (sequentially) exponentially equicontinuous and (sequentially) polynomially equicontinuous families 
$(S(t))_{t\in I}$ in the sense of \prettyref{defi:exp_pol_equicont} are also (sequentially) locally equicontinuous. 
However, we note that there are strongly continuous locally equicontinuous semigroups that are not exponentially 
equicontinuous, see e.g.~\cite[Remark 2.2 (iii)]{AlbaneseBonetRicker2013} for a case 
where $X$ is a Fr\'echet space. 
For the rest of this section, we now take a closer look at the case when $X$ is a normed space. 

\begin{definition}\label{defi:bound_lin_norm}
Let $(X,\norm{\cdot})$ be a normed space, $I$ a set, $\calF$ a set of functions from $I$ to $[0,\infty)$ 
and $(S(t))_{t\in I}$ a family in $\calL(X)$. Then $(S(t))_{t\in I}$ is called 
\begin{enumerate}[label=\upshape(\alph*), leftmargin=*]
\item \emph{bounded} if $\sup_{t\in I}\norm{S(t)}_{\calL(X)}<\infty$. 
\item \emph{locally bounded} if $I$ is a Hausdorff topological space and $(S(t))_{t\in M}$ is bounded for all compact sets $M\subseteq I$. 
\item $\calF$-\emph{bounded} if there is $k\in\calF$ such that $(k(t)S(t))_{t\in I}$ is bounded.   
\end{enumerate} 
\end{definition}

We note that these notions of (local, $\calF$-) boundedness (w.r.t.~$\norm{\cdot}_{\calL(X)})$ are equivalent 
to (local, $\calF$-) equicontinuity (w.r.t.~$\topo_{\norm{\cdot}}$). We use the terms \emph{exponentially bounded} 
and \emph{polynomially bounded} in the sense of \prettyref{defi:exp_pol_equicont} accordingly. 

\begin{lemma}
\label{lem:equi_cont_tight_norm_bound}
    Let $(X,\norm{\cdot})$ be a normed space, $\topo\subseteq\topo_{\norm{\cdot}}$ a Hausdorff locally convex topology on $X$ 
    such that $\topo$-bounded sets are also $\norm{\cdot}$-bounded, and $(S(t))_{t\in I}$ a family of linear maps from $X$ to $X$. 
    If $(S(t))_{t\in I}$ is sequentially $\topo$-equicontinuous,
    then $S(t)\in\calL(X)$ for all $t\in I$ and $\sup_{t\in I}\norm{S(t)}_{\calL(X)}<\infty$.
\end{lemma}

\begin{proof}
First, we note that the conditions $\topo\subseteq\topo_{\norm{\cdot}}$ and that $\topo$-bounded sets are also $\norm{\cdot}$-bounded 
already imply that a subset of $X$ is $\topo$-bounded if and only if it is $\norm{\cdot}$-bounded.
Let $(S(t))_{t\in I}$ be sequentially $\topo$-equicontinuous. 
Then $\{S(t)D\;|\;t\in I\}$ is $\topo$-bounded for all $\topo$-bounded sets 
$D\subseteq X$ by \cite[Proposition A.1 (iii)]{FedericoRosestolato2020}. By our first observation, this yields that 
$\{S(t)B_{\norm{\cdot}}\;|\;t\in I\}$ is $\norm{\cdot}$-bounded where $B_{\norm{\cdot}}$ is the $\norm{\cdot}$-closed unit ball. 
Hence, $S(t)\in\calL(X)$ for all $t\in I$ as $(X,\norm{\cdot})$ is bornological, and $\sup_{t\in I}\norm{S(t)}_{\calL(X)}<\infty$.
\end{proof}

\begin{lemma}
\label{lem:loc_bound_from_strong_cont}
    Let $(X,\norm{\cdot})$ be a Banach space, $\topo$ a Hausdorff locally convex topology on $X$ such that $\topo$-bounded 
    sets are also $\norm{\cdot}$-bounded, $I$ a Hausdorff topological space and $(S(t))_{t\in I}$ a family in $\calL(X)$.
    If $(S(t))_{t\in I}$ is $\topo$-strongly continuous, then it is locally bounded.
\end{lemma}

\begin{proof}
     Let $M\subseteq I$ be compact and $x\in X$. Since $S(\cdot)x$ is $\topo$-continuous and $M$ is compact, 
     $\{S(t)x\;|\; t\in M\} = S(M)x$ is $\topo$-compact and hence $\topo$-bounded; 
     thus, by assumption, also $\norm{\cdot}$-bounded. 
     Put differently, $(S(t))_{t\in M}$ is pointwise bounded in the Banach space $(X,\norm{\cdot})$ and 
     so also uniformly bounded by the uniform boundedness principle, 
     i.e.~$\sup_{t\in M} \norm{S(t)}_{\calL(X)}<\infty$.
\end{proof}

Even if $\topo=\topo_{\norm{\cdot}}$, then $\alpha$-times integrated $C$-semigroups and $C$-cosine families for $\alpha\geq 0$ 
are in general not exponentially bounded, see \cite[Example, p.~201]{DaviesPang1987}, 
\cite[p.~203--204]{KuoShaw2000}, \cite[p.~634--635]{KuoShaw1997a} and \cite[p.~145]{Kuo2006}. 
The same is true for evolution families, see \cite[Section VI.9.6]{EngelNagel2000} and \cite[p.~19]{Scirratt2016}. 
This is different for semigroups and cosine families. 

\begin{lemma}
\label{lem:growth_bound}
    Let $(X,\norm{\cdot})$ be a Banach space, $\topo$ a Hausdorff locally convex topology on $X$ such that $\topo$-bounded 
    sets are also $\norm{\cdot}$-bounded, and $(S(t))_{t\geq 0}$ a semigroup or a cosine family in $\calL(X)$. 
    If $(S(t))_{t\geq 0}$ is $\topo$-strongly continuous, then it is exponentially bounded.
\end{lemma}

\begin{proof}
     By the proof of \prettyref{lem:loc_bound_from_strong_cont}, we know that 
     $K\coloneq \sup_{t\in [0,2]} \norm{S(t)}_{\calL(X)}<\infty$.
     
     (i) Let $(S(t))_{t\geq 0}$ be a semigroup. Let $\omega\geq 0$ such that $\norm{S(1)}_{\calL(X)}\euler^{-\omega}\leq 1$. 
     We show that $\norm{S(t)}_{\calL(X)} \leq K\euler^{\omega t}$ for all $t\geq 0$ by induction. 
     The inequality is clearly true for $t\in [0,2]$. 
     Assume that $n\in\N$, $n\geq 2$, is such that the inequality is true for all $t\in [0,n]$. For $t\in (n-1,n]$ we thus observe 
     that
     \begin{align*}
            \norm{S(t+1)}_{\calL(X)} 
        & = \norm{S(t)S(1)}_{\calL(X)} \leq K\euler^{\omega t} \norm{S(1)}_{\calL(X)} 
          =  \norm{S(1)}_{\calL(X)}\euler^{-\omega} K\euler^{\omega (t+1)} \\
        & \leq K\euler^{\omega(t+1)}.
    \end{align*}
     
     (ii) Let $(S(t))_{t\geq 0}$ be a cosine family. The proof is quite similar to case (i). 
     Let $\omega\geq 0$ such that $2\norm{S(1)}_{\calL(X)}\euler^{-\omega} + \euler^{-2\omega} \leq 1$. 
     We show that $\norm{S(t)}_{\calL(X)} \leq K\euler^{\omega t}$ for all $t\geq 0$ by induction. 
     Again, the inequality is clearly true for $t\in [0,2]$. 
     Assume that $n\in\N$, $n\geq 2$, is such that the inequality is true for all $t\in [0,n]$. For $t\in (n-1,n]$ we thus observe 
     that
     \begin{align*}
            \norm{S(t+1)}_{\calL(X)} 
        & = \norm{2S(t)S(1)-S(t-1)}_{\calL(X)} \leq 2K\euler^{\omega t} \norm{S(1)}_{\calL(X)} + K\euler^{\omega(t-1)} \\
        & = (2\norm{S(1)}_{\calL(X)}\euler^{-\omega} + \euler^{-2\omega}) K\euler^{\omega(t+1)} 
          \leq K\euler^{\omega(t+1)}.\qedhere
    \end{align*}
\end{proof}

For $\topo=\topo_{\norm{\cdot}}$ \prettyref{lem:growth_bound} is already known, 
see \cite[Chap.~I, 5.5 Proposition]{EngelNagel2000} and \cite[Lemma 3.14.3 a)]{ArendtBHN2011}, and our proof is an adaptation 
of those proofs.

Let $I$ be Hausdorff topological space and $\calF$ a set of functions from $I$ to $[0,\infty)$. 
We say that $\calF$ has the \emph{rational vanishing at infinity property}, in short \emph{$\calF$ has rV$_{\infty}$}, if 
\[
\forall\;k_{1}\in\calF\;\exists\;k_{2}\in\calF\;\forall\;\varepsilon >0\;\exists\;M\subseteq I\text{ compact }
\forall\;t\in I\setminus M:\; k_{2}(t)\leq\varepsilon k_{1}(t).
\] 
If $k_{1}(t)\neq 0$ for all $t\in I$, then this means that $k_{2}/k_{1}$ vanishes at infinity. 
In the context of families of weight functions for weighted spaces of continuous functions, a condition like this 
with swapped roles of $k_{1}$ and $k_{2}$ appears in \cite[1.3 Bemerkung]{Bierstedt1973} and \cite[p.~34]{Kruse2023}.

\begin{theorem}\label{thm:equicont_from_loc_eqicont_norm}
Let $(X,\norm{\cdot})$ be a normed space, $\topo\subseteq\topo_{\norm{\cdot}}$ a Hausdorff locally convex topology on $X$ 
such that $(X,\topo)$ is a strong Mackey space, $I$ a locally compact Hausdorff topological space, 
$\calF$ a set of continuous functions from $I$ to $[0,\infty)$ and $(S(t))_{t\in I}$ a $\topo$-strongly continuous 
family in $\calL(X,\sigma(X,(X,\topo)'))$. 
If $\calF$ has rV$_{\infty}$ and $(S(t))_{t\in I}$ is $\calF$-bounded, then $(S(t))_{t\in I}$ is $\calF$-$\topo$-equicontinuous.
\end{theorem}

\begin{proof}
Due to \prettyref{prop:str_M_str_cont_loc_equicont} and \prettyref{rem:conditions_str_M_str_cont_loc_equicont} \ref{it:strong_Mackey}, we only need to consider the case that $I$ is noncompact.
We note that there is $k_{1}\in\calF$ such that $K_{1}\coloneq\sup_{t\in I}\norm{k_{1}(t)S(t)}_{\calL(X)}<\infty$ 
since $(S(t))_{t\in I}$ is $\calF$-bounded. Let $p\in\semis_{\topo}$. 
It follows from $\topo\subseteq\topo_{\norm{\cdot}}$ that there is $K_{2}> 0$ such that $p(w)\leq K_{2}\norm{w}$ for all $w\in X$.
Let $\varepsilon>0$ and $x\in X$. As $\calF$ has rV$_{\infty}$, there are $k\in\calF$, independent of $\varepsilon$, 
and a compact set $M\subseteq I$ such that $k(t)\leq\frac{\varepsilon}{(K_{1}\norm{x}+1)K_{2}} k_{1}(t)$ for all $t\in I\setminus M$.
This implies that 
\begin{align*}
     p(k(t)S(t)x)
&\leq K_{2}\norm{k(t)S(t)}_{\calL(X)}\norm{x}
 \leq K_{2}\frac{\varepsilon}{(K_{1}\norm{x}+1)K_{2}}\norm{k_{1}(t)S(t)}_{\calL(X)}\norm{x}\\
&=    K_{2}\frac{\varepsilon}{(K_{1}\norm{x}+1)K_{2}}K_{1}\norm{x}
 <    \varepsilon 
\end{align*}
for all $t\in I\setminus M$. Hence, $\topo$-$\lim_{t\to\infty}k(t)S(t)x=0$. Since $x$ is arbitrary, we obtain that $(S(t))_{t\in I}$ is 
$\calF$-equicontinuous by \prettyref{prop:equicont_from_loc_eqicont} with $Y\coloneq (X,\topo)'$ and 
\prettyref{rem:conditions_str_M_str_cont_loc_equicont} \ref{it:strong_Mackey} .
\end{proof}

\prettyref{thm:equicont_from_loc_eqicont_norm} has the restriction that $(X,\topo)$ should be a Mackey space. 
Next, we want to obtain sufficient conditions for $\calF$-$\topo$-equicontinuity without this restriction. 
We start with the following property of $\calF$. Let $I$ be a set and $\calF$ a set of functions from $I$ to $[0,\infty)$. 
We say that $\calF$ has the \emph{rational summability property}, in short \emph{$\calF$ has rsp}, 
if for every $k_{1}\in\calF$ there are a sequence $(M_{n})_{n\in\N}$ of subsets of $I$ such that $\bigcup_{n\in\N}M_{n}=I$, 
and $k_{2}\in\calF$ and a non-negative sequence $(b_{n})_{n\in\N_{0}}$ with $0<\sum_{n=0}^{\infty}b_{n}<\infty$ 
such that $k_{2}(t)\leq b_{0}k_{1}(t)$ for all $t\in M_{1}$ and for every $n\in\N$ it holds that 
$k_{2}(t)\leq b_{n}k_{1}(t)$ for all $t\in M_{n+1}\setminus M_{n}$. 

\begin{remark}\label{rem:rational_summable_weights}
The sets $\calF\coloneq\calF_{\exp}$ and $\calF_{\operatorname{poly}}$ from \prettyref{defi:exp_pol_equicont} 
where $I=[0,\infty)$ have rsp. Indeed, let $M_{n}\coloneq [0,n]$ for $n\in\N$.

In the case $\calF=\calF_{\exp}$ let $\alpha\in\R$, $\lambda<\alpha$ and $n\in\N$. 
Then we have $\euler^{(\lambda-\alpha) t}\leq \euler^{(\lambda-\alpha)n}$ 
for all $t\in (n,n+1]=M_{n+1}\setminus M_{n}$, and $\euler^{(\lambda-\alpha) t}\leq 1$ for all $t\in [0,1]= M_{1}$. 
The choice $b_{0}\coloneq 1$ and $b_{n}\coloneq \euler^{(\lambda-\alpha)n}$ yields that $\calF_{\exp}$ has rsp since 
$\sum_{n=0}^{\infty}b_n=\frac{1}{1-\euler^{\lambda-\alpha}}$.

In the case $\calF=\calF_{\operatorname{poly}}$ let $\alpha>0$ and $n\in\N$. Then we have 
$\frac{1+t^{\alpha}}{1+t^{\alpha+2}}\leq \frac{2t^{\alpha}}{t^{\alpha+2}}=\frac{2}{t^{2}}\leq\frac{2}{n^2}$ 
for all $t\in (n,n+1]$, and $\frac{1+t^{\alpha}}{1+t^{\alpha+2}}\leq 2$ for all $t\in [0,1]$. 
The choice $b_{0}\coloneq 2$ and $b_{n}\coloneq \frac{2}{n^2}$ yields that $\calF_{\operatorname{poly}}$ has rsp since
$\sum_{n=0}^{\infty}b_n=2+\frac{\pi^2}{3}$. 
\end{remark}

In our next result we will combine the rational summability property of $\calF$ with the following property. 
Let $(X,\topo)$ be a Hausdorff locally convex space. A system of seminorms $\semis_{\topo}$ generating the topology $\topo$ 
is called \emph{countably quasi-convex} if for every sequence 
$(p_{n})_{n\in\N}$ of seminorms in $\semis_{\topo}$ and every non-negative sequence $(a_{n})_{n\in\N}$ such that 
$\sum_{n=1}^{\infty}a_{n}=1$ there is $p\in\semis_{\topo}$ with $\sup_{n\in\N}a_{n}p_{n}\leq p$. 

\begin{proposition}\label{prop:functional_equicont_from_loc}
Let $(X,\topo)$ be a Hausdorff locally convex space, $\calF$ a set of functions from a set $I$ to $[0,\infty)$ and 
$(S(t))_{t\in I}$ a family of linear maps from $X$ to $X$. 
If $\calF$ has rsp and there is a system of seminorms $\semis_{\topo}$ generating $\topo$ such that there are $k_{1}\in\calF$ and 
$K\geq 0$ such that for all $p\in\semis_{\topo}$ there is a sequence $(q_{n})_{n\in\N}$ 
in $\semis_{\topo}$ such that for all $n\in\N$
\[
\sup_{t\in M_{n}}p(k_{1}(t)S(t)x)\leq Kq_{n}(x) \quad (x\in X)
\]
with $(M_{n})_{n\in\N}$ from rsp, then there is a non-negative sequence $(a_{n})_{n\in\N}$ such that 
$\sum_{n=1}^{\infty}a_{n}=1$ and 
\[
\sup_{t\in I}p(k_{2}(t)S(t)x)\leq K \sum_{k=0}^{\infty}b_{k} \sup_{n\in\N}a_{n}q_{n}(x) \quad (x\in X)
\]
with $k_2 \in\calF$ and $(b_{n})_{n\in\N_{0}}$ from rsp. 
In particular, if there is $q\in\semis_{\topo}$ such that $\sup_{n\in\N}a_{n}q_{n}\leq q$, then 
$(S(t))_{t\in I}$ is $\calF$-$\topo$-equiconti\-nuous. For instance, this is the case if $\semis_{\topo}$ is 
countably quasi-convex.
\end{proposition}

\begin{proof}
Since $\calF$ has rsp, there are a sequence $(M_{n})_{n\in\N}$ of subsets of $I$ such that $\bigcup_{n\in\N}M_{n}=I$, 
and $k_{2}\in\calF$ and a non-negative sequence $(b_{n})_{n\in\N_{0}}$ with $0<\sum_{n=0}^{\infty}b_{n}<\infty$ 
such that $k_{2}(t)\leq b_{0}k_{1}(t)$ for all $t\in M_{1}$ and for every $n\in\N$ it holds that $k_{2}(t)\leq b_{n}k_{1}(t)$ 
for all $t\in M_{n+1}\setminus M_{n}$. Let $p\in\semis_{\topo}$. Then there is a sequence $(q_{n})_{n\in\N}$ 
in $\semis_{\topo}$ such that 
\[
\sup_{t\in M_{n}}p(k_{1}(t)S(t)x)\leq Kq_{n}(x) \quad (x\in X).
\]
For all $n\in\N$ and $x\in X$ we have that
\[
     \sup_{t\in M_{n+1}\setminus M_{n}}p(k_{2}(t)S(t)x)
\leq \sup_{t\in M_{n+1}\setminus M_{n}}b_{n}p(k_{1}(t)S(t)x)
\leq Kb_{n}q_{n+1}(x)
\]
and 
\[
     \sup_{t\in M_{1}}p(k_{2}(t)S(t)x)
\leq \sup_{t\in M_{1}}b_{0}p(k_{1}(t)S(t)x)
\leq Kb_{0}q_{1}(x).
\]
We set $K_{1}\coloneq \sum_{n=0}^{\infty}b_{n}$ and $a_{n}\coloneq b_{n-1}/K_{1}$ for all $n\in\N$. 
Then $a_{n}\geq 0$ for all $n\in\N$ and $\sum_{n=1}^{\infty}a_{n}=1$. 
Assuming w.l.o.g.~$M_{n+1}\setminus M_{n}\neq \varnothing$ for all $n\in\N$, we obtain that
\begin{align*}
      \sup_{t\in I}p(k_{2}(t)S(t)x)
&\leq \max\Bigl\{\sup_{t\in M_{1}}p(k_{2}(t)S(t)x),\,\sup_{n\in\N}\sup_{t\in M_{n+1}\setminus M_{n}}p(k_{2}(t)S(t)x)\Bigr\}\\
&\leq K\sup_{n\in\N_0}b_{n}q_{n+1}(x)
 \leq K K_{1}\sup_{n\in\N}a_{n}q_{n}(x)
\end{align*}
for all $x\in X$, implying our statement.
\end{proof}

The condition of existence of a system of seminorms $\semis_{\topo}$ having the properties described 
in \prettyref{prop:functional_equicont_from_loc} (apart from countable quasi-convexity) 
is also necessary for a semigroup or (commutative) cosine family to be exponentially $\topo$-equicontinuous. 

\begin{remark}
    Let $(X,\topo)$ be a Hausdorff locally convex space and $(S(t))_{t\geq 0}$ an exponentially $\topo$-equicontinuous 
    semigroup or cosine family on $X$. If $(S(t))_{t\geq 0}$ is a cosine family, then we assume in addition that $S(t)S(s)=S(s)S(t)$ for all $t,s\geq 0$, see \prettyref{rem:commutative_cosine}. 
    By the exponential $\topo$-equicontinuity, there is $\alpha\geq 0$ such that $(\euler^{-\alpha t}S(t))_{t\geq 0}$ 
    is $\topo$-equicontinuous. 
    Let $\semis$ be a system of seminorms that generates the topology $\topo$ on $X$ and for $p\in\semis$ define 
    \[
    \widetilde{p}(x)\coloneq \sup_{t\geq 0}p(\euler^{-\alpha t}S(t)x) \quad (x\in X).
    \]
    Then the system of seminorms $\widetilde{\semis}\coloneq\{\widetilde{p}\;|\; p\in\semis\}$ generates the topology $\topo$ on $X$ 
    and for every $p\in\semis$ it holds that
    \[
    \widetilde{p}(\euler^{-\alpha t}S(t)x)\leq \widetilde{p}(x) \quad (t\geq 0, x\in X).
    \]
    Indeed, for semigroups this follows from \cite[Remark 2.2 (i)]{AlbaneseBonetRicker2013}.\footnote{In \cite[Remark 2.2 (i)]{AlbaneseBonetRicker2013} it is additionally assumed that $(S(t))_{t\geq 0}$ is $\topo$-strongly continuous at $t=0$ but its proof shows that this is not needed.} 
    Let us turn to cosine families. We note that by the $\topo$-equicontinuity of 
    $(\euler^{-\alpha t}S(t))_{t\geq 0}$, for every $p\in\semis$ there are $q\in\semis$ and $K_1\geq 0$ such that
    \[
    p(x)\leq \widetilde{p}(x)\leq K_1 q(x) \quad (x\in X).
    \]
    Further, setting $S(t)\coloneq S(-t)$ for $t<0$, it follows from $S(t)S(s)=S(s)S(t)$ for all $t,s\geq 0$ 
    that $2S(t)S(s)=S(t+s)-S(t-s)$ for all $t,s\geq 0$. Hence, we have for all $p\in\semis$ that 
    \[
    p(\euler^{-\alpha (t+s)}S(t-s)x)
    =
    \begin{cases}
        \euler^{-2\alpha s}p(\euler^{-\alpha (t-s)}S(t-s)x) &,\;t\geq s,\\
        \euler^{-2\alpha t}p(\euler^{-\alpha (s-t)}S(s-t)x) &,\;t< s,\\
    \end{cases}
    \]
    for all $t,s\geq 0$ and $x\in X$, and
    \begin{align*}
      \widetilde{p}(\euler^{-\alpha s}S(s)x)
    &=\sup_{t\geq 0}p(\euler^{-\alpha (t+s)}S(t)S(s)x)
     =\frac{1}{2}\sup_{t\geq 0}p\bigl(\euler^{-\alpha (t+s)}(S(t+s)x-S(t-s)x)\bigr)\\
    &\leq \frac{1}{2}\Bigl(\sup_{t\geq 0}p(\euler^{-\alpha (t+s)}S(t+s)x)
                           +\sup_{t\geq 0}p(\euler^{-\alpha (t+s)}S(t-s)x)\Bigr)
    \leq\widetilde{p}(x) 
    \end{align*}  
     for all $s\geq 0$ and $x\in X$ since $\alpha\geq 0$.

     Now, let $p\in\semis$. We set $k_{1}(t)\coloneq \euler^{-\alpha t}$ for $t\geq 0$, $K\coloneq 1$, $M_n\coloneq [0,n]$ and 
     $q_n\coloneq \widetilde{p}\in\widetilde{\semis}$ for $n\in\N$ and note that 
     \[
     \sup_{t\in M_{n}}\widetilde{p}(k_{1}(t)S(t)x)\leq Kq_{n}(x) \quad (x\in X)
     \]
     for all $n\in\N$, and $\sup_{n\in\N}a_n q_n\leq \widetilde{p}\eqcolon q$ since $0\leq a_n\leq 1$ for all $n\in\N$.
\end{remark}

We will restrict our attention now to the setting where there is also a norm on $X$. 
Let $(X,\norm{\cdot})$ be a normed space and $\topo\subseteq\topo_{\norm{\cdot}}$ a Hausdorff locally convex topology on $X$. 
Then $\norm{\cdot}$-bounded subsets of $X$ are also $\topo$-bounded, and for all $\topo$-continuous seminorms $p$ 
on $X$ there exists $K\geq 0$ such that $p(x)\leq K\norm{x}$ for all $x\in X$. 
We write 
\[
\calN_{\topo}\coloneq\{p\colon X\to [0,\infty)\;|\; p\text{ is a $\topo$-continuous seminorm, } p\leq \norm{\cdot}\}
\]
and note that $\calN_{\topo}$ generates the topology $\topo$.

According to \cite[p.~163]{Kraaij2016}, $\calN_{\topo}$ is called \emph{countably convex} if for every sequence 
$(p_{n})_{n\in\N}$ of seminorms in $\calN_{\topo}$ and every non-negative sequence $(a_{n})_{n\in\N}$ such that 
$\sum_{n=1}^{\infty}a_{n}=1$ it holds that $p\coloneq \sum_{n=1}^{\infty}a_{n}p_{n}\in\calN_{\topo}$. 
In particular, $\calN_{\topo}$ is countably quasi-convex if it is countably convex.

Next, we turn to some sufficient conditions for $\calN_{\topo}$ to be countably convex. 
Let $(X,\topo)$ be a Hausdorff locally convex space and set $X'\coloneq (X,\topo)'$.
We recall that $(X,\topo)$ is \emph{$c_{0}$-barrelled} if 
any $\sigma(X',X)$-null sequence in $X'$ is $\topo$-equicontinuous, see \cite[p.~249]{Jarchow1981}. 
$c_{0}$-barrelled spaces are also called \emph{sequentially barrelled}, see \cite[Definition, p.~353]{Webb1968}. 
The space $(X,\topo)$ is called \emph{transseparable} if for every $\topo$-neighbourhood $U$ of zero in $X$, there exists a 
countable set $A\subseteq X$ such that $X=A+U$, see \cite[Definition 2.5.1]{CarrerasBonet1987}. 
Clearly, $(X,\topo)$ is transseparable if it is separable. 

\begin{remark}\label{rem:count_conv}
Let $(X,\norm{\cdot})$ be a normed space, $\topo\subseteq\topo_{\norm{\cdot}}$ a Hausdorff locally convex topology on $X$ 
and set $X'\coloneq (X,\topo)'$. 
Then $\calN_{\topo}$ is countably convex, thus countably quasi-convex, in each of the following cases.
\begin{enumerate}[label=\upshape(\roman*), leftmargin=*, widest=iii]
\item\label{it:count_conv_1} $(X,\topo)$ is C-sequential.
\item\label{it:count_conv_2} $(X,\topo)$ is a $c_0$-barrelled Mackey space.
\item\label{it:count_conv_3} $(X,\topo)$ is transseparable and $(X',\sigma(X',X))$ sequentially complete.
\end{enumerate}
This follows from \cite[Proposition 4.9]{Kraaij2016}\footnote{In \cite[Proposition 4.9]{Kraaij2016} it is additionally assumed that 
$(X,\topo)$ is sequentially complete and that $\topo$-bounded sets are also $\norm{\cdot}$-bounded but this is not used in its proof.} in combination with \cite[Theorem 7.4]{Wilansky1981} in case \ref{it:count_conv_1}, and 
\cite[Theorem 12.1.4]{Jarchow1981} in case \ref{it:count_conv_2}. 
In particular, if $(X,\topo)$ is barrelled, then \ref{it:count_conv_2} is fulfilled by \cite[Theorem 12.1.4]{Jarchow1981}
and \prettyref{rem:conditions_str_M_str_cont_loc_equicont} \ref{it:sufficient_strong_Mackey_1}.
If $(X,\topo)$ is a Mackey--Mazur space, then \ref{it:count_conv_1} is fulfilled by \cite[Corollary 7.6]{Wilansky1981}. 
If $(X,\topo)$ is a sequentially complete Mackey--Mazur space, then \ref{it:count_conv_2} is also fulfilled 
by \cite[Proposition 4.3]{Webb1968}. 
Concerning \ref{it:count_conv_3}, we note that $(X',\sigma(X',X))$ is sequentially complete by \cite[Proposition 4.4]{Webb1968} 
if $(X,\topo)$ is a $c_0$-barrelled Mazur space.
\end{remark}

At least for locally $\topo$-equicontinuous semigroups and cosine families we can show by 
\prettyref{prop:functional_equicont_from_loc} that they are exponentially $\topo$-equicontinuous 
(under the usual assumptions on $\norm{\cdot}$ and $\topo$, and of countable quasi-convexity of $\calN_{\topo}$). 
In order to do this, we need the next lemma, which is \cite[Corollary 4.7]{Kraaij2016} in the case of semigroups and whose 
proof for cosine families is quite similar to the one of \cite[Corollary 4.7]{Kraaij2016} and \prettyref{lem:growth_bound}. 

\begin{lemma}
\label{lem:growth_bound_lcx}
Let $(X,\norm{\cdot})$ be a normed space, $\topo\subseteq\topo_{\norm{\cdot}}$ a Hausdorff locally convex topology on $X$ 
such that $\topo$-bounded sets are also $\norm{\cdot}$-bounded and $(S(t))_{t\geq 0}$ a semigroup or cosine family on $X$.
If $(S(t))_{t\geq 0}$ is locally $\topo$-equicontinuous, then there exist $K\geq 1$ and $\omega\in\R$ such that 
for all $p\in\calN_{\topo}$ there exists a sequence $(q_{n})_{n\in\N}$ in $\calN_{\topo}$ such that
\[
\sup_{t\in[0,n]} p(\euler^{-\omega t}S(t)x) \leq K q_{n}(x) \quad(x\in X).
\]
\end{lemma}

\begin{proof}
First, we note that \cite[Condition B, p.~163]{Kraaij2016} is fulfilled. 
In the case that $(S(t))_{t\geq 0}$ is a semigroup the proof is already given in \cite[Corollary 4.7]{Kraaij2016}. 
So, let $(S(t))_{t\geq 0}$ be a cosine family. 

    (i) 
    By \cite[Lemma 4.5]{Kraaij2016}, there exists $K\geq 1$ such that for every $p\in\calN_{\topo}$ there exists 
    $q\in\calN_{\topo}$ such that 
    \begin{equation}
    \label{eq:equibounded}
        \sup_{t\in [0,2]} p(S(t)x) \leq Kq(x) \quad(x\in X).
    \end{equation}
    Without loss of generality, we may assume that $q\geq p$.
    Let $\omega\geq 0$ such that $2K \euler^{-\omega} + \euler^{-2\omega}\leq 1$.
    
    (ii) 
    Let $p\in\calN_{\topo}$. Then there is $q_{1}\coloneq q\in\calN_{\topo}$ with $q_{1}\geq p$ such that \eqref{eq:equibounded} holds.
    We now construct an increasing sequence of seminorms $(q_n)_{n\in\N}$ in $\calN_{\topo}$ fulfilling 
    $\sup_{t\in [0,2]} q_n(S(t)x) \leq Kq_{n+1}(x)$ for all $x\in X$ and $n\in\N$.
    By (i), there exists $q_2\in\calN_{\topo}$ with $q_2\geq q_{1}$ such that 
    $\sup_{t\in [0,2]} q_{1}(S(t)x) \leq Kq_2(x)$ for all $x\in X$.
    Inductively, by (i), for $n\in\N$, there exists $q_{n+1}\in\calN_{\topo}$ with $q_{n+1}\geq q_n$ 
    such that $\sup_{t\in [0,2]} q_n(S(t)x) \leq Kq_{n+1}(x)$ for all $x\in X$.

    (iii)
    We now show inductively that for all $n\in\N$, we have
    \[
    \sup_{t\in [0,n]}  p(\euler^{-\omega t}S(t)x) \leq K q_n(x) \quad (x\in X).
    \]
    For $n=1$, we obtain
    \[
         \sup_{t\in [0,1]} p(\euler^{-\omega t}S(t)x) 
    \leq \sup_{t\in [0,2]} p(S(t)x) 
    \leq K q_1(x) \quad (x\in X)
    \]
    and for $n=2$, we have
    \[
    \sup_{t\in [0,2]}  p(\euler^{-\omega t}S(t)x) \leq \sup_{t\in [0,2]} q_{1}(S(t)x) \leq K q_2(x) \quad (x\in X)
    \]
    by (ii).
    Assume that $n\in\N$, $n\geq 2$, is such that we have
    \[
    \sup_{t\in [0,n]} \euler^{-\omega t} p(S(t)x) \leq K q_n(x) \quad (x\in X).
    \]
    Then, for $t\in (n-1,n]$ and $x\in X$, we obtain by (ii), $q_{n+1}\geq q_n$ and by the choice of $\omega$ in (i) that
    \begin{align*}
    p(S(t+1)x) & = p((2S(t)S(1)-S(t-1))x) \leq 2p(S(t)S(1)x) + p(S(t-1)x) \\
        & \leq 2K \euler^{\omega t} q_n(S(1)x) + K\euler^{\omega (t-1)} q_n(x) \\
        & \leq 2K \euler^{\omega t} Kq_{n+1}(x) + K\euler^{\omega (t-1)}q_{n+1}(x) \\
        & = (2K\euler^{-\omega} + \euler^{-2\omega})K\euler^{\omega(t+1)}q_{n+1}(x) 
        \leq K\euler^{\omega(t+1)}q_{n+1}(x)
    \end{align*}
    and so
    \begin{align*}
          \sup_{t\in [0,n+1]}p( \euler^{-\omega t}S(t)x) 
    &\leq \max\Bigl\{\sup_{t\in [0,n]} p(\euler^{-\omega t}S(t)x),\,\sup_{t\in (n-1,n]}p(\euler^{-\omega (t+1)}S(t+1)x)\Bigr\}\\
    &\leq K q_{n+1}(x)
    \end{align*}
    for all $x\in X$.
\end{proof}

\begin{remark}\label{rem:type_sg_cosine}
    Let $(X,\norm{\cdot})$ and $\topo\subseteq\topo_{\norm{\cdot}}$ be as in \prettyref{lem:growth_bound_lcx}. 
    If we already know that $(S(t))_{t\geq 0}$ is a semigroup or cosine family on $X$ such that there exist $K\geq 1$ and $\omega\in\R$ 
    such that for all $t_0\geq 0$ and $p\in\calN_{\topo}$ there is $\widetilde{q}_{t_0}\in\calN_{\topo}$ such that
    \[
    \sup_{t\in[0,t_0]} p(\euler^{-\omega t}S(t)x) \leq K \widetilde{q}_{t_0}(x) \quad(x\in X),
    \]
    then we can choose $q_n\coloneq\max\{\widetilde{q}_{i}\;|\;1\leq i\leq n\}\in\calN_{\topo}$ for $n\in\N$.
\end{remark}

\begin{theorem}\label{thm:loc_equicont_exp_sgcs_countconv}
Let $(X,\norm{\cdot})$ be a normed space, $\topo\subseteq\topo_{\norm{\cdot}}$ a Hausdorff locally convex topology on $X$ 
such that $\topo$-bounded sets are also $\norm{\cdot}$-bounded and $(S(t))_{t\geq 0}$ a semigroup or cosine family on $X$.
If $\calN_{\topo}$ is countably quasi-convex and $(S(t))_{t\geq 0}$ locally $\topo$-equicontinuous, 
then $(S(t))_{t\geq 0}$ is exponentially $\topo$-equicontinuous.
\end{theorem}

\begin{proof}
Our statement follows from \prettyref{rem:rational_summable_weights}, \prettyref{prop:functional_equicont_from_loc} 
and \prettyref{lem:growth_bound_lcx} with $I\coloneq [0,\infty)$, $M_{n}\coloneq [0,n]$ for $n\in\N$, 
$\calF\coloneq\calF_{\exp}$ and $k_{1}(t)\coloneq \euler^{-\omega t}$ for $t\geq 0$.
\end{proof}

\begin{corollary}\label{cor:type}
    Let $(X,\norm{\cdot})$ be a normed space, $\topo\subseteq\topo_{\norm{\cdot}}$ a Hausdorff locally convex topology on $X$ 
    such that $\topo$-bounded sets are also $\norm{\cdot}$-bounded and $(S(t))_{t\geq 0}$ a semigroup or cosine family on $X$.
    Let $\calN_{\topo}$ be countably quasi-convex, $M\geq 1$ and $\omega\in\R$. Then the following assertions are equivalent.
    \begin{enumerate}[label=\upshape(\alph*), leftmargin=*]
    \item\label{it:type_1} $(S(t))_{t\geq 0}$ is locally $\topo$-equicontinuous and 
    \[
    \norm{S(t)}_{\calL(X)}\leq M\euler^{\omega t}\quad(t\geq 0).\footnote{We note that $S(t)\in\calL(X)$ for all 
    $t\geq 0$ by \prettyref{lem:equi_cont_tight_norm_bound}.}
    \]
    \item\label{it:type_2} For all $t_0\geq 0$ and $p\in\calN_{\topo}$ there exists $q\in\calN_{\topo}$ such that 
    \[
    p(\euler^{-\omega t}S(t)x)\leq Mq(x) \quad(t\in [0,t_0],\, x\in X).
    \]
    \item\label{it:type_3} For all $\alpha>\omega$ and $p\in\calN_{\topo}$ there exists $q\in\calN_{\topo}$ such that 
    \[
    p(\euler^{-\alpha t}S(t)x)\leq Mq(x) \quad(t\geq 0,\, x\in X).
    \]
\end{enumerate}
Furthermore, if the semigroup or cosine family $(S(t))_{t\geq 0}$ is locally $\topo$-equicontinuous, then there exist $M\geq 1$ and $\omega\in\R$ such 
that these assertions hold. 
\end{corollary}

\begin{proof}
    ``\ref{it:type_1}$\Rightarrow$\ref{it:type_2}'' Since $(S(t))_{t\geq 0}$ is locally $\topo$-equicontinuous by \ref{it:type_1} 
    and $[t\mapsto \euler^{-\omega t}]$ locally bounded on $[0,\infty)$, 
    we get that $(\euler^{-\omega t}S(t))_{t\geq 0}$ is locally $\topo$-equicontinuous. 
    Further, for all $t_0\geq 0$ we have $\norm{\euler^{-\omega t}S(t)}_{\calL(X)}\leq M$ for all $t\in [0,t_0]$. 
    Thus, we obtain that for all $t_0\geq 0$ and $p\in\calN_{\topo}$ there exists $q\in\calN_{\topo}$ such that 
    \[
    p(\euler^{-\omega t}S(t)x)\leq Mq(x) \quad(t\in [0,t_0],\, x\in X)
    \]
    by \cite[Lemma 4.5]{Kraaij2016}.

    ``\ref{it:type_2}$\Rightarrow$\ref{it:type_3}'' 
    Let $\alpha>\omega$. It follows from \ref{it:type_2}, \prettyref{rem:rational_summable_weights}, \prettyref{prop:functional_equicont_from_loc} 
    and \prettyref{rem:type_sg_cosine} with $I\coloneq [0,\infty)$, $M_{n}\coloneq [0,n]$ for $n\in\N$, 
    $\calF\coloneq\calF_{\exp}$, $K\coloneq M$, $k_{1}(t)\coloneq\euler^{-\omega t}$ and $k_{2}(t)\coloneq\euler^{-\alpha t}$ 
    for $t\geq 0$ that for all $p\in\calN_{\topo}$ there is $q^{\alpha}\in\calN_{\topo}$ such that
    \[
    \sup_{t\geq 0}p(\euler^{-\alpha t}S(t)x)\leq \frac{M}{1-\euler^{\omega-\alpha}}q^{\alpha}(x)\quad (x\in X).
    \]
    Fix $\alpha>\omega$ and $\varepsilon>0$ with $\alpha-\varepsilon>\omega$. 
    Let $t_{\alpha}>0$ such that $\euler^{-\varepsilon t_{\alpha}}\leq 1-\euler^{\omega-(\alpha-\varepsilon)}$.
    Then it follows that
    \[
    \sup_{t\geq t_\alpha}p(\euler^{-\alpha t}S(t)x)
    \leq \euler^{-\varepsilon t_\alpha}\sup_{t\geq t_\alpha}p(\euler^{-(\alpha-\varepsilon) t}S(t)x)
    \leq \euler^{-\varepsilon t_\alpha}\frac{M}{1-\euler^{\omega-(\alpha-\varepsilon)}}q^{\alpha-\varepsilon}(x)
    \leq Mq^{\alpha-\varepsilon}(x) \quad (x\in X).
    \]
    Moreover, by \ref{it:type_2}, there exists $q_{t_\alpha}\in\calN_{\topo}$ such that 
    \[
    p(\euler^{-\alpha t}S(t)x)\leq p(\euler^{-\omega t}S(t)x)\leq Mq_{t_\alpha}(x) \quad(t\in [0,t_\alpha],\, x\in X).
    \]
    Hence, we have 
    \begin{align*}
    \sup_{t\geq 0}p(\euler^{-\alpha t}S(t)x)
    & \leq \max\bigl\{\sup_{t\in [0,t_{\alpha}]}p(\euler^{-\alpha t}S(t)x),\, \sup_{t\geq t_{\alpha}}p(\euler^{-\alpha t}S(t)x)\bigr\}
    \leq \max\{Mq_{t_\alpha}(x),\, M q^{\alpha-\varepsilon}(x) \}\\
    & \leq M \widetilde{q}_{\alpha}(x) \quad (x\in X)
    \end{align*}
    where $\widetilde{q}_{\alpha}\coloneq\max\{q_{t_\alpha},\,q^{\alpha-\varepsilon}\}\in\calN_{\topo}$.

    ``\ref{it:type_3}$\Rightarrow$\ref{it:type_1}'' 
    Since $q\in\calN_{\topo}$, so $q(x)\leq \norm{x}$ for all $x\in X$, the assumption \ref{it:type_3} implies that 
    \[
    p(\euler^{-\alpha t}S(t)x)\leq M\norm{x} \quad(t\geq 0,\, x\in X).
    \]
    Next, taking the supremum over all $p\in\calN_{\topo}$ on the left-hand side, we obtain that 
    \[
    \norm{\euler^{-\alpha t}S(t)x}=\sup_{p\in\calN_{\topo}}p(\euler^{-\alpha t}S(t)x)\leq M\norm{x} \quad(t\geq 0,\, x\in X)
    \]
    by \prettyref{rem:Saks_bounded_sets} \ref{it:Saks_bounded_sets_1}. 
    Now, letting $\alpha\to\omega\rlim$, we get that 
    \[
    \norm{\euler^{-\omega t}S(t)x}\leq M\norm{x} \quad(t\geq 0,\, x\in X).
    \]

    The addendum about the existence of $M\geq 1$ and $\omega\in\R$ in the case of a locally $\topo$-equicontinuous semigroup or cosine family $(S(t))_{t\geq 0}$ 
    is a consequence of \prettyref{lem:growth_bound_lcx} and part \ref{it:type_2}.
\end{proof}

We note that $\omega\geq 0$ by \prettyref{prop:cosine_type} \ref{it:cosine_type_3} and \prettyref{cor:type} \ref{it:type_3} 
if $(S(t))_{t\geq 0}$ is a cosine family and $X\neq \{0\}$. 
In the context of bi-continuous semigroups, the tuple $(M,\omega)$ in \prettyref{cor:type} \ref{it:type_1} is sometimes called 
the \emph{type} of the semigroup, see \cite[p.~334]{BuddeFarkas2019}. 
In the terminology for semigroups considered in \cite[Definition 4.8]{Kraaij2016}, 
the tuple $(M,\omega)$ in \prettyref{cor:type} \ref{it:type_2} is called the \emph{type} of a semigroup and 
\prettyref{cor:type} \ref{it:type_3} means that a semigroup is of type $(M,\alpha)^{\ast}$ for all $\alpha>\omega$. 
Assuming in addition to the assumptions of \prettyref{cor:type} 
that $(X,\topo)$ is sequentially complete and $\calN_{\topo}$ countably convex, see \cite[Condition C, p.~165--166]{Kraaij2016}, 
and restricting to $\topo$-strongly continuous semigroups $(S(t))_{t\geq 0}$, it is shown in \cite[Corollary 6.5]{Kraaij2016}\footnote{It seems that the condition that the semigroups considered in \cite[Corollary 6.5]{Kraaij2016} should be $\topo$-strongly continuous is missing. At least, it is indicated in the sketch of the proof given in \cite[p.~176]{Kraaij2016} that $\topo$-strong continuity is used.} 
that if $(S(t))_{t\geq 0}$ is of type $(M,\omega)$ in the sense of part \ref{it:type_2}, then it is of type $(M,\alpha)^{\ast}$ 
for all $\alpha>\omega$. 
Our way to prove \prettyref{cor:type} allows us to drop these additional conditions (and to consider cosine families as well) since we 
circumvent the Hille--Yosida type theorem \cite[Theorem 6.4]{Kraaij2016}, see \prettyref{rem:Hille-Yosida_correction}, where the sequential completeness and 
$\topo$-strong continuity are needed to define the Laplace transform of the semigroup by an improper Riemann integral 
in $(X,\topo)$.

\section{Bi-equicontinuity and bi-continuity}
\label{sect:bi-continuity}

We start this section with a version of sequential equicontinuity on a normed space 
that combines boundedness with respect to the norm and sequential equicontinuity w.r.t.~a weaker Hausdorff locally convex topology.

\begin{definition}\label{defi:bi-equicont}
Let $(X,\norm{\cdot})$ be a normed space, $\topo\subseteq\topo_{\norm{\cdot}}$ a Hausdorff locally convex topology on $X$, 
$I$ a set, $\calF$ a set of functions from $I$ to $[0,\infty)$ and $(S(t))_{t\in I}$ a family of 
linear maps from $X$ to $X$. Then $(S(t))_{t\in I}$ is called 
\begin{enumerate}[label=\upshape(\alph*), leftmargin=*]
\item\label{it:bi-equicont} \emph{bi-equicontinuous} if for every $(x_n)_{n\in\N}$ in $X$ and $x\in X$ with $\sup\limits_{n\in\N}\norm{x_n}<\infty$ and 
$\topo\text{-}\lim\limits_{n\to\infty} x_n = x$ we have
\[
\topo\text{-}\lim_{n\to\infty} S(t) (x_n-x) = 0
\]
uniformly for $t\in I$. 
\item\label{it:loc_bi-equicont} \emph{locally bi-equicontinuous} if $I$ is a Hausdorff topological space and 
$(S(t))_{t\in M}$ is bi-equicontinuous for all compact sets $M\subseteq I$.
\item \emph{$\calF$-bi-equicontinuous} if there is $k\in\calF$ such that $(k(t)S(t))_{t\in I}$ is bi-equicontin\-uous. 
\end{enumerate}
\end{definition}

\prettyref{defi:bi-equicont} \ref{it:bi-equicont} and \ref{it:loc_bi-equicont} are already introduced in 
\cite[Definition 1.2.3]{Farkas2003} for sequentially complete Saks spaces $(X,\norm{\cdot},\topo)$, 
see \prettyref{defi:saks_mixed} and \prettyref{defi:seq_compl_C-seq_Saks}, families $(S(t))_{t\in I}$ 
in $\calL(X)$ and in case \ref{it:loc_bi-equicont} under the assumption $I=[0,\infty)$. 
The origin of the definition of (local) bi-equicontinuity for $I\coloneq [0,\infty)$ is \cite[Definition 1.2]{Kuehnemund2001}, 
see \cite[Definition 2]{Kuehnemund2003} as well.

\begin{definition}[{\cite[I.3.2 Definition]{Cooper1978}}]\label{defi:saks_mixed}
    Let $(X,\norm{\cdot})$ be a normed space 
    %where we write $\topo_{\norm{\cdot}}$ for the induced topology, 
    and $\topo\subseteq\topo_{\norm{\cdot}}$ a Hausdorff locally convex topology on $X$.
 \begin{enumerate}[label=\upshape(\alph*), leftmargin=*]
    \item The triple $(X,\norm{\cdot},\topo)$ is called a \emph{Saks space} if there is a directed system of seminorms 
    $\semis_{\topo}$ that generates $\topo$ and fulfils    
    \begin{equation}\label{eq:saks}
    \norm{x} = \sup_{p\in\semis_{\topo}} p(x) \quad(x\in X).
    \end{equation}
    \item Let $(X,\norm{\cdot},\topo)$ be a Saks space. The \emph{mixed topology} $\gamma\coloneq\gamma(\norm{\cdot},\topo)$ is
	the finest linear topology on $X$ that coincides with $\topo$ on $\norm{\cdot}$-bounded sets and such that 
	$\topo\subseteq\gamma \subseteq\topo_{\norm{\cdot}}$.
\end{enumerate}
\end{definition}

The mixed topology $\gamma$ is actually Hausdorff locally convex and the definition given above is equivalent to the one 
from the literature \cite[Section 2.1]{Wiweger1961} by \cite[2.2.1, 2.2.2 Lemmas]{Wiweger1961}. 
Having Saks spaces at hand, we can now look back at the relation between $\topo$-bounded sets and $\norm{\cdot}$-bounded sets 
in \prettyref{lem:growth_bound_lcx} and \prettyref{thm:loc_equicont_exp_sgcs_countconv}.

\begin{remark}\label{rem:Saks_bounded_sets} 
    Let $(X,\norm{\cdot})$ be a normed space and $\topo\subseteq\topo_{\norm{\cdot}}$ a Hausdorff locally convex topology on $X$.
    \begin{enumerate}[label=\upshape(\alph*), leftmargin=*]
    \item\label{it:Saks_bounded_sets_1} If $\topo$-bounded sets are also $\norm{\cdot}$-bounded, then \cite[Lemma 4.4 (a)$\Rightarrow$(b)]{Kraaij2016}\footnote{In \cite[Lemma 4.4]{Kraaij2016} it is additionally assumed that $(X,\topo)$ is sequentially complete but this is not needed for the implication ``(a)$\Rightarrow$(b)'' there, only for its converse.} and 
    \cite[I.3.1 Lemma]{Cooper1978} imply that $(X,\norm{\cdot},\topo)$ is a Saks space and 
    \[
    \norm{x}=\sup_{p\in\calN_{\topo}}p(x) \quad (x\in X).
    \]
    \item\label{it:Saks_bounded_sets_2}  On the other hand, if $(X,\norm{\cdot},\topo)$ is a Saks space, then the $\gamma$-bounded sets are exactly the 
    $\norm{\cdot}$-bounded sets by \cite[I.1.11 Proposition]{Cooper1978}.
    \end{enumerate}
\end{remark}

Further, we recall the following related topology whose origin is \cite[3.1.1 Theorem]{Wiweger1961}. 

\begin{definition}[{\cite[Definition 3.9]{KruseSchwenninger2022}}]\label{defi:submixed_top}
Let $(X,\norm{\cdot},\topo)$ be a Saks space and $\semis_{\topo}$ a directed system of seminorms 
that generates the topology $\topo$ and fulfils \eqref{eq:saks}. We set 
\[
\mathcal{R}\coloneq \{(p_{n},a_{n})_{n\in\N}\;|\;(p_{n})_{n\in\N}\subseteq\semis_{\topo},\,(a_{n})_{n\in\N}\in c_{0}^{+}\}
\]
where $c_0^{+}$ is the set of all real non-negative null-sequences.
For $(p_{n},a_{n})_{n\in\N}\in\mathcal{R}$ we define the seminorm
\[
 \vertiii{x}_{(p_{n},a_{n})_{n\in\N}}\coloneq\sup_{n\in\N}p_{n}(x)a_{n} \quad (x\in X).
\]
We denote by $\gamma_{\operatorname{s}}\coloneq\gamma_{\operatorname{s}}(\norm{\cdot},\topo)$ 
the Hausdorff locally convex topology that is generated by 
the system of seminorms $(\vertiii{\cdot}_{(p_n,a_n)_{n\in\N}})_{(p_n,a_n)_{n\in\N}\in\mathcal{R}}$ and call it the 
\emph{submixed topology}.
\end{definition}

By \cite[I.1.10 Proposition]{Cooper1978}, \cite[3.1.1 Theorem]{Wiweger1961} and 
\cite[Lemma A.1.2]{Farkas2003}, the mixed and the submixed topology are related in the following way.

\begin{remark}[{\cite[Remark 3.10]{KruseSchwenninger2022}}]\label{rem:mixed=submixed}
Let $(X,\norm{\cdot},\topo)$ be a Saks space, $\semis_{\topo}$ a directed system of seminorms 
that generates the topology $\topo$ and fulfils \eqref{eq:saks}.
\begin{enumerate}[label=\upshape(\alph*), leftmargin=*]
\item We have $\topo\subseteq \gamma_{\operatorname{s}}\subseteq \gamma$ and $\gamma_{\operatorname{s}}$ has the same convergent 
sequences as $\gamma$.
\item If 
 \begin{enumerate}[label=\upshape(\roman*), leftmargin=*, widest=ii]
 \item for every $x\in X$, $\varepsilon>0$ and $p\in\semis_{\topo}$ there are $y,z\in X$ such that $x=y+z$, 
 $p(z)=0$ and $\norm{y}\leq p(x)+\varepsilon$, or 
 \item the $\norm{\cdot}$-closed unit ball $B_{\norm{\cdot}}\coloneq\{x\in X\;|\; \norm{x}\leq 1\}$ is $\topo$-compact,
 \end{enumerate}
then $\gamma=\gamma_{\operatorname{s}}$ holds. 
\end{enumerate}
\end{remark}

Several examples of Saks spaces for which $\gamma=\gamma_{\operatorname{s}}$ holds are given in 
\cite[Example 3.11]{KruseSchwenninger2022} and \cite[3.5 Corollary \& Section 4]{Kruse2024a}.

\begin{definition}[{\cite[Definitions 2.2, 5.7]{KruseSeifert2023a}}]\label{defi:seq_compl_C-seq_Saks}
We call a Saks space $(X,\norm{\cdot},\topo)$ \emph{sequentially complete} if $(X,\gamma)$ is sequentially complete. 
We call a Saks space $(X,\norm{\cdot},\topo)$ \emph{C-sequential} if $(X,\gamma)$ is C-sequential.
\end{definition}

\begin{remark}\label{rem:Saks_seq_complete_Banach}
Let $(X,\norm{\cdot},\topo)$ be a Saks space. 
\begin{enumerate}[label=\upshape(\alph*), leftmargin=*]
    \item\label{it:Saks_seq:compl_Banach_1} If $(X,\norm{\cdot},\topo)$ is sequentially complete, then $(X,\norm{\cdot})$ is a Banach space since 
    the sequential completeness of $(X,\gamma)$ and $\gamma\subseteq\topo_{\norm{\cdot}}$ imply that the normed space 
    $(X,\norm{\cdot})$ is sequentially complete and thus complete.
    \item\label{it:Saks_seq:compl_Banach_2} In \cite{KruseMeichsnerSeifert2021,KruseSchwenninger2022,KruseSeifert2023b,KruseSchwenninger2024} it is for convenience assumed 
    in the definition of a Saks space that $(X,\norm{\cdot})$ is a Banach space. Concerning the results from these references that 
    we use in the present paper, this additional assumption is not needed.  
\end{enumerate}    
\end{remark}

A sufficient condition for a Saks space to be C-sequential is that 
$\topo$ is metrisable on $B_{\norm{\cdot}}$, see \cite[Proposition 5.7]{KruseMeichsnerSeifert2021}. 
Another sufficient condition is that $(X,\gamma)$ is a Mackey--Mazur space, see \prettyref{rem:count_conv} 
(cf.~\cite[Proposition 3.14]{KruseSchwenninger2022}). 
We have the following relation of bi-equicontinuity to the notions of equicontinuity that were introduced before.

\begin{proposition}\label{prop:bi_equicont_mixed_equicont}
Let $(X,\norm{\cdot},\topo)$ be a Saks space, $I$ a set, $\calF$ a set of functions from $I$ to $[0,\infty)$ 
and $(S(t))_{t\in I}$ a family of linear maps from $X$ to $X$. Then the following assertions hold.
\begin{enumerate}[label=\upshape(\alph*), leftmargin=*]
\item\label{it:bi_equicont_mixed_1} $(S(t))_{t\in I}$ is (locally) bi-equicontinuous if and only if it is sequentially (locally) 
$\gamma$-equicontinuous. 
\item\label{it:bi_equicont_mixed_2} $(S(t))_{t\in I}$ is $\calF$-bi-equicontinuous if and only if it is sequentially $\calF$-$\gamma$-equicontin\-uous. 
\end{enumerate}
If in addition $(X,\norm{\cdot},\topo)$ is C-sequential, then the following assertions hold.
\begin{enumerate}[label=\upshape(\alph*), leftmargin=*]\setcounter{enumi}{2}
\item\label{it:bi_equicont_mixed_3} $(S(t))_{t\in I}$ is (locally) bi-equicontinuous if and only if it is (locally) 
$\gamma$-equicontin\-uous. 
\item\label{it:bi_equicont_mixed_4} $(S(t))_{t\in I}$ is $\calF$-bi-equicontinuous if and only if it is $\calF$-$\gamma$-equicontinuous. 
\end{enumerate}
\end{proposition}
\begin{proof}
Let $k\colon I\to [0,\infty)$. We note that our statements \ref{it:bi_equicont_mixed_1} and \ref{it:bi_equicont_mixed_2} follow 
once we have proved that $(k(t)S(t))_{t\in I}$ is bi-equicontinuous if and only if $(k(t)S(t))_{t\in I}$ is 
sequentially $\gamma$-equicontinuous. Parts \ref{it:bi_equicont_mixed_3} and \ref{it:bi_equicont_mixed_4} 
then follow from \ref{it:bi_equicont_mixed_1} and \ref{it:bi_equicont_mixed_2}, respectively, 
and \prettyref{thm:seq_equicont_equicont_equiv}. 

We observe that by \cite[I.1.10 Proposition]{Cooper1978} a sequence in $X$ is $\gamma$-convergent if and only if 
it is $\norm{\cdot}$-bounded and $\topo$-convergent. This implies that $(k(t)S(t))_{t\in I}$ is bi-equicontinuous if and only if 
$(k(t)S(t))_{t\in I}$ is sequentially $\gamma$-equicontinuous.
\end{proof}

\begin{definition}
  \label{defi:bicontinuoussg}
  Let $(X,\norm{\cdot})$ be a normed space, $\topo\subseteq\topo_{\norm{\cdot}}$ a Hausdorff locally convex topology 
  on $X$, $I$ a Hausdorff topological space and $(S(t))_{t\in I}$ a family in $\calL(X)$. 
  We say that $(S(t))_{t\in I}$ is \emph{(locally) $\topo$-bi-continuous} on $X$ if 
  it is $\topo$-strongly continuous, locally bounded and (locally) bi-equicontinuous.
\end{definition}

We point out that if we drop the word ``locally'' in locally $\topo$-bi-continuous, 
then we mean a $\topo$-strongly continuous, locally bounded and bi-equicontinuous family, so ``locally'' is not dropped 
in the boundedness condition. However, it actually does not make a difference when we consider Saks spaces 
due to \prettyref{lem:equi_cont_tight_norm_bound} and \prettyref{prop:bi_equicont_mixed_equicont} \ref{it:bi_equicont_mixed_1} 
since by \prettyref{rem:Saks_bounded_sets} \ref{it:Saks_bounded_sets_2}, a subset of $X$ is $\gamma$-bounded if and only if it is $\norm{\cdot}$-bounded. 
Our definition of (local) $\topo$-bi-continuity generalises \cite[Definition 1.2.4]{Farkas2003}.
In the literature different types of (locally) bi-continuous families on sequentially complete Saks spaces are considered. 
Locally (exponentially bounded) bi-continuous semigroups are studied for instance in \cite{Kuehnemund2001,Kuehnemund2003,Farkas2003,Es_SarhirFarkas2006,Farkas2011,BuddeFarkas2019}
and locally (exponentially bounded) bi-continuous cosine families in \cite{DuanSun2006,Budde2024,Budde2025}. 
Further, we mention \cite{WangSun2007,LuZhao2009} for locally bounded bi-continuous $C$-semigroups and $C$-cosine families,
\cite{LiSongZhao2010a,ChangLiu2012} for $\alpha$-times integrated exponentially bounded bi-continuous $C$-semigroups and $C$-cosine families 
for $\alpha\in\N$,
\cite{LiSongZhao2010b,LiSongYu2013} for $\alpha$-times integrated exponentially bounded bi-continuous $C$-semigroups and $C$-cosine families 
for $\alpha>0$. 

\begin{remark}
\label{rem:Saks_bi-admissible}
    In the literature, (local) bi-continuity is typically defined for families $(S(t))_{t\geq 0}$ on triples $(X,\norm{\cdot},\topo)$ where $(X,\norm{\cdot})$ is a Banach space and $\topo\subseteq \topo_{\norm{\cdot}}$ is a Hausdorff locally convex topology which is sequentially complete on $\topo$-closed and $\norm{\cdot}$-bounded sets (or, equivalenty, on the $\norm{\cdot}$-closed unit ball) and such that $(X,\topo)'$ is norming for $(X,\norm{\cdot})$, see e.g.~\cite[Assumptions 1]{Kuehnemund2003} or \cite[Assumption 1.2]{Farkas2003}. Due to \prettyref{rem:Saks_seq_complete_Banach} \ref{it:Saks_seq:compl_Banach_1} and 
    \cite[Lemma 2.7]{KruseSeifert2023b}, a triple $(X,\norm{\cdot},\topo)$ fulfils \cite[Assumptions 1]{Kuehnemund2003} if and only if 
    it is a sequentially complete Saks space.
\end{remark}

Let us turn to the relation between $\topo$-strong continuity and $\gamma$-strong continuity on a Saks space. 
We call a Hausdorff topological space $I$ a \emph{$k_{\operatorname{lcs}}$-space} 
if for any Hausdorff locally convex space $X$ and any map $f\colon I\to X$, 
whose restriction to each compact $M\subseteq I$ is continuous, the map is already continuous on 
$I$. Examples of $k_{\operatorname{lcs}}$-spaces are Hausdorff $k_{\R}$-spaces by 
\cite[(2.3.7) Proposition]{Buchwalter1969} and \cite[Proposition 3.27]{Fabian2011}. 
So, recalling from \cite[p.~15]{Kruse2024a}, examples of $k_{\operatorname{lcs}}$-spaces are completely regular Hausdorff 
$k$-spaces by 
\cite[3.3.21 Theorem]{Engelking1989}, metrisable spaces by \cite[Proposition 11.5]{James1999} and 
\cite[3.3.20 Theorem]{Engelking1989}, locally compact Hausdorff topological spaces 
and strong duals of Fr\'echet--Montel spaces by \cite[Proposition 3.27]{Fabian2011} and 
\cite[4.11 Theorem (5)]{Kriegl1997}.

\begin{proposition}\label{prop:tau_gamma_strong}
Let $(X,\norm{\cdot},\topo)$ be a Saks space, $I$ a Hausdorff topological space and $(S(t))_{t\in I}$ a family of linear maps 
from $X$ to $X$. Consider the following assertions.
\begin{enumerate}[label=\upshape(\alph*), leftmargin=*]
\item\label{it:tau_gamma_strong_1} $(S(t))_{t\in I}$ is $\topo$-strongly continuous.
\item\label{it:tau_gamma_strong_2} $(S(t))_{t\in I}$ is $\gamma$-strongly continuous.
\end{enumerate}
Then we have \ref{it:tau_gamma_strong_2}$\Rightarrow$\ref{it:tau_gamma_strong_1}. 
If $I$ is a $k_{\operatorname{lcs}}$-space and $(S(t))_{t\in I}$ a locally bounded family in $\calL(X)$, then 
\ref{it:tau_gamma_strong_1}$\Rightarrow$\ref{it:tau_gamma_strong_2}.
\end{proposition}
\begin{proof}
``\ref{it:tau_gamma_strong_2}$\Rightarrow$\ref{it:tau_gamma_strong_1}'' 
Let $(S(t))_{t\geq 0}$ be $\gamma$-strongly continuous. Then it is also $\topo$-strongly continuous since 
$\topo\subseteq\gamma$.

``\ref{it:tau_gamma_strong_1}$\Rightarrow$\ref{it:tau_gamma_strong_2}'' 
Let $I$ be a $k_{\operatorname{lcs}}$-space and $(S(t))_{t\in I}$ a $\topo$-strongly continuous locally bounded family 
in $\calL(X)$. Since $(S(t))_{t\in I}$ is locally bounded, the set $S(M)x$ is $\norm{\cdot}$-bounded for all compact $M\subseteq I$ 
and $x\in X$. As $\gamma$ coincides with $\topo$ on $\norm{\cdot}$-bounded sets and $(S(t))_{t\in I}$ is $\topo$-strongly 
continuous, this implies that the map $M\ni t\mapsto S(t)x\in X$ is $\gamma$-continuous for all compact 
$M\subseteq I$ and $x\in X$. Hence, $(S(t))_{t\in I}$ is $\gamma$-strongly continuous because $I$ is a $k_{\operatorname{lcs}}$-space. 
\end{proof}

\begin{remark}\label{rem:gamma_seq_cont_norm_cont}
Let $(X,\norm{\cdot},\topo)$ be a Saks space and $S$ a linear map from $X$ to $X$. 
If $S$ is sequentially $\gamma$-continuous on $X$, then $S\in\calL(X)$. 
Indeed, let $S$ be sequentially $\gamma$-continuous on $X$. Then $S$ maps $\gamma$-bounded subsets of $X$ to 
$\gamma$-bounded subsets of $X$ by \cite[4-4-3 Theorem]{Wilansky1978}. By \prettyref{rem:Saks_bounded_sets} \ref{it:Saks_bounded_sets_2}, 
a subset of $X$ is $\gamma$-bounded if and only if it is $\norm{\cdot}$-bounded. 
This implies that $S$ is $\topo_{\norm{\cdot}}$-continuous since the normed space $(X,\norm{\cdot})$ is bornological.
\end{remark}

For (locally) $\topo$-bi-continuous semigroups $(S(t))_{t\geq 0}$ on sequentially complete Saks spaces, 
the following observation is already contained in \cite[Remark 2.9]{KruseSeifert2023b}.

\begin{proposition}\label{prop:bi_cont_mixed_seq_cont}
Let $(X,\norm{\cdot},\topo)$ be a Saks space, $I$ a Hausdorff topological space and $(S(t))_{t\in I}$ a family of linear maps 
from $X$ to $X$. Consider the following assertions.
\begin{enumerate}[label=\upshape(\alph*), leftmargin=*]
\item\label{it:bi_mixed_1} $(S(t))_{t\in I}$ is (locally) $\topo$-bi-continuous.
\item\label{it:bi_mixed_2} $(S(t))_{t\in I}$ is $\gamma$-strongly continuous, sequentially (locally) $\gamma$-equicontinuous 
and locally bounded.
\end{enumerate}
Then we have \ref{it:bi_mixed_2}$\Rightarrow$\ref{it:bi_mixed_1}. 
If $I$ is a $k_{\operatorname{lcs}}$-space, then 
\ref{it:bi_mixed_1}$\Rightarrow$\ref{it:bi_mixed_2}.
\end{proposition}

\begin{proof}
First, we observe that by \prettyref{prop:bi_equicont_mixed_equicont} \ref{it:bi_equicont_mixed_1} 
(local) bi-equicontinuity is equivalent to sequential (local) $\gamma$-equicontinuity.

``\ref{it:bi_mixed_2}$\Rightarrow$\ref{it:bi_mixed_1}'' Let $(S(t))_{t\geq 0}$ be $\gamma$-strongly continuous, 
locally bounded and sequentially (locally) $\gamma$-equicontinuous. Due to our first observation and 
\prettyref{rem:gamma_seq_cont_norm_cont}, we only need to show that $(S(t))_{t\geq 0}$ is $\topo$-strongly continuous, 
which follows from \prettyref{prop:tau_gamma_strong}.

``\ref{it:bi_mixed_1}$\Rightarrow$\ref{it:bi_mixed_2}'' Let $(S(t))_{t\geq 0}$ be (locally) $\topo$-bi-continuous and 
$I$ a $k_{\operatorname{lcs}}$-space. Again, due to our first observation we only need to show that 
$(S(t))_{t\geq 0}$ is $\gamma$-strongly continuous, which follows from \prettyref{prop:tau_gamma_strong}.
\end{proof}

The following corollary was already observed in \cite[Remarks, p.~6]{Farkas2003} for bi-continuous semigroups 
on sequentially complete Saks spaces.

\begin{corollary}\label{cor:bi_cont_loc_exp_sg_cosine}
Let $(X,\norm{\cdot},\topo)$ be a Saks space such that $(X,\norm{\cdot})$ is a Banach space 
and $(S(t))_{t\geq 0}$ a semigroup or a cosine family on $X$. Then the following assertions are equivalent.
\begin{enumerate}[label=\upshape(\alph*), leftmargin=*]
  \item\label{it:bi_cont_loc_exp_sg_cosine_1}
   $(S(t))_{t\geq 0}$ is locally $\topo$-bi-continuous. 
  \item\label{it:bi_cont_loc_exp_sg_cosine_2}
    $(S(t))_{t\geq 0}$ is locally $\topo$-bi-continuous and exponentially bounded.  
\end{enumerate}  
\end{corollary}

\begin{proof}
The implication ``\ref{it:bi_cont_loc_exp_sg_cosine_2}$\Rightarrow$\ref{it:bi_cont_loc_exp_sg_cosine_1}'' is clear. 
Let us turn to ``\ref{it:bi_cont_loc_exp_sg_cosine_1}$\Rightarrow$\ref{it:bi_cont_loc_exp_sg_cosine_2}''. 
Let $(S(t))_{t\geq 0}$ be locally $\topo$-bi-continuous. By \prettyref{prop:bi_cont_mixed_seq_cont}, 
$(S(t))_{t\geq 0}$ is $\gamma$-strongly continuous.
We recall again that a subset of $X$ is $\norm{\cdot}$-bounded if and only if it is $\gamma$-bounded 
by \prettyref{rem:Saks_bounded_sets} \ref{it:Saks_bounded_sets_2}. 
Thus, $(S(t))_{t\geq 0}$ is exponentially bounded by \prettyref{lem:growth_bound}, which proves the implication.
\end{proof}

Let $I$ be a Hausdorff topological space and $\calF$ a set of functions from $I$ to $[0,\infty)$. 
We say that a function $k\in\calF$ is \emph{locally bounded} if $\sup_{t\in M}k(t)<\infty$ for all compact sets 
$M\subseteq I$, see e.g.~\cite[p.~53]{Kruse2023}. 
In particular, if $\calF$ consists of continuous functions, then all its elements are locally bounded. 
The proof of the next observation is a modification of the proof of \cite[Proposition 4(b)]{Kuehnemund2003}, 
where the case of locally bi-continuous semigroups and $\calF=\calF_{\exp}$ is treated.

\begin{proposition}\label{prop:loc_glob_bi_equicont}
Let $(X,\norm{\cdot})$ be a normed space, $\topo\subseteq\topo_{\norm{\cdot}}$ a Hausdorff locally convex topology on $X$, 
$I$ a Hausdorff topological space, $\calF$ a set of locally bounded functions from $I$ to $[0,\infty)$ having rV$_{\infty}$ 
and $(S(t))_{t\in I}$ a family in $\calL(X)$.
If $(S(t))_{t\in I}$ is locally bi-equicontinuous and $\calF$-bounded, then 
there is $k\in\calF$ such that $(k(t)S(t))_{t\in I}$ is bi-equicontinuous. 
If in addition all elements of $\calF$ are locally bounded away from zero, then $(k(t)S(t))_{t\in I}$ is bounded.
\end{proposition}

\begin{proof}
Let $(S(t))_{t\in I}$ be locally bi-equicontinuous and $\calF$-bounded. 
Let $(x_n)$ be a sequence in $X$ and $x\in X$ with $K_1\coloneqq\sup_{n\in\N}\norm{x_n}<\infty$ and $
\topo\text{-}\lim_{n\to\infty} x_n = x$. Let $p\in\semis_{\topo}$. Then there is $K_2>0$ such that 
$p(z)\leq K_{2}\norm{z}$ for all $z\in X$ as $\topo\subseteq\topo_{\norm{\cdot}}$. 
Since $(S(t))_{t\in I}$ is $\calF$-bounded, there is $k_{1}\in\calF$ such that 
$K_{3}\coloneq\sup_{t\in I}\norm{k_{1}(t)S(t)}_{\calL(X)}<\infty$. We set $K_{4}\coloneq K_{2}(K_{3}+1)(K_{1}+\norm{x}+1)>0$. 
Let $\varepsilon>0$. As $\calF$ has rV$_{\infty}$, there are $k\in\calF$, independent of $\varepsilon$, and a compact set 
$M\subseteq I$ such that $k(t)\leq\frac{\varepsilon}{K_{4}} k_{1}(t)$ for all $t\in I\setminus M$. This implies that 
\begin{align*}
     \norm{k(t)S(t)(x_n -x)}
&\leq k(t)\norm{S(t)}_{\calL(X)}\norm{x_n -x}
 \leq \frac{\varepsilon}{K_{4}}\norm{k_{1}(t)S(t)}_{\calL(X)}(K_1+\norm{x})\\
&=    \frac{\varepsilon}{K_{4}}K_{3}(K_1+\norm{x})
 \leq \frac{\varepsilon}{K_2}
\end{align*}
for all $t\in I\setminus M$ and $n\in\N$. Therefore, we have
\begin{equation}\label{eq:loc_glob_1}
     \sup_{t\in I\setminus M}p(k(t)S(t)(x_n -x))
\leq K_2\sup_{t\in I\setminus M}\norm{k(t)S(t)(x_n -x)}
\leq \varepsilon .
\end{equation}
By the local bi-equicontinuity of $(S(t))_{t\in I}$ and the local boundedness of $k$, there is $n_0\in\N$ such that for all $n\geq n_0$ it holds that
\begin{equation}\label{eq:loc_glob_2}
\sup_{t\in M}p(k(t)S(t)(x_n -x))
\leq \sup_{s\in M}k(s)\sup_{t\in M}p(S(t)(x_n -x))
\leq \varepsilon .
\end{equation}
Combining \eqref{eq:loc_glob_1} and \eqref{eq:loc_glob_2}, we obtain that $(k(t)S(t))_{t\in I}$ 
is bi-equicontinuous. 

Furthermore, fix some $x\in X$ and let in addition all elements of $\calF$ be locally bounded away from zero. Then we have
\begin{align*}
     \sup_{t\in I}\norm{k(t)S(t)}_{\calL(X)}
&\leq \max\Bigl\{\sup_{t\in M}\norm{k(t)S(t)}_{\calL(X)},\,\sup_{t\in I\setminus M}\norm{k(t)S(t)}_{\calL(X)}\Bigr\}\\
&\leq \max\Bigl\{\sup_{s\in M}k(s)\frac{1}{\inf_{r\in M}k_{1}(r)}\sup_{t\in M}\norm{k_{1}(t)S(t)}_{\calL(X)},\,\frac{\varepsilon K_3}{K_4}\Bigr\}
< \infty,
\end{align*}
so $(k(t)S(t))_{t\in I}$ is bounded.
\end{proof}

\begin{remark}\label{rem:standard_Fs_fine}
The sets $\calF\coloneq\calF_{\exp}$ and $\calF_{\operatorname{poly}}$ from \prettyref{defi:exp_pol_equicont} 
consist of strictly positive continuous functions 
on the metrisable space $I=[0,\infty)$, so $I$ is a $k_{\operatorname{lcs}}$-space, all the elements of these $\calF$ are locally bounded and locally bounded away from zero, and it is easy to check that these $\calF$ have rV$_{\infty}$.

We note that the set $\calF_{\exp ,2}$ from \prettyref{defi:exp_pol_equicont} does not have rV$_{\infty}$ since 
$\euler^{\alpha (t-s)}=1$ for all $\alpha\in\R$ and $t=s\in\R$.
\end{remark}

%\begin{remark}\label{rem:shift_k1_k}
%   Note that in the proof of \prettyref{prop:loc_glob_bi_equicont} we have to transition from $k_1$ which makes $(k_1(t)S(t))_{t\in I}$ bounded to $k$ which makes $(k(t)S(t))_{t\in I}$ bi-equicontinuous, which we obtained by rV$_\infty$. For the case $I\coloneqq [0,\infty)$ and $\calF\coloneqq\calF_{\exp}$ and $k_1 \coloneq [t\mapsto \euler^{-\omega t}]$ for some $\omega\in \R$ we may take any $k \coloneq [t\mapsto \euler^{-\alpha t}]$ with $\alpha>\omega$.
%\end{remark}

\section{Relations between bi-continuity, strong continuity, equicontinuity and equitightness}
\label{sect:relations_bi-continuity_equicontinuity}

We start with another related notion to equicontinuity on Saks spaces.

\begin{definition}[{\cite[Definitions 3.4, 3.5]{KruseSchwenninger2022}}]\label{defi:equitight}
Let $(X,\norm{\cdot})$ be a normed space, $\topo\subseteq\topo_{\norm{\cdot}}$ a Hausdorff locally convex topology on $X$, 
$\semis_{\topo}$ a directed system of seminorms generating the topology $\topo$, $I$ a set and 
$\calF$ a set functions from $I$ to $[0,\infty)$. A family of linear maps $(S(t))_{t\in I}$ from $X$ to $X$ is called 
\begin{enumerate}[label=\upshape(\alph*), leftmargin=*]
\item\label{it:equitight_1} $(\norm{\cdot},\topo)$\emph{-equitight}
if 
\[
\forall\;\varepsilon>0,\,p\in\semis_{\topo}\;\exists\;q\in\semis_{\topo},\,K\geq 0\;
\forall\;t\in I,\,x\in X:\;p(S(t)x)\leq K q(x)+\varepsilon\norm{x}.
\]
\item\label{it:equitight_2} \emph{locally $(\norm{\cdot},\topo)$-equitight} if $I$ is a Hausdorff topological space and 
$(S(t))_{t\in M}$ is $(\norm{\cdot},\topo)$-equitight for all compact sets $M\subseteq I$. 
\item\label{it:equitight_3} \emph{$\calF$-$(\norm{\cdot},\topo)$-equitight} if there is $k\in\calF$ such that 
$(k(t)S(t))_{t\in I}$ is $(\norm{\cdot},\topo)$-equi\-tight. 
\item\label{it:equitight_4} \emph{exponentially (polynomially)} $(\norm{\cdot},\topo)$-equi\-tight if it is $\calF$-$(\norm{\cdot},\topo)$-equitight 
with $\calF=\calF_{\exp}$ or $\calF_{\exp ,2}$ and $I$ accordingly ($\calF=\calF_{\operatorname{pol}})$, see \prettyref{defi:exp_pol_equicont}.
\end{enumerate}
\end{definition}

We remark that the definition of $(\norm{\cdot},\topo)$-equitightness given here is a bit more general than in 
\cite[Definitions 3.4, 3.5]{KruseSchwenninger2022} since we do not assume that $(X,\norm{\cdot},\topo)$ is a Saks space. 
Exponential $(\norm{\cdot},\topo)$-equi\-tightness is also called \emph{quasi-$(\norm{\cdot},\topo)$-equi\-tightness}. 
Equitight families $(S(t))_{t\in [0,t_{0}]}$ in $\calL(X)$ for $t_{0}\geq 0$ appeared in 
\cite[Definitions 1.2.20, 1.2.21]{Farkas2003} under the name \emph{local} first.
In the case of locally $\topo$-bi-continuous semigroups $(S(t))_{t\geq 0}$ the notion of tightness is used in 
\cite[Definition 1.1]{Es_SarhirFarkas2006}, meaning that $(S(t))_{t\in [0,t_{0}]}$ is equitight for all $t_{0}\geq 0$.  
Local equitightness plays an important role in perturbation results for locally bi-continu\-ous semigroups and 
we refer the reader for more information and related literature to \cite[p.~8--9]{KruseSchwenninger2022}, 
in particular \cite[Remark 3.8]{KruseSchwenninger2022}. Furthermore, for linear functionals on $X$ an estimate like the one 
in \prettyref{defi:equitight} \ref{it:equitight_1} also appears in \cite[3.6 Theorem]{dePauw2026}.
The following result is proved as \prettyref{prop:loc_bound_away_zero_equicont}.

\begin{proposition}\label{prop:loc_bound_away_zero_tight}
Let $(X,\norm{\cdot})$ be a normed space, $\topo\subseteq\topo_{\norm{\cdot}}$ a Hausdorff locally convex topology on $X$, 
$I$ a Hausdorff topological space, $\calF$ a set of functions from $I$ to $[0,\infty)$ that are locally bounded away from zero 
and $(S(t))_{t\in I}$ a family of linear maps from $X$ to $X$. 
If $(S(t))_{t\in I}$ is $\calF$-$(\norm{\cdot},\topo)$-equitight, then it is locally $(\norm{\cdot},\topo)$-equitight.
\end{proposition}

In particular, \prettyref{prop:loc_bound_away_zero_tight} yields that 
exponential (polynomial) $(\norm{\cdot},\topo)$-equi\-tightness implies local 
$(\norm{\cdot},\topo)$-equi\-tightness. We can now formulate one of the main theorems of this section.

\begin{theorem}\label{thm:bi_cont_loc_equicont_C_seq}
    Let $(X,\norm{\cdot},\topo)$ be a Saks space, $I$ a Hausdorff topological space 
    and $(S(t))_{t\in I}$ a family of linear maps from $X$ to $X$. 
    Consider the following assertions.
    \begin{enumerate}[label=\upshape(\alph*), leftmargin=*]
    \item\label{it:bi_loc_equicont_1}
        $(S(t))_{t\in I}$ is (locally) $\topo$-bi-continuous. 
    \item\label{it:bi_loc_equicont_2}
        $(S(t))_{t\in I}$ is $\gamma$-strongly continuous and (locally) $\gamma$-equicontinuous.
    \item\label{it:bi_loc_equicont_3} 
        $(S(t))_{t\in I}$ is $\gamma$-strongly continuous and (locally) $(\norm{\cdot},\topo)$-equitight. 
    \item\label{it:bi_loc_equicont_4} 
        $(S(t))_{t\in I}$ is $\gamma$-strongly continuous and in $\calL(X,\gamma)$.      
    \end{enumerate}
    Then we have the following.
    \begin{enumerate}[label=\upshape(\roman*), leftmargin=*, widest=iii]
    \item\label{it:bi_loc_equicont_rom1} If $(X,\gamma)$ is C-sequential and $I$ a $k_{\operatorname{lcs}}$-space, 
    then \ref{it:bi_loc_equicont_1}$\Leftrightarrow$\ref{it:bi_loc_equicont_2}. 
    \item\label{it:bi_loc_equicont_rom2} If $\gamma=\gamma_{\operatorname{s}}$, 
    then \ref{it:bi_loc_equicont_2}$\Rightarrow$\ref{it:bi_loc_equicont_3}. 
    If $(X,\norm{\cdot})$ is a Banach space and $(S(t))_{t\in I}$ a family in $\calL(X)$, then the local version of 
    \ref{it:bi_loc_equicont_3}$\Rightarrow$\ref{it:bi_loc_equicont_2} holds. 
    If $(S(t))_{t\in I}$ is a family in $\calL(X)$ and bounded, then the global version of 
    \ref{it:bi_loc_equicont_3}$\Rightarrow$\ref{it:bi_loc_equicont_2} holds. 
    \item\label{it:bi_loc_equicont_rom3} If $(X,\gamma)$ is a strong Mackey space, then the local version of 
    \ref{it:bi_loc_equicont_2}$\Leftrightarrow$\ref{it:bi_loc_equicont_4} holds. 
    \item\label{it:bi_loc_equicont_rom4} If $(X,\gamma)$ is a strong Mackey--Mazur space and 
    $I$ a $k_{\operatorname{lcs}}$-space, then the local version of 
    \ref{it:bi_loc_equicont_1}$\Leftrightarrow$\ref{it:bi_loc_equicont_2}$\Leftrightarrow$\ref{it:bi_loc_equicont_4} holds. 
    \end{enumerate}
\end{theorem}

\begin{proof}
    ``\ref{it:bi_loc_equicont_1}$\Rightarrow$\ref{it:bi_loc_equicont_2}'' 
    Let $(X,\gamma)$ be C-sequential, $I$ a $k_{\operatorname{lcs}}$-space and $(S(t))_{t\in I}$ (locally) $\topo$-bi-continuous. 
    Due to \prettyref{prop:bi_cont_mixed_seq_cont}, this implies that $(S(t))_{t\in I}$ is $\gamma$-strongly continuous 
    and (locally) sequentially $\gamma$-equicontinuous. 
    We deduce that $(S(t))_{t\in I}$ is (locally) $\gamma$-equicontinuous 
    by \prettyref{prop:bi_equicont_mixed_equicont} \ref{it:bi_equicont_mixed_3}.
    
    ``\ref{it:bi_loc_equicont_2}$\Rightarrow$\ref{it:bi_loc_equicont_1}'' We remark that a subset of $X$ is 
    $\norm{\cdot}$-bounded if and only if it is $\gamma$-bounded by \prettyref{rem:Saks_bounded_sets} \ref{it:Saks_bounded_sets_2}. 
    Therefore, the (local) $\gamma$-equicontinuity of $(S(t))_{t\in I}$ implies that $(S(t))_{t\in I}$ is a family in 
    $\calL(X)$ and (locally) bounded by \prettyref{lem:equi_cont_tight_norm_bound}. 
    This proves the implication by \prettyref{prop:bi_cont_mixed_seq_cont}.
  
    If $\gamma=\gamma_{\operatorname{s}}$, then ``\ref{it:bi_loc_equicont_2}$\Rightarrow$\ref{it:bi_loc_equicont_3}'' 
    follows from \cite[Proposition 3.16, (a)$\Rightarrow$(c)]{KruseSchwenninger2022}. 
    If $(X,\norm{\cdot})$ is a Banach space and $(S(t))_{t\in I}$ a $\gamma$-strongly continuous family in $\calL(X)$, 
    then it is locally bounded by \prettyref{lem:loc_bound_from_strong_cont}. So, the local and the global versions of 
    ``\ref{it:bi_loc_equicont_3}$\Rightarrow$\ref{it:bi_loc_equicont_2}'' 
    follow from \cite[Proposition 3.16, (c)$\Rightarrow$(f)$\Rightarrow$(g)]{KruseSchwenninger2022}.
    
    ``\ref{it:bi_loc_equicont_2}$\Leftrightarrow$\ref{it:bi_loc_equicont_4} (locally)'' 
    We only need to prove the local version of ``\ref{it:bi_loc_equicont_4}$\Rightarrow$\ref{it:bi_loc_equicont_2}'' 
    which follows from \prettyref{prop:str_M_str_cont_loc_equicont} and 
    \prettyref{rem:conditions_str_M_str_cont_loc_equicont} \ref{it:weak_Mackey_cont} and \ref{it:strong_Mackey} 
    if $(X,\gamma)$ is a strong Mackey space. 
    
     ``\ref{it:bi_loc_equicont_1}$\Leftrightarrow$\ref{it:bi_loc_equicont_4} (locally)'' 
     This follows from parts \ref{it:bi_loc_equicont_rom1} and \ref{it:bi_loc_equicont_rom3}, and 
     \prettyref{rem:count_conv} if $(X,\gamma)$ is a strong Mackey--Mazur space. 
\end{proof}

\begin{theorem}\label{thm:bi_cont_quasi_equicont_C_seq}
    Let $(X,\norm{\cdot},\topo)$ be a Saks space, $I$ a Hausdorff topological space, 
    $\calF$ a set functions from $I$ to $[0,\infty)$ and $(S(t))_{t\in I}$ a family of linear maps from $X$ to $X$. 
    Consider the following assertions.
    \begin{enumerate}[label=\upshape(\alph*), leftmargin=*]
    \item\label{it:bi_quasi_equicont_1}
        $(S(t))_{t\in I}$ is locally $\topo$-bi-continuous and $\calF$-bounded. 
    \item\label{it:bi_quasi_equicont_2}
        $(S(t))_{t\in I}$ is $\gamma$-strongly continuous and $\calF$-$\gamma$-equicontinuous.
    \item\label{it:bi_quasi_equicont_3} 
        $(S(t))_{t\in I}$ is $\gamma$-strongly continuous and $\calF$-$(\norm{\cdot},\topo)$-equitight.   
    \item\label{it:bi_quasi_equicont_4} 
        $(S(t))_{t\in I}$ is $\gamma$-strongly continuous, in $\calL(X,\gamma)$ and $\calF$-bounded.
    \end{enumerate}
    Then we have the following.
    \begin{enumerate}[label=\upshape(\roman*), leftmargin=*, widest=iii]
    \item\label{it:bi_quasi_equicont_rom1} If $(X,\gamma)$ is C-sequential, $I$ a $k_{\operatorname{lcs}}$-space and 
    $\calF$ consists of locally bounded functions and has rV$_{\infty}$, 
    then \ref{it:bi_loc_equicont_1}$\Rightarrow$\ref{it:bi_loc_equicont_2}. 
    If all elements of $\calF$ are locally bounded away from zero, 
    then \ref{it:bi_loc_equicont_2}$\Rightarrow$\ref{it:bi_loc_equicont_1}.
    \item\label{it:bi_quasi_equicont_rom2} If $\gamma=\gamma_{\operatorname{s}}$, 
    then \ref{it:bi_loc_equicont_2}$\Rightarrow$\ref{it:bi_loc_equicont_3}. 
    If $(S(t))_{t\in I}$ is a family in $\calL(X)$ and $\calF$-bounded, 
    then \ref{it:bi_loc_equicont_3}$\Rightarrow$\ref{it:bi_loc_equicont_2}.
    \item\label{it:bi_quasi_equicont_rom3} If $(X,\gamma)$ is a strong Mackey space, $I$ is locally compact, 
    $\calF$ consists of strictly positive continuous functions and has rV$_{\infty}$, 
    then \ref{it:bi_loc_equicont_2}$\Leftrightarrow$\ref{it:bi_loc_equicont_4}.
    \item\label{it:bi_quasi_equicont_rom4} If $(X,\gamma)$ is a strong Mackey--Mazur space, $I$ is locally compact, 
    $\calF$ consists of strictly positive continuous functions and has rV$_{\infty}$, 
    then \ref{it:bi_loc_equicont_1}$\Leftrightarrow$\ref{it:bi_loc_equicont_2}$\Leftrightarrow$\ref{it:bi_loc_equicont_4}.
    \end{enumerate}
\end{theorem}

\begin{proof}
    ``\ref{it:bi_quasi_equicont_1}$\Rightarrow$\ref{it:bi_quasi_equicont_2}'' 
    Let $(X,\gamma)$ be C-sequential, $I$ a $k_{\operatorname{lcs}}$-space, $\calF$ consist of locally bounded functions 
    and have rV$_{\infty}$, and 
    $(S(t))_{t\geq 0}$ locally $\topo$-bi-continuous and $\calF$-bounded. 
    By \prettyref{prop:loc_glob_bi_equicont}, there is $k\in \calF$ such that the set 
    $(k(t)S(t))_{t\geq 0}$ is $\topo$-bi-continuous on $X$. 
    Due to \prettyref{prop:bi_equicont_mixed_equicont} \ref{it:bi_equicont_mixed_4} and \prettyref{prop:bi_cont_mixed_seq_cont},
    this implies that $(k(t)S(t))_{t\in I}$ is $\gamma$-strongly continuous and $\calF$-$\gamma$-equicontinuous.
    
    ``\ref{it:bi_quasi_equicont_2}$\Rightarrow$\ref{it:bi_quasi_equicont_1}'' 
    Let all elements of $\calF$ be locally bounded away from zero and $(S(t))_{t\in I}$ $\gamma$-strongly continuous 
    and $\calF$-$\gamma$-equicontinuous. Again, we remark that a subset of $X$ is 
    $\norm{\cdot}$-bounded if and only if it is $\gamma$-bounded by \prettyref{rem:Saks_bounded_sets} \ref{it:Saks_bounded_sets_2}. 
    Therefore, the $\calF$-$\gamma$-equicontinuity of $(S(t))_{t\in I}$ implies that there is $k\in\calF$ such that 
    $(k(t)S(t))_{t\geq 0}$ is a family in 
    $\calL(X)$ and bounded by \prettyref{lem:equi_cont_tight_norm_bound}. 
    This proves the implication by \prettyref{prop:loc_bound_away_zero_equicont} and \prettyref{prop:bi_cont_mixed_seq_cont}.
       
    If $\gamma=\gamma_{\operatorname{s}}$, then ``\ref{it:bi_quasi_equicont_2}$\Rightarrow$\ref{it:bi_quasi_equicont_3}'' 
    follows from \cite[Proposition 3.16, (a)$\Rightarrow$(c)]{KruseSchwenninger2022}. 
    If $(S(t))_{t\in I}$ is a family in $\calL(X)$ and $\calF$-bounded, then 
    ``\ref{it:bi_quasi_equicont_3}$\Rightarrow$\ref{it:bi_quasi_equicont_2}'' follows from 
    \cite[Proposition 3.16, (c)$\Rightarrow$(f)$\Rightarrow$(g)]{KruseSchwenninger2022}.
    
    ``\ref{it:bi_quasi_equicont_2}$\Leftrightarrow$\ref{it:bi_quasi_equicont_4}'' First, let us consider
    ``\ref{it:bi_quasi_equicont_2}$\Rightarrow$\ref{it:bi_quasi_equicont_4}''. 
    If $(S(t))_{t\in I}$ is $\gamma$-strongly continuous and $\calF$-$\gamma$-equi\-con\-tin\-u\-ous, then it follows as above that 
    there is $k\in\calF$ such that $(k(t)S(t))_{t\geq 0}$ is a family in $\calL(X)$ and bounded. 
    By assumption, $k(t)>0$ for all $t\in I$, 
    and so $S(t)\in\calL(X,\gamma)$ for all $t\in I$ by the $\calF$-$\gamma$-equicontinuity. 
    Next, let us turn to ``\ref{it:bi_quasi_equicont_4}$\Rightarrow$\ref{it:bi_quasi_equicont_2}''. 
    Let $(S(t))_{t\in I}$ be a $\gamma$-strongly continuous family in $\calL(X,\gamma)$, so in particular in $\calL(X)$ 
    by \prettyref{rem:gamma_seq_cont_norm_cont}, and $\calF$-bounded. The implication follows from 
    \prettyref{thm:equicont_from_loc_eqicont_norm} in combination with 
    \prettyref{rem:conditions_str_M_str_cont_loc_equicont} \ref{it:weak_Mackey_cont} and \ref{it:strong_Mackey}. 
    
    ``\ref{it:bi_quasi_equicont_1}$\Leftrightarrow$\ref{it:bi_quasi_equicont_4}'' 
    This follows from parts \ref{it:bi_quasi_equicont_rom1} and \ref{it:bi_quasi_equicont_rom3}, the observation 
    that strong Mackey--Mazur spaces are C-sequential strong Mackey spaces by \prettyref{rem:count_conv}, 
    and the facts that the Hausdorff locally compact space $I$ is a $k_{\operatorname{lcs}}$-space 
    and that strictly positive continuous functions are locally bounded and locally bounded away from zero.
\end{proof}

Finally, we turn our attention to semigroups and cosine families. Let $(X,\topo)$ be a Hausdorff locally convex space.
We say that $(A,D(A))$ is a linear operator on $X$ if $D(A)$ is a linear subspace of $X$ and $A\colon D(A)\to X$ 
a linear operator. For a linear operator $(A,D(A))$ on $X$ and $\lambda\in\C$ we set $\lambda-A\coloneq \lambda\id-A$ and write 
\[
\rho_{\topo}(A)\coloneq\{\lambda\in\C\;|\;\lambda -A\text{ is bijective and }(\lambda-A)^{-1}\in\calL(X,\topo)\}
\]
for the \emph{resolvent set} of $A$ and $R(\lambda,A)\coloneqq (\lambda-A)^{-1}$ for the \emph{resolvent} of $A$ for $\lambda\in\rho_{\topo}(A)$, see \cite[p.~258]{AlbaneseBonetRicker2013}. If $\topo=\topo_{\norm{\cdot}}$ for a 
normed space $(X,\norm{\cdot})$, we just write $\rho_{\norm{\cdot}}(A)\coloneqq\rho_{\topo_{\norm{\cdot}}}(A)$. 

The \emph{generator} $A$ of a $\topo$-strongly continuous semigroup $(S(t))_{t\geq 0}$ on $X$ is the linear operator defined by 
\begin{align*}
  D(A)& \coloneq \bigl\{x\in X \;|\; \lim_{t\to 0\rlim} \tfrac{1}{t}(S(t)x-x)\text{ exists} \bigr\},\\
  Ax & \coloneq \lim_{t\to 0\rlim} \tfrac{1}{t}(S(t)x-x) \quad(x\in D(A)).
\end{align*}
The \emph{generator} $A$ of a $\topo$-strongly continuous cosine family $(S(t))_{t\geq 0}$ on $X$ is the linear operator 
defined by 
\begin{align*}
  D(A)& \coloneq \bigl\{x\in X \;|\; \lim_{t\to 0\rlim} \tfrac{2}{t^2} (S(t)x-x)\text{ exists} \bigr\},\\
  Ax & \coloneq \lim_{t\to 0\rlim} \tfrac{2}{t^2} (S(t)x-x) \quad(x\in D(A)).
\end{align*}
In the case that $(S(t))_{t\geq 0}$ is additionally assumed to be in $\calL(X,\topo)$, our definition of the generator 
coincides with the ones given in \cite[p.~260]{Komura1968}, \cite[p.~89]{Fattorini1969a} and \cite[p.~446]{Konishi1972}. 

Now, we relate the concepts introduced so far to the so-called integrable semigroups (and cosine families) from \cite{Kunze2009}, 
which are similar to weakly integrable semigroups studied in \cite{Jeffries1986,Jeffries1987} and 
should not be confused with integrated semigroups (and cosine families). 
Let $(X,\norm{\cdot}_{X})$ and $(Y,\norm{\cdot}_{Y})$ be Banach spaces which form a dual pair $\langle X,Y\rangle$. 
By \cite[Definition 1.1]{Kunze2009}, $\langle X, Y\rangle$ is called a \emph{norming dual pair} if 
\[
\norm{x}_{X}=\sup\{|\langle x,y\rangle|\;|\;y\in Y,\,\norm{y}_{Y}\leq 1\} \quad (x\in X)
\]
and 
\[
\norm{y}_{Y}=\sup\{|\langle x,y\rangle|\;|\;x\in X,\,\norm{x}_{X}\leq 1\} \quad (y\in Y).
\]
Let $\langle X, Y\rangle$ be a norming dual pair. 
By \cite[Definition 2.1]{Kunze2009}, a family $(S(t))_{t\geq 0}$ in $\calL(X,\sigma(X,Y))$ is called 
an \emph{integrable semigroup} on $\langle X, Y\rangle$ if $(S(t))_{t\geq 0}$ is an exponentially bounded semigroup 
and there is some $\omega\in\R$ such that for all $\lambda\in\C$ with $\TextRe(\lambda)>\omega$, there exists a map 
$R(\lambda)\in\calL(X,\sigma(X,Y))$ such that
\[
\langle R(\lambda)x,y\rangle=\int_{0}^{\infty}\euler^{-\lambda t}\langle S(t)x,y\rangle\d t\quad (x\in X,\,y\in Y).
\]
In particular, it is assumed that all the integrals on the right-hand side exist.
In the same style we define integrable cosine families on norming dual pairs. 
We call a family $(S(t))_{t\geq 0}$ in $\calL(X,\sigma(X,Y))$ an \emph{integrable cosine family} 
on $\langle X, Y\rangle$ if $(S(t))_{t\geq 0}$ is an exponentially bounded cosine family
and there is some $\omega\geq 0$ such that for all $\lambda\in\C$ with $\TextRe(\lambda)>\omega$, there exists a map 
$R(\lambda^2)\in\calL(X,\sigma(X,Y))$ such that
\[
\langle \lambda R(\lambda^2)x,y\rangle=\int_{0}^{\infty}\euler^{-\lambda t}\langle S(t)x,y\rangle\d t\quad (x\in X,\,y\in Y).
\]
In particular, it is again assumed that all the integrals on the right-hand side exist. 

Let $(X,\norm{\cdot},\topo)$ be a Saks space such that $(X,\norm{\cdot})$ is a Banach space. 
By \cite[I.1.18 Proposition]{Cooper1978}, $(X,\gamma)'$ is a closed subspace 
of $X_{\operatorname{n}}'\coloneq(X,\norm{\cdot})'$ w.r.t.~to the dual norm $\norm{\cdot}_{X_{\operatorname{n}}'}$. 
Thus, $\langle X,(X,\gamma)'\rangle$ is a dual pair of Banach spaces 
in a canonical way and $(X,\norm{\cdot})$ is norming for $(X,\gamma)'$. 
On the other hand, $((X,\gamma)',\norm{\cdot}_{X_{\operatorname{n}}'})$ is norming for $(X,\norm{\cdot})$ by 
\cite[Lemma 5.5 (a)]{KruseMeichsnerSeifert2021}. Hence, $\langle X,(X,\gamma)'\rangle$ is a norming dual pair.

\begin{corollary}\label{cor:bi_cont_quasi_equicont_C_seq_sg_cosine} 
    Let $(X,\norm{\cdot},\topo)$ be a Saks space and $(S(t))_{t\geq 0}$ a semigroup or a cosine family on $X$. 
    Consider the following assertions.
    \begin{enumerate}[label=\upshape(\alph*), leftmargin=*]
    \item\label{it:bi_cont_sg_cosine_1}
        $(S(t))_{t\geq 0}$ is locally $\topo$-bi-continuous. 
    \item\label{it:bi_cont_sg_cosine_2}
        $(S(t))_{t\geq 0}$ is locally $\topo$-bi-continuous and exponentially bounded.     
    \item\label{it:bi_cont_sg_cosine_3}
        $(S(t))_{t\geq 0}$ is $\gamma$-strongly continuous and locally $\gamma$-equicontinuous.
    \item\label{it:bi_cont_sg_cosine_4}
        $(S(t))_{t\geq 0}$ is $\gamma$-strongly continuous and exponentially $\gamma$-equicontinuous.   
    \item\label{it:bi_cont_sg_cosine_5} 
        $(S(t))_{t\geq 0}$ is $\gamma$-strongly continuous, locally $(\norm{\cdot},\topo)$-equitight and in $\calL(X)$.
    \item\label{it:bi_cont_sg_cosine_6}
        $(S(t))_{t\geq 0}$ is $\gamma$-strongly continuous, exponentially $(\norm{\cdot},\topo)$-equitight and in $\calL(X)$. 
    \item\label{it:bi_cont_sg_cosine_7} 
        $(S(t))_{t\geq 0}$ is $\gamma$-strongly continuous and in $\calL(X,\gamma)$.   
    \item\label{it:bi_cont_sg_cosine_8}  
    $(S(t))_{t\geq 0}$ is $\gamma$-strongly continuous and integrable on $\langle X,(X,\gamma)'\rangle$.  
    \end{enumerate}
    Then we have the following.
    \begin{enumerate}[label=\upshape(\roman*), leftmargin=*, widest=iii]
    \item\label{it:bi_cont_sg_cosine_rom1} If $(X,\norm{\cdot})$ is a Banach space and $(X,\gamma)$ C-sequential, then \ref{it:bi_cont_sg_cosine_1}$\Leftrightarrow$\ref{it:bi_cont_sg_cosine_2}$\Leftrightarrow$\ref{it:bi_cont_sg_cosine_3}$\Leftrightarrow$\ref{it:bi_cont_sg_cosine_4}.
    \item\label{it:bi_cont_sg_cosine_rom2} If $\calN_{\gamma}$ is countably quasi-convex, then 
    \ref{it:bi_cont_sg_cosine_3}$\Leftrightarrow$\ref{it:bi_cont_sg_cosine_4}. 
    In particular, this holds if $(X,\gamma)$ is a Mackey space.
    \item\label{it:bi_cont_sg_cosine_rom3} If $\gamma=\gamma_{\operatorname{s}}$, then \ref{it:bi_cont_sg_cosine_3}$\Rightarrow$\ref{it:bi_cont_sg_cosine_5} and 
    \ref{it:bi_cont_sg_cosine_4}$\Rightarrow$\ref{it:bi_cont_sg_cosine_6}.
    If $(X,\norm{\cdot})$ is a Banach space, then \ref{it:bi_cont_sg_cosine_5}$\Rightarrow$\ref{it:bi_cont_sg_cosine_3} and 
    \ref{it:bi_cont_sg_cosine_6}$\Rightarrow$\ref{it:bi_cont_sg_cosine_4}. 
    \item\label{it:bi_cont_sg_cosine_rom4} If $(X,\gamma)$ is sequentially complete, then \ref{it:bi_cont_sg_cosine_4}$\Rightarrow$\ref{it:bi_cont_sg_cosine_8}. 
    \item\label{it:bi_cont_sg_cosine_rom5} If $(X,\norm{\cdot})$ is a Banach space and $(X,\gamma)$ a Mackey space, then 
    \ref{it:bi_cont_sg_cosine_8}$\Rightarrow$\ref{it:bi_cont_sg_cosine_7}.
    \item\label{it:bi_cont_sg_cosine_rom6} If $(X,\gamma)$ is a strong Mackey space, then \ref{it:bi_cont_sg_cosine_3}$\Leftrightarrow$\ref{it:bi_cont_sg_cosine_4}$\Leftrightarrow$\ref{it:bi_cont_sg_cosine_7}. If $(X,\gamma)$ is in addition sequentially complete, 
    then \ref{it:bi_cont_sg_cosine_3}$\Leftrightarrow$\ref{it:bi_cont_sg_cosine_4}$\Leftrightarrow$\ref{it:bi_cont_sg_cosine_7}$\Leftrightarrow$\ref{it:bi_cont_sg_cosine_8}.
    \item\label{it:bi_cont_sg_cosine_rom7} If $(X,\norm{\cdot})$ is a Banach space and $(X,\gamma)$ a strong Mackey--Mazur space, 
    then \ref{it:bi_cont_sg_cosine_1}$\Leftrightarrow$\ref{it:bi_cont_sg_cosine_2}$\Leftrightarrow$\ref{it:bi_cont_sg_cosine_3}$\Leftrightarrow$\ref{it:bi_cont_sg_cosine_4}$\Leftrightarrow$\ref{it:bi_cont_sg_cosine_7}. 
    If $(X,\gamma)$ is a sequentially complete Mackey--Mazur space, then \ref{it:bi_cont_sg_cosine_1}$\Leftrightarrow$\ref{it:bi_cont_sg_cosine_2}$\Leftrightarrow$\ref{it:bi_cont_sg_cosine_3}$\Leftrightarrow$\ref{it:bi_cont_sg_cosine_4}$\Leftrightarrow$\ref{it:bi_cont_sg_cosine_7}$\Leftrightarrow$\ref{it:bi_cont_sg_cosine_8}. 
    \end{enumerate}
\end{corollary}

\begin{proof}
\ref{it:bi_cont_sg_cosine_rom1} This statement follows from \prettyref{cor:bi_cont_loc_exp_sg_cosine}, 
\prettyref{rem:standard_Fs_fine}, \prettyref{thm:bi_cont_loc_equicont_C_seq} \ref{it:bi_loc_equicont_rom1} and \prettyref{thm:bi_cont_quasi_equicont_C_seq} \ref{it:bi_quasi_equicont_rom1} 
with $I\coloneq [0,\infty)$ and $\calF\coloneq\calF_{\exp}$.

\ref{it:bi_cont_sg_cosine_rom2} The first part of the statement follows from \prettyref{prop:loc_bound_away_zero_equicont} with 
$I$ and $\calF$ as above, \prettyref{thm:loc_equicont_exp_sgcs_countconv} and \prettyref{rem:Saks_bounded_sets} \ref{it:Saks_bounded_sets_2}. 
For the second part, let $(X,\gamma)$ be a Mackey space. Then it is $c_0$-barrelled by \cite[Corollary A.11 (a)]{KruseSchwenninger2025}. 
Thus, $\calN_{\gamma}$ is countably (quasi-)convex by \prettyref{rem:count_conv} \ref{it:count_conv_2}, which 
yields the second part of the statement by the first part.

\ref{it:bi_cont_sg_cosine_rom3} The local version follows from \prettyref{thm:bi_cont_loc_equicont_C_seq} \ref{it:bi_loc_equicont_rom2}. 
The exponential version follows from \prettyref{lem:growth_bound}, \prettyref{rem:Saks_bounded_sets} \ref{it:Saks_bounded_sets_2} 
and \prettyref{thm:bi_cont_quasi_equicont_C_seq} \ref{it:bi_quasi_equicont_rom2}.

\ref{it:bi_cont_sg_cosine_rom4} Let $(X,\gamma)$ be sequentially complete and 
$(S(t))_{t\geq 0}$ $\gamma$-strongly continuous and exponentially $\gamma$-equicontinuous. 
Then $(X,\norm{\cdot})$ is a Banach space by \prettyref{rem:Saks_seq_complete_Banach} \ref{it:Saks_seq:compl_Banach_1}. 
Let us consider the case that $(S(t))_{t\geq 0}$ is a semigroup first.
By (the proof of) \cite[Lemma 4.4]{AlbaneseBonetRicker2013} and the exponential $\gamma$-equicontinu\-ity, 
there is some $\omega\in\R$ such that $\{\lambda\in\C\;|\;\TextRe(\lambda)>\omega\}\subseteq\rho_{\gamma}(A)$ and for all $\lambda\in\C$ 
with $\TextRe(\lambda)>\omega$, we have that
\[
R(\lambda,A)x=\int_{0}^{\infty}\euler^{-\lambda t}S(t)x\d t
\]
for all $x\in X$ where $A$ is the generator of the semigroup $(S(t))_{t\geq 0}$ and the integral above is an improper 
Riemann integral in the sequentially complete space $(X,\gamma)$. 
In particular, $R(\lambda,A)\in\calL(X,\gamma)$ for all $\lambda\in\C$ with $\TextRe(\lambda)>\omega$. 
In combination with \prettyref{rem:conditions_str_M_str_cont_loc_equicont} \ref{it:weak_Mackey_cont} this implies that 
$(S(t))_{t\geq 0}$ is integrable w.r.t.~the norming dual pair $\langle X,(X,\gamma)'\rangle$.

Now, let $(S(t))_{t\geq 0}$ be a cosine family. 
By \cite[Proposition 2]{LiShaw1993} and the exponential $\gamma$-equicontinuity, 
there is some $\omega\geq 0$ such that $(\omega^2,\infty)\subseteq\rho_{\gamma}(A)$ and for all $\lambda\in\C$ 
with $\TextRe(\lambda)>\omega$, we have that
\begin{equation}\label{eq:resol_laplace_cosine}
\lambda R(\lambda^2,A)x=\int_{0}^{\infty}\euler^{-\lambda t}S(t)x\d t
\end{equation}
for all $x\in X$ where is $A$ is the generator of the cosine family $(S(t))_{t\geq 0}$ and the integral above is an improper Riemann integral in the sequentially complete space $(X,\gamma)$ again. As in the case of semigroups, this implies our statement.

\ref{it:bi_cont_sg_cosine_rom5} This implication follows 
from \prettyref{rem:conditions_str_M_str_cont_loc_equicont} \ref{it:weak_Mackey_cont}. 

\ref{it:bi_cont_sg_cosine_rom6} The equivalence ``\ref{it:bi_cont_sg_cosine_3}$\Leftrightarrow$\ref{it:bi_cont_sg_cosine_7}'' 
is a consequence of \prettyref{thm:bi_cont_loc_equicont_C_seq} \ref{it:bi_loc_equicont_rom3}. 
The rest follows from parts \ref{it:bi_cont_sg_cosine_rom2}, \ref{it:bi_cont_sg_cosine_rom4} and \ref{it:bi_cont_sg_cosine_rom5}, and 
\prettyref{rem:Saks_seq_complete_Banach} \ref{it:Saks_seq:compl_Banach_1}.

\ref{it:bi_cont_sg_cosine_rom7} This statement follows from parts \ref{it:bi_cont_sg_cosine_rom1} and 
\ref{it:bi_cont_sg_cosine_rom6}, and the fact that Mackey--Mazur spaces are C-sequential, 
and if they are in addition sequentially complete, strong Mackey spaces 
by \prettyref{rem:conditions_str_M_str_cont_loc_equicont} \ref{it:sufficient_strong_Mackey_1} and \prettyref{rem:count_conv}.
\end{proof}

If $(S(t))_{t\geq 0}$ is a semigroup, then \prettyref{cor:bi_cont_quasi_equicont_C_seq_sg_cosine} \ref{it:bi_cont_sg_cosine_rom1} 
and \ref{it:bi_cont_sg_cosine_rom3} imply and 
generalise \cite[Theorem 7.4]{Kraaij2016} and \cite[Theorem 3.17]{KruseSchwenninger2022}.

\begin{remark}
    \prettyref{cor:bi_cont_quasi_equicont_C_seq_sg_cosine} \ref{it:bi_cont_sg_cosine_rom1}  states that for C-sequential Saks spaces 
    $(X,\norm{\cdot},\topo)$ such that $(X,\norm{\cdot})$ is a Banach space, the locally $\topo$-bi-continuous semigroups are exactly the $\gamma$-strongly continuous and exponentially $\gamma$-equicontinuous semigroups.

    If we drop the assumption that the Saks space is C-sequential, there do exist examples of locally $\topo$-bi-continuous semigroups which are not exponentially $\gamma$-equicontinuous, see \cite[Example 4.1]{Farkas2011} and 
    \cite[Example 3.7]{KruseSchwenninger2022}.    
    Namely, let $\omega_{1}$ be the first uncountable ordinal and $\Omega\coloneq [0,\omega_{1})$ equipped with the order topology. We denote by 
    $\mathrm{C}_{\operatorname{b}}(\Omega)$ the space of bounded $\R$-valued continuous functions on $\Omega$, by $\norm{\cdot}_{\infty}$ the supremum norm 
    and by $\topo_{\operatorname{co}}$ the compact-open topology on this space, respectively. 
    The tuple $(\mathrm{C}_{\operatorname{b}}(\Omega),\norm{\cdot}_{\infty},\topo_{\operatorname{co}})$ is a (sequentially) complete Saks space 
    by \cite[Example D), p.~65--66]{Wiweger1961} and \cite[II.1.9 Corollary]{Cooper1978} since $\Omega$ is a locally compact Hausdorff topological space.
    Further, $(\mathrm{C}_{\operatorname{b}}(\Omega),\gamma)$ is not a Mackey space by \cite[Theorem 5]{Conway1966} and \cite[Theorem 2.4]{Sentilles1972}. 
    Let $\beta\Omega$ be the Stone--\v{C}ech compactification of $\Omega$, i.e.~$\beta\Omega=[0,\omega_{1}]$, and set $x\coloneq \omega_{1}$. 
    Define the operator $A\colon\mathrm{C}_{\operatorname{b}}(\Omega)\to \mathrm{C}_{\operatorname{b}}(\Omega)$ by $Af(z)\coloneq f(x)$, 
    $z\in\Omega$, for $f\in\mathrm{C}_{\operatorname{b}}(\Omega)$ with $f$ continuously extended to $\beta\Omega$. 
    Then $A\in\calL(\mathrm{C}_{\operatorname{b}}(\Omega))$ and it generates the locally $\topo_{\operatorname{co}}$-bi-continuous semigroup on 
    $\mathrm{C}_{\operatorname{b}}(\Omega)$ given by 
    \[
    S(t)\coloneq\id-A+\euler^{t}A \quad (t\geq 0)
    \]
    and $S(t)\notin\calL(\mathrm{C}_{\operatorname{b}}(\Omega),\gamma)$ for $t>0$ by \cite[Example 4.1]{Farkas2011} and 
    \cite[Proposition 3.12]{KruseSchwenninger2022}. In particular, $(S(t))_{t\geq 0}$ is $\gamma$-strongly continuous 
    by \prettyref{prop:tau_gamma_strong} but not locally $\gamma$-equicontinuous, and 
    $(\mathrm{C}_{\operatorname{b}}(\Omega),\gamma)$ is not C-sequential by \prettyref{thm:seq_equicont_equicont_equiv} 
    and \prettyref{prop:bi_equicont_mixed_equicont} \ref{it:bi_equicont_mixed_3}.

    Now, we will show that $(S(t))_{t\geq 0}$ is also not integrable on the norming dual pair 
    $\langle \mathrm{C}_{\operatorname{b}}(\Omega),(\mathrm{C}_{\operatorname{b}}(\Omega),\gamma)'\rangle$, which was to the 
    best of our knowledge not observed in the literature before. Suppose that $(S(t))_{t\geq 0}$ belongs to 
    $\calL(\mathrm{C}_{\operatorname{b}}(\Omega),\sigma(\mathrm{C}_{\operatorname{b}}(\Omega),(\mathrm{C}_{\operatorname{b}}(\Omega),\gamma)'))$. We denote by $\mathds{1}$ the function that is constant $1$ on $\Omega$. 
    Fix $t>0$ and $y\in (\mathrm{C}_{\operatorname{b}}(\Omega),\gamma)'$. By \cite[Lemma 23.28]{MeiseVogt1997}, there is $y_{t}\in (\mathrm{C}_{\operatorname{b}}(\Omega),\gamma)'$ such that 
    \[
     y_{t}(f)
    =y(S(t)f)
    =y(f)-y(f(x)\mathds{1})+y(f(x)\euler^{t}\mathds{1}) 
    =y(f)-f(x)y(\mathds{1})+f(x)\euler^{t}y(\mathds{1})  \quad (f\in \mathrm{C}_{\operatorname{b}}(\Omega)),
    \]
    implying that $[f\mapsto -f(x)y(\mathds{1})+f(x)\euler^{t}y(\mathds{1})]\in(\mathrm{C}_{\operatorname{b}}(\Omega),\gamma)'$. 
    Due to \cite[Theorems 2.3 (b), 2.4]{Sentilles1972} and the local compactness of $\Omega$, the topology $\gamma$ is induced by the system of seminorms given by  
    \[
    |f|_{\nu}\coloneq\sup_{z\in\Omega}|f(z)|\nu(z) \quad (f\in\mathrm{C}_{\operatorname{b}}(\Omega))
    \]
    for non-negative $\nu\in\mathrm{C}_{0}(\Omega)$ where $\mathrm{C}_{0}(\Omega)$ is the space of real-valued continuous 
    functions on $\Omega$ that vanish at infinity. Hence, there are $\nu\in\mathrm{C}_{0}(\Omega)$ and $C\geq 0$ such that 
    \[
    |-f(x)y(\mathds{1})+f(x)\euler^{t}y(\mathds{1})|\leq C|f|_{\nu} \quad (f\in\mathrm{C}_{\operatorname{b}}(\Omega)).
    \]
    As $[0,\omega_{1}]$ is a completely regular Hausdorff space by \cite[Part II, 43.~Closed ordinal space {$[0,\Omega]$}; 4, p.~69]{SteenSeebach1978}, for every compact $K\subseteq\Omega$ there exists a continuous function 
    $f_{K}\colon [0,\omega_{1}]\to [0,1]$ such that $f_{K}=0$ on $K$ and $f_{K}(x)=1$ by \cite[(2.1.5) Proposition]{Buchwalter1969}. 
    Let $\varepsilon>0$. Since $\nu\in\mathrm{C}_{0}(\Omega)$, there is a compact set $K_{\varepsilon}\subseteq\Omega$ such that 
    $\nu(z)<\varepsilon$ for all $z\in\Omega\setminus K_{\varepsilon}$. It follows that 
    \[
      |-1+\euler^{t}||y(\mathds{1})|
     =|-y(\mathds{1})+\euler^{t}y(\mathds{1})|
     =|-f_{K_{\varepsilon}}(x)y(\mathds{1})+f_{K_{\varepsilon}}(x)\euler^{t}y(\mathds{1})|
     \leq |f_{K_{\varepsilon}}|_{\nu}<\varepsilon ,
    \]
    yielding that $y(\mathds{1})=0$ because $\varepsilon>0$ is arbitrary and $t>0$. 
    It follows that $y(S(t)f)=y(f)$ for all $t>0$ and $f\in\mathrm{C}_{\operatorname{b}}(\Omega)$, so $S(t)=\id$ for all $t>0$ by the Hahn--Banach theorem, which is a contradiction. Therefore, $(S(t))_{t\geq 0}$ is not integrable.
    
    This example may be interesting from a theoretical point of view; however, in all relevant applications (known so far) 
    the Saks space is C-sequential, see \cite[Section 4]{KruseSchwenninger2022}.
\end{remark}

\section{Applications and Consequences}
\label{sect:applications}

In this section, we show that we can recover statements for special classes of (bi-continuous) semigroups or cosine families available in the literature from the corresponding statements for semigroups or cosine families on Hausdorff locally convex spaces. 
We recall some notions first. 

\begin{definition}[{\cite[Definition 6]{AlbaneseKuehnemund2002}, \cite[Definitions 3.1, 3.5]{AlbaneseJornet2016}}]
Let $(X,\topo)$ be a Hausdorff locally convex space, $(A,D(A))$ a linear operator on $X$ and $D$ a linear subspace of $D(A)$. 
\begin{enumerate}[label=\upshape(\alph*), leftmargin=*]
\item $A$ is called \emph{(sequentially) closed} if for each net $(x_{i})_{i\in I}\subseteq D(A)$ (sequence $(x_{i})_{i\in\N}\subseteq D(A)$)
satisfying $x_{i}\to x$ and $Ax_{i}\to y$ w.r.t.~$\topo$ for some $x,y\in X$, we have $x\in D(A)$ and $Ax=y$. 
\item A linear operator $(B,D(B))$ on $X$ is called an \emph{extension} of $(A,D(A))$ if $D(A)\subseteq D(B)$ and $B_{\mid D(A)}=A$.  
The operator $A$ is called \emph{(sequentially) closable} if it admits a (sequentially) closed extension. 
The smallest (sequentially) closed extension of a (sequentially) closable operator $A$ is called the \emph{(sequential) closure} 
of $A$ and denoted by $\overline{A}$ ($\overline{A}^{\operatorname{seq}}$).
\item $A$ is called \emph{(sequentially) densely defined} if $D(A)$ is (sequentially) dense in $(X,\topo)$.
\item $D$ is called a \emph{(sequential) core} for $A$ if $D$ is (sequentially) dense in $D(A)$ w.r.t.~the topology induced by the system of seminorms given by 
\[
p_{A}(x)\coloneq p(x)+p(Ax)\quad (x\in D(A))
\]
for $p\in\semis_{\topo}$.
\end{enumerate}
If we want to emphasise the dependency of these notions on $\topo$, we write (sequentially) $\topo$-closed (closable), (sequential) $\topo$-closure (core), 
$\overline{A}^{\topo}$ ($\overline{A}^{\topo\text{-}\operatorname{seq}}$) and (sequentially) $\topo$-densely defined instead of just 
(sequentially) closed (closable), (sequential) closure (core), $\overline{A}$ ($\overline{A}^{\operatorname{seq}}$)
and (sequentially) densely defined, respectively.
\end{definition}

We note that every sequentially dense set is already dense and so a sequentially densely defined linear operator 
is densely defined, and a sequential core is a core. On Saks spaces the following variants of the preceding notions exist.

\begin{definition}[{\cite[Definition 3.4]{AlbaneseMangino2004}, \cite[Definitions 1.10, 1.20, Proposition 1.16 (c)]{Kuehnemund2001}}]
Let $(X,\norm{\cdot},\topo)$ be a Saks space, $(A,D(A))$ a linear operator on $X$ and $D$ a linear subspace of $D(A)$. 
\begin{enumerate}[label=\upshape(\alph*), leftmargin=*]
\item $A$ is called \emph{bi-closed} if for each sequence $(x_i)_{i\in\N}$ in $D(A)$ with $(x_i)$, $(Ax_i)$ $\norm{\cdot}$-bounded 
and $x_i\to x$ and $Ax_i \to y$ w.r.t.~$\topo$ for some $x,y\in X$, we have $x\in D(A)$ and $Ax=y$.
\item $A$ is called \emph{bi-closable} if it admits a bi-closed extension. 
The smallest bi-closed extension of a bi-closable operator $A$ is called the \emph{bi-closure} 
of $A$ and denoted by $\overline{A}^{\operatorname{bi}}$.
\item $A$ is called \emph{bi-densely defined} if $D(A)$ is bi-dense in $X$. A set $M\subseteq X$ is called \emph{bi-dense} if 
for all $x\in X$ there is a $\norm{\cdot}$-bounded sequence $(x_{i})_{i\in\N}\subseteq M$ such that $x_{i}\to x$ w.r.t.~$\topo$.
\item $D$ is called a \emph{bi-core} for $A$ if for all $x\in D(A)$ there is a sequence $(x_{i})_{i\in\N}\subseteq D$ such that 
$(x_i)$, $(Ax_i)$ are $\norm{\cdot}$-bounded and $x_i\to x$ w.r.t.~the topology induced by the system of seminorms given by 
\[
p_{A}(x)\coloneq p(x)+p(Ax)\quad (x\in D(A))
\]
for $p\in\semis_{\topo}$.
\end{enumerate}
\end{definition}

Using the mixed topology $\gamma$, these notions can be rephrased as the sequential notions introduced before 
(cf.~\cite[Remark 3.2]{KruseSeifert2023b}).

\begin{remark}\label{rem:bi_notions_operators}
Let $(X,\norm{\cdot},\topo)$ be a Saks space, $(A,D(A))$ a linear operator on $X$ and $D$ a linear subspace of $D(A)$. 
Due to \cite[I.1.10 Proposition]{Cooper1978}, 
a sequence in $X$ is $\gamma$-convergent if and only if it is $\norm{\cdot}$-bounded and $\topo$-convergent. 
Hence, $A$ is bi-closed (bi-closable, bi-densely defined) if and only if it is sequentially $\gamma$-closed ($\gamma$-closable, $\gamma$-densely defined). 
Moreover, $\overline{A}^{\operatorname{bi}}=\overline{A}^{\gamma\text{-}\operatorname{seq}}$ if $A$ is bi-closable.
Further, a set $M\subseteq X$ is bi-dense if and only if it is sequentially $\gamma$-dense. In combination with $\gamma\subseteq\topo$ this implies that 
$D$ is a bi-core for $A$ if and only if it is a sequential $\gamma$-core.
\end{remark}

\subsection{Generation Theorems}
We will demonstrate how to obtain generation theorems for locally bi-continuous semigroups and cosine families from the well-known Hille--Yosida theorem for Hausdorff locally convex spaces, first observed in \cite[Theorem (Hille--Yosida), p.~126]{Schwartz1958} and \cite[Proposition 6.1]{Komatsu1964}. 
We cite here the version of \cite{Choe1985}.

\begin{proposition}[{\cite[Corollary 4.5]{Choe1985}}]
\label{prop:Hille--Yosida_lcx}
    Let $(X,\topo)$ be a sequentially complete Hausdorff locally convex space and $(A,D(A))$ a linear operator on $X$. Then the following assertions are equivalent.
    \begin{enumerate}[label=\upshape(\alph*), leftmargin=*]
        \item $A$ is the generator of a strongly continuous and exponentially equicontinuous semigroup on $X$.
        \item\label{it:Hille--Yosida_lcx_2} $A$ is densely defined and there is $\alpha\in\R$ such that $(\alpha,\infty)\subseteq\rho_{\topo}(A)$ and $((\lambda-\alpha)^mR(\lambda,A)^m)_{\lambda>\alpha, m\in\N }$ is equicontinuous.
    \end{enumerate}
\end{proposition}

The $\alpha$ in \ref{it:Hille--Yosida_lcx_2} is such that $(\euler^{-\alpha t} S(t))_{t\geq 0}$ is $\topo$-equicontinuous where $(S(t))_{t\geq 0}$ is the semigroup generated by $A$. 
It is known that the generator of a strongly continuous semigroup or cosine family on a sequentially complete space 
is densely defined. However, the known proofs actually reveal that they are even sequentially densely defined.

\begin{proposition}\label{prop:seq_densely_defined}
    Let $(X,\topo)$ be a sequentially complete Hausdorff locally convex space and $(S(t))_{t\geq 0}$ a $\topo$-strongly continuous 
    semigroup or cosine family in $\calL(X,\topo)$ with generator $A$. Then $A$ is sequentially $\topo$-densely defined. 
\end{proposition}
\begin{proof}
This follows from inspection of the proofs of \cite[Proposition 1.3]{Komura1968} if $(S(t))_{t\geq 0}$ is a semigroup, and 
\cite[5.4 Lemma]{Fattorini1969a} if $(S(t))_{t\geq 0}$ is a cosine family.
\end{proof}

Our developed theory gives rise to slightly generalise the Hille--Yosida theorem of Kraiij \cite[Theorem 6.4]{Kraaij2016} 
(and \cite[Corollary 6.5]{Kraaij2016}) from countably convex to countably quasi-convex $\calN_{\topo}$ and we also provide a different (correct) proof, 
see \prettyref{rem:Hille-Yosida_correction}.

\begin{theorem}\label{thm:Hille_Yosida_sg}  
    Let $(X,\norm{\cdot})$ be a normed space, $\topo\subseteq\topo_{\norm{\cdot}}$ a Hausdorff locally convex topology on $X$ 
    such that $(X,\topo)$ is sequentially complete, $\topo$-bounded sets are also $\norm{\cdot}$-bounded and $\calN_{\topo}$ 
    is countably quasi-convex. Let $(A,D(A))$ be a linear operator on $X$, $\omega\in\R$ and $M\geq 1$. Then the following assertions 
    are equivalent.
    \begin{enumerate}[label=\upshape(\alph*), leftmargin=*]
        \item\label{it:Hil_Yos_sg_1} $A$ is the generator of a $\topo$-strongly continuous semigroup $(S(t))_{t\geq 0}$ on $X$ such that for all $t_0\geq 0$ and $p\in\calN_{\topo}$ there exists $q\in\calN_{\topo}$ such that 
        \[
        p(\euler^{-\omega t}S(t)x)\leq Mq(x) \quad(t\in [0,t_0],\, x\in X).
        \]
        \item\label{it:Hil_Yos_sg_2} $A$ is the generator of a $\topo$-strongly continuous semigroup $(S(t))_{t\geq 0}$ on $X$ such that for all $\alpha>\omega$ and $p\in\calN_{\topo}$ there exists $q\in\calN_{\topo}$ such that 
        \[
        p(\euler^{-\alpha t}S(t)x)\leq Mq(x) \quad(t\geq 0,\, x\in X).
        \]
        \item\label{it:Hil_Yos_sg_3} $A$ is $\topo$-densely defined, $(\omega,\infty)\subseteq\rho_{\topo}(A)$ and for all $\alpha>\omega$ and all $p\in\calN_{\topo}$ there exists $q\in\calN_{\topo}$ such that
        \begin{equation}\label{eq:res_estimate_HY_sg}
        p((\lambda-\alpha)^mR(\lambda,A)^mx) \leq Mq(x) \quad(\lambda>\alpha,\, m\in\N,\, x\in X).
        \end{equation}
    \end{enumerate}
\end{theorem}

\begin{proof}
    ``\ref{it:Hil_Yos_sg_1}$\Rightarrow$\ref{it:Hil_Yos_sg_3}'' 
     Note that $(\euler^{-\omega t}S(t))_{t\geq 0}$ is locally $\topo$-equicontinuous with a uniform constant $M$.   
     By \prettyref{cor:type}, for all $\alpha>\omega$ we have that $(\euler^{-\alpha t}S(t))_{t\geq 0}$ is $\topo$-equicontinuous with the same uniform constant $M$.
     Hence, \prettyref{prop:Hille--Yosida_lcx} yields that $A$ is $\topo$-densely defined and $(\alpha,\infty)\subseteq \rho_{\topo}(A)$ for all $\alpha>\omega$. 
     Moreover, for all $\alpha>\omega$, the family
     $((\lambda-\alpha)^mR(\lambda,A)^m)_{\lambda>\alpha, m\in\N}$ is $\topo$-equicontinuous.
     For all $\alpha>\omega$, $\lambda>\alpha$ and $x\in X$ we have
     \[
     R(\lambda,A)x = \int_0^\infty \euler^{-\lambda t} S(t)x\,\d t
     \]
     by (the proof of) \cite[Lemma 4.4]{AlbaneseBonetRicker2013} where the integral is an improper Riemann integral in the sequentially complete space $(X,\topo)$. 
     It also follows from \cite[Lemma 4.4]{AlbaneseBonetRicker2013} that $R(\cdot,A)$ is holomorphic from
     $\C_{\alpha}\coloneq\{\mu\in\C\;|\;\re(\mu)>\alpha\}\subseteq \rho_{\topo}(A)$ to $\calL_{\operatorname{b}}(X,\topo)$ where $\calL_{\operatorname{b}}(X,\topo)$ denotes the space $\calL(X,\topo)$ equipped 
     with the topology of uniform convergence on the $\topo$-bounded subsets of $X$. Furthermore, $\frac{\d^m}{\d \lambda^m} R(\lambda,A)x = (-1)^m m! R(\lambda,A)^{m+1}x$ 
     for all $m\in\N$, $\lambda\in\C_{\alpha}$ and $x\in X$ by \cite[Proposition 3.4 (i)+(iii), Remark 3.5 (iv), Lemma 4.4]{AlbaneseBonetRicker2013}. 
     Now, we have 
     \[
       \frac{\d^m}{\d \lambda^m} R(\lambda,A)x 
     = \int_0^{\infty} (-t)^m\euler^{-\lambda t}S(t)x\,\d t
     =\int_0^{\infty} (-t)^m\euler^{-(\lambda-\alpha) t}\euler^{-\alpha t}S(t)x\,\d t
     \]
     for all $\lambda\in\C_{\alpha}$, $m\in\N_{0}$ and $x\in X$.
     Let $p\in\calN_{\topo}$. By \prettyref{cor:type}, there exists $\widetilde{q}_{\alpha}\in\calN_{\topo}$ such that
     \[
          p\left(\frac{\d^m}{\d \lambda^m} R(\lambda,A)x\right) 
     \leq M\int_0^{\infty} t^m\euler^{-(\lambda-\alpha) t} \,\d t\, \widetilde{q}_{\alpha}(x)
     \]
     for all $\lambda\in\C_{\alpha}$, $m\in\N_{0}$ and $x\in X$.
     Since
     \[
       \int_0^{\infty} t^m\euler^{-(\lambda-\alpha) t} \,\d t  
     = \frac{m!}{(\lambda-\alpha)^{m+1}}
     \]
     for all $\lambda\in\C_{\alpha}$ and $m\in\N_{0}$, we obtain
     \[
      p((\lambda-\alpha)^mR(\lambda,A)^mx) 
    = p\left((-1)^{m-1}\frac{(\lambda-\alpha)^m}{(m-1)!} \frac{\d^{m-1}}{\d \lambda^{m-1}} R(\lambda,A)x\right) 
    \leq M\widetilde{q}_{\alpha}(x)
     \]
     for all $\lambda\in\C_{\alpha}$, $m\in\N$ and $x\in X$.
     
    ``\ref{it:Hil_Yos_sg_3}$\Rightarrow$\ref{it:Hil_Yos_sg_2}''
    By \prettyref{prop:Hille--Yosida_lcx}, $A$ generates a $\topo$-strongly continuous exponentially $\topo$-equicontinuous semigroup $(S(t))_{t\geq 0}$ on $X$. Moreover, $(\euler^{-\alpha t}S(t))_{t\geq 0}$ is $\topo$-equicontinuous for all $\alpha>\omega$.
 
    Let $\alpha>\omega$. First, we consider the case $\alpha=0$ and modify the proof of \cite[Theorem 2.12]{Isem18}.
    For $t>0$, $m\in\N$ and $x\in X$ we have
    \begin{equation}\label{eq:Hil_Yos_eq1}
      \left(I-\frac{t}{m}A\right)^{-m}x 
    = \left(\frac{m}{t}\right)^m\left(\frac{m}{t} - A\right)^{-m}x 
    = \left(\frac{m}{t}\right)^m R\left(\frac{m}{t},A\right)^mx 
    \end{equation}
    where $I\coloneq \id$ is the identity on $X$. Moreover, for $0<s<t$, $m\in\N$ and $x\in D(A)$ we have
    \begin{align*}
        \frac{\d}{\d s} S(t-s)\left(I-\frac{s}{m}A\right)^{-m}x & = S(t-s)\left(-A+A\left(I-\frac{s}{m}A\right)^{-1}\right)\left(I-\frac{s}{m}A\right)^{-m}x \\
        & = S(t-s)\left(I-\frac{s}{m}A\right)^{-m}\left(-I+\left(I-\frac{s}{m}A\right)^{-1}\right)Ax.
    \end{align*}
    Now, let $p\in\calN_{\topo}$. By equicontinuity of $(S(t))_{t\geq 0}$, there exist $q_1\in\semis_{\topo}$ and $K_1\geq 0$ such that $p(S(r)y) \leq K_1q_1(y)$ for all $r\geq 0$ and $y\in X$.
    By assumption and \prettyref{eq:Hil_Yos_eq1}, there exists $q_2\in\semis_{\topo}$ such that
    \[
    q_1\left(\left(I-\frac{r}{m}A\right)^{-m}y\right) \leq M q_2 (y) \quad(r>0,\, m\in\N,\, y\in X).
    \]    
    We note that this estimate is also valid for $r=0$.
    For $t\geq 0$, $m\in\N$ and $x\in D(A)$ we thus have by the fundamental theorem of calculus 
    (see e.g.~\cite[Proposition 11]{AlbaneseBonetRicker2012}) that
    \begin{align*}
            p\left(\left(I-\frac{t}{m}A\right)^{-m}x - S(t)x\right) 
        & = p\left(\int_0^t S(t-s)\left(I-\frac{s}{m}A\right)^{-m}\left(-I+\left(I-\frac{s}{m}A\right)^{-1}\right)Ax\,\d s\right) \\
        & \leq \int_0^{t} p\left(S(t-s)\left(I-\frac{s}{m}A\right)^{-m}\left(-I+\left(I-\frac{s}{m}A\right)^{-1}\right)Ax\right)\,\d s\\
        & \leq  K_1 M \int_0^{t} q_2\left(\left(-I+\left(I-\frac{s}{m}A\right)^{-1}\right) Ax\right)\,\d s.
    \end{align*}
    Now, we remark that $\left(I-\frac{s}{m}A\right)^{-1}Ax \to Ax$ as $m\to\infty$ uniformly for $s\in [0,t]$ for all $t\geq 0$. Indeed, for $y\in D(A)$ we have
    \[
      \left(I-\frac{s}{m}A\right)^{-1}y - y 
    = \frac{s}{m}\left(I-\frac{s}{m}A\right)^{-1}Ay = R\left(\frac{m}{s},A\right)Ay
    \]
    and hence, by assumption and \prettyref{eq:Hil_Yos_eq1}, for $\widetilde{p}\in\semis_\topo$ there exists 
    $\widetilde{q}\in\semis_\topo$ such that
    \[
      \widetilde{p}\left(\left(I-\frac{s}{m}A\right)^{-1}y - y\right) 
    = \widetilde{p}\left(R\left(\frac{m}{s},A\right)Ay\right) 
    \leq  \frac{s}{m} M \widetilde{q}\left(Ay\right) \quad(s>0,\, m\in\N)
    \]
    We observe that this estimate is also valid for $s=0$.
    Thus,
    \[
         \sup_{s\in [0,t]} \widetilde{p}\left(\left(I-\frac{s}{m}A\right)^{-1}y - y\right) 
    \leq \frac{t}{m} M \widetilde{q}\left(Ay\right) \to 0.
    \]
    As $((I-\frac{s}{m}A)^{-1} - I)_{s\geq 0, m\in\N}$ is $\topo$-equicontinuous by assumption and \prettyref{eq:Hil_Yos_eq1}, and $D(A)$ is $\topo$-dense, we obtain for all $y\in X$ that $(I-\frac{s}{m}A)^{-1}y \to y$ uniformly for $s\in [0,t]$.
    
    Hence, we arrive at $(I-\frac{t}{m}A)^{-m}x \to S(t)x$ as $m\to\infty$.
    Since $((I-\frac{t}{m}A)^{-m})_{t\geq 0, m\in\N}$ is $\topo$-equicontinuous by assumption and \prettyref{eq:Hil_Yos_eq1}, $S(t)$ is $\topo$-continuous and $D(A)$ is $\topo$-dense, we observe that
    $(I-\frac{t}{m}A)^{-m}x \to S(t)x$ for all $t\geq 0$ and $x\in X$.
    
    By assumption and \prettyref{eq:Hil_Yos_eq1}, for $p\in \calN_{\topo}$ there exists $q\in \calN_{\topo}$ such that
    \[
    p\left(\left(I-\frac{t}{m}A\right)^{-m}x\right)\leq Mq(x) \quad(t \geq 0,\, m\in\N,\, x\in X).
    \]
    Therefore, $p(S(t)x)\leq M q(x)$ for all $t\geq 0$ and $x\in X$.

    For general $\alpha>\omega$ we shift and consider $\widetilde{A}\coloneq A-\omega$. By assumption, for all $\alpha>\omega$ and $p\in\calN_\topo$ there exists $q\in\calN_\topo$ such that, with 
    $\mu\coloneq\lambda+\alpha>\alpha$ for $\lambda>0$,
    \[
      p(\lambda^m R(\lambda,\widetilde{A})^m x) 
    = p(\lambda^m R(\lambda+\alpha,A)^m x)
    = p((\mu-\alpha)^m R(\mu,A)^m x) 
    \leq Mq(x)
    \]
    for all $\lambda>0$, $m\in\N$ and $x\in X$. Hence, by the first part, $\widetilde{A}$ generates a $\topo$-strongly continuous semigroup $(\widetilde{S}(t))_{t\geq0}$ on $X$ such that for all $p\in \calN_\topo$ there exists $q\in\calN_\topo$ such that 
    \[
    p(\widetilde{S}(t)x) \leq Mq(x) \quad(t\geq 0,\, x\in X).
    \]
    Therefore, by rescaling, $A = \widetilde{A}+\alpha$ also generates the $\topo$-strongly continuous semigroup $(S(t))_{t\geq0} \coloneq (\euler^{\alpha t}\widetilde{S}(t))_{t\geq 0}$ on $X$ such that for all $p\in \calN_\topo$ 
    there exists $q\in\calN_\topo$ such that 
    \[
    p(\euler^{-\alpha t}S(t)x) = p(\widetilde{S}(t)x) \leq Mq(x) \quad(t\geq 0,\, x\in X).
    \]
    
    ``\ref{it:Hil_Yos_sg_2}$\Rightarrow$\ref{it:Hil_Yos_sg_1}'' This follows from \prettyref{cor:type}.     
\end{proof}

\begin{remark}\label{rem:Hille-Yosida_correction}
     Comparing \prettyref{thm:Hille_Yosida_sg} with \cite[Theorem 6.4]{Kraaij2016}, we note that instead of \eqref{eq:res_estimate_HY_sg} 
     the latter has that for all $\alpha>\omega$ and $p\in\calN_{\topo}$ there exists $q\in\calN_{\topo}$ such that
     \[
     p(m(\lambda-\omega)^m R(m\lambda,A)^mx) \leq Mq(x) \quad(\lambda\geq \alpha,\, m\in\N,\, x\in X).
     \]
     First, we note that in (the proof of) \cite[Theorem 6.4]{Kraaij2016} the rescaling argument is not worked out correctly and the inequality 
     above should actually be 
     \[
     p((m\lambda-\omega)^m R(m\lambda,A)^mx) \leq Mq(x) \quad(\lambda\geq \alpha, m\in\N, x\in X).
     \]
     Furthermore, even this corrected inequality is problematic if $\omega<0$. If $\lambda\in\R$ is such that $\omega<\alpha\leq\lambda<0$, then 
     it might happen that there is $m\in\N$ such that $m\lambda\notin\rho_{\topo}(A)$, 
     see e.g.~\cite[Chap.~I, 4.2 Proposition (iv), 4.3 Definition]{EngelNagel2000}. 
     Looking at the proof of \cite[Theorem 6.4]{Kraaij2016}, it seems that (even with the correct rescaling), 
     the statements of \cite[Theorem 6.4]{Kraaij2016} are still only valid for $\omega=0$. 
     For this reason, we provide \prettyref{thm:Hille_Yosida_sg} as a corrected version of it and its consequence 
     \cite[Corollary 6.5]{Kraaij2016}. 
     Moreover, \cite[Proposition 4.8 (c)]{KruseSchwenninger2024} and \cite[Theorem 3.17 (a)]{KruseSchwenninger2022} use 
     \cite[Theorem 6.4]{Kraaij2016} and \cite[Corollary 6.5]{Kraaij2016}, respectively, in their proofs and 
     they remain valid by replacing the latter results by \prettyref{thm:Hille_Yosida_sg}.
\end{remark}

We can also recover the Hille--Yosida theorem for bi-continuous semigroups, see \cite[Theorem 1.28]{Kuehnemund2001}, 
\cite[Theorem 16]{Kuehnemund2003} and \cite[Theorem 1.2.8]{Farkas2003}, in the case where $(X,\gamma)$ is C-sequential. 
Let $(X,\norm{\cdot},\topo)$ be a Saks space and $(S(t))_{t\geq 0}$ a locally $\topo$-bi-continuous semigroup 
on $X$. According to \cite[Definition 1.2.6]{Farkas2003}, its \emph{bi-generator} is the linear operator defined by 
    \begin{align*}
    D(A_{\norm{\cdot},\topo})&\coloneq\bigl\{x\in X \;|\; \topo\text{-}\lim_{t\to 0\rlim} \tfrac{1}{t} (S(t)x-x)\text{ exists and }\sup_{t\in(0,1]}\tfrac{1}{t}\norm{S(t)x-x}<\infty \bigr\},\\
    A_{\norm{\cdot},\topo}x &\coloneq \topo\text{-}\lim_{t\to 0\rlim} \tfrac{1}{t} (S(t)x-x) \quad(x\in  D(A_{\norm{\cdot},\topo})). 
    \end{align*}
    Due to \prettyref{prop:tau_gamma_strong}, $(S(t))_{t\geq 0}$ is also $\gamma$-strongly continuous. Let $A$ with domain $D(A)$ 
    denote its generator when considered as a $\gamma$-strongly continuous semigroup. Then we have
    \begin{equation}\label{eq:bi_gen_sg}
     D(A_{\norm{\cdot},\topo})=D(A)\quad\text{and}\quad  A_{\norm{\cdot},\topo}=A
    \end{equation}
    by \cite[I.1.10 Proposition]{Cooper1978} (cf.~\cite[5.1 Proposition]{Kruse2024b}).

\begin{corollary}[{cf.~\cite[Theorem 1.28]{Kuehnemund2001}}]
\label{cor:Hille--Yosida_bi-continuous}
    Let $(X,\norm{\cdot},\topo)$ be a sequentially complete C-sequential Saks space.
    Let $(A,D(A))$ be a linear operator on $X$, $\omega\in\R$ and $M\geq 1$. Then the following assertions 
    are equivalent.
    \begin{enumerate}[label=\upshape(\alph*), leftmargin=*]
        \item\label{it:Hil_Yos_bicont_1} $A$ is the bi-generator of a locally $\topo$-bi-continuous and exponentially bounded semigroup $(S(t))_{t\geq 0}$ on $X$ such that 
        \[
        \norm{S(t)}_{\calL(X)}\leq M\euler^{\omega t}\quad(t\geq 0).
        \]
        \item\label{it:Hil_Yos_bicont_2} $A$ is bi-closed, bi-densely defined, $(\omega,\infty)\subseteq\rho_{\norm{\cdot}}(A)$ with
            \[
            \sup_{\lambda>\omega, m\in\N} \norm{(\lambda-\omega)^m R(\lambda,A)^m}_{\calL(X)} \leq M
            \]
            and the family $\bigl((s-\alpha)^m R(s,A)^m\bigr)_{s> \alpha,\, m\in\N}$ 
            is bi-equicontinuous for all $\alpha>\omega$.
    \end{enumerate}
\end{corollary}

\begin{proof}
    We recall that $(X,\norm{\cdot},\topo)$ being a sequentially complete C-sequential Saks space implies that 
    $(X,\gamma)$ is sequentially complete and C-sequential (see \prettyref{defi:seq_compl_C-seq_Saks}).

    ``\ref{it:Hil_Yos_bicont_1}$\Rightarrow$\ref{it:Hil_Yos_bicont_2}''
    Since $(X,\gamma)$ is sequentially complete and C-sequential, the locally $\topo$-bi-continuous and exponentially bounded semigroup 
    $(S(t))_{t\geq 0}$ generated by $A$ is $\gamma$-strongly continuous and exponentially $\gamma$-equi\-con\-tin\-u\-ous
    by \prettyref{rem:Saks_seq_complete_Banach} \ref{it:Saks_seq:compl_Banach_1} and \prettyref{cor:bi_cont_quasi_equicont_C_seq_sg_cosine} \ref{it:bi_cont_sg_cosine_rom1}. 
    It follows from \eqref{eq:bi_gen_sg}, \prettyref{prop:seq_densely_defined} and \cite[Proposition 1.4]{Komura1968} that $A$ is sequentially 
    $\gamma$-densely defined and $\gamma$-closed, implying that $A$ is bi-densely defined and bi-closed 
    by \prettyref{rem:bi_notions_operators}.
    Further, the local $\gamma$-equicontinuity of $(S(t))_{t\geq 0}$ and $\norm{S(t)}_{\calL(X)}\leq M\euler^{\omega t}$ for all $t\geq 0$ 
    yield that for all $t_0\geq 0$ and $p\in\calN_{\gamma}$ there exists $q\in\calN_{\gamma}$ such that 
    \[
    p(\euler^{-\omega t}S(t)x)\leq Mq(x) \quad(t\in [0,t_0],\, x\in X)
    \]
    by \prettyref{rem:count_conv} \ref{it:count_conv_1}, \prettyref{cor:type} and \prettyref{rem:Saks_bounded_sets} \ref{it:Saks_bounded_sets_2}.
    Due to \prettyref{rem:count_conv} \ref{it:count_conv_1}, \prettyref{rem:Saks_bounded_sets} \ref{it:Saks_bounded_sets_2} 
    and \prettyref{thm:Hille_Yosida_sg}, we have that $(\omega,\infty)\subseteq\rho_{\gamma}(A)$ and for all $\alpha>\omega$ 
    and all $p\in\calN_{\gamma}$ there exists $q\in\calN_{\gamma}$ such that
    \[
    p((s-\alpha)^m R(s,A)^m x) \leq Mq(x) \quad(s>\alpha,\, m\in\N,\, x\in X).
    \] 
    In particular, this means that $\bigl((s-\alpha)^m R(s,A)^m\bigr)_{s> \alpha,\, m\in\N}$  
    is $\gamma$-equicontinuous and therefore bi-equiconti\-nuous by \prettyref{prop:bi_equicont_mixed_equicont} \ref{it:bi_equicont_mixed_3}. Moreover, $\rho_{\gamma}(A)\subseteq \rho_{\norm{\cdot}}(A)$ 
    by \prettyref{rem:gamma_seq_cont_norm_cont} and so $(\omega,\infty)\subseteq\rho_{\norm{\cdot}}(A)$.
    Furthermore, using that $q(x)\leq\norm{x}$ for all $x\in X$, and taking the supremum over all $p\in\calN_{\gamma}$, 
    we obtain
    \[
    \norm{(s-\alpha)^m R(s,A)^m x} \leq M\norm{x} \quad(s>\alpha,\, m\in\N,\, x\in X)
    \]
    for all $\alpha>\omega$ by \prettyref{rem:Saks_bounded_sets} \ref{it:Saks_bounded_sets_1}. 
    Now, letting $\alpha\to\omega\rlim$, we get that 
    \[
    \norm{(s-\omega)^m R(s,A)^m x} \leq M\norm{x} \quad(s>\omega,\, m\in\N,\, x\in X), 
    \]
    implying 
    \[
    \sup_{\lambda>\omega, m\in\N} \norm{(\lambda-\omega)^m R(\lambda,A)^m}_{\calL(X)} \leq M.
    \]

    ``\ref{it:Hil_Yos_bicont_2}$\Rightarrow$\ref{it:Hil_Yos_bicont_1}''
    Since $A$ bi-densely defined, it is sequentially $\gamma$-densely defined by \prettyref{rem:bi_notions_operators} 
    and so $\gamma$-densely defined. 
    By \prettyref{prop:bi_equicont_mixed_equicont} \ref{it:bi_equicont_mixed_3}, for all $\alpha>\omega$ 
    we have that $((s-\alpha)^m R(s,A)^m)_{s> \alpha,\, m\in\N}$ is $\gamma$-equicontinuous. 
    In particular, it holds that $(\omega,\infty)\subseteq\rho_{\gamma}(A)$. Noting that $0<s-\alpha<s-\omega$ for all 
    $s>\alpha$, we have 
    \[
         \sup_{s> \alpha,\, m\in\N}\norm{(s-\alpha)^m R(s,A)^m}_{\calL(X)}
    \leq \sup_{s> \omega,\, m\in\N}\norm{(s-\omega)^m R(s,A)^m }_{\calL(X)}\leq M
    \]
    for all $\alpha>\omega$. Hence, by \cite[Lemma 4.5]{Kraaij2016} and \prettyref{rem:Saks_bounded_sets} \ref{it:Saks_bounded_sets_2}, for all $\alpha>\omega$ and all $p\in\calN_{\gamma}$ there exists $q\in\calN_{\gamma}$ such that
    \[
    p((\lambda-\alpha)^m R(\lambda,A)^m x) \leq Mq(x) \quad(\lambda>\alpha,\, m\in\N,\, x\in X).
    \]  
    By \prettyref{rem:count_conv} \ref{it:count_conv_1}, \prettyref{rem:Saks_bounded_sets} \ref{it:Saks_bounded_sets_2} 
    and \prettyref{thm:Hille_Yosida_sg}, this implies that $A$ is the generator of a $\gamma$-strongly continuous semigroup $(S(t))_{t\geq 0}$ on $X$ such that for all $t_0\geq 0$ and $p\in\calN_{\gamma}$ there exists $q\in\calN_{\gamma}$ such that 
    \[
    p(\euler^{-\omega t}S(t)x)\leq Mq(x) \quad(t\in [0,t_0],\, x\in X).
    \]
    In particular, $(S(t))_{t\geq 0}$ is locally $\gamma$-equicontinuous, so locally bi-equicontinuous by \prettyref{prop:bi_equicont_mixed_equicont} \ref{it:bi_equicont_mixed_3}, and $\norm{S(t)}_{\calL(X)}\leq M\euler^{\omega t}$ for all $t\geq 0$ by 
    \prettyref{rem:count_conv} \ref{it:count_conv_1}, \prettyref{cor:type} and \prettyref{rem:Saks_bounded_sets} \ref{it:Saks_bounded_sets_2}. 
    We conclude that $(S(t))_{t\geq 0}$ is locally $\topo$-bi-continuous and exponentially bounded. 
\end{proof}

\begin{remark}\label{rem:additional_assump}
    Due to \prettyref{rem:Saks_bi-admissible}, the only additional assumption 
    in \prettyref{cor:Hille--Yosida_bi-continuous} compared to \cite[Theorem 1.28]{Kuehnemund2001} is 
    that $(X,\norm{\cdot},\topo)$ is C-sequential. This additional assumption is mild and fulfilled in all the interesting cases where bi-continuous semigroups appeared in applications, see \cite[Section 4]{KruseSchwenninger2022}.
\end{remark}

We can also recover the generation theorem for bi-continuous cosine families, see \cite[Theorem, p.~58]{DuanSun2006} and \cite[Theorem 4.6]{Budde2024}, from the corresponding one in 
Hausdorff locally convex spaces, in a similar manner as for semigroups.

\begin{proposition}[{\cite[Chapter 1, Theorem 4.4]{XiaoLiang1998}}]
\label{prop:generation_cosine_lcx}
    Let $(X,\topo)$ be a sequentially complete Hausdorff locally convex space, $(A,D(A))$ a linear operator on $X$ and $\omega\geq 0$. 
    Then the following assertions are equivalent.
    \begin{enumerate}[label=\upshape(\alph*), leftmargin=*]
        \item $A$ is the generator of a strongly continuous cosine family $(S(t))_{t\geq 0}$ on $X$ such that $(\euler^{-\omega t}S(t))_{t\geq 0}$ is equicontinuous.
        \item $A$ is densely defined and closed, $(\omega^2,\infty)\in\rho_{\topo}(A)$ and $\left(\frac{(\lambda-\omega)^{m+1}}{m!} \frac{\d^m}{\d\lambda^m} \lambda R(\lambda^2,A)\right)_{\lambda>\omega, m\in\N_0}$ is equicontinuous.
    \end{enumerate}
\end{proposition}

\begin{remark}
Note that \cite[Chapter 1, Theorem 4.4]{XiaoLiang1998} states
\begin{enumerate}
    \item[(b')] $A$ is densely defined and closed, and there is $a\geq\omega$ such that $(a^2,\infty)\in\rho_{\topo}(A)$ and\\ $\left(\frac{(\lambda-\omega)^{m+1}}{m!} \frac{\d^m}{\d\lambda^m} \lambda R(\lambda^2,A)\right)_{\lambda>a, m\in\N_0}$ is equicontinuous.
\end{enumerate}
instead of (b). Inspecting the proof, it relies on a version of Post--Widder's theorem \cite[Chapter 1, Theorem 2.1]{XiaoLiang1998} in which $a=\omega$ can be chosen. See also \cite[Theorem 3.1]{Fattorini1969b} for the case of $(X,\topo)$ being complete and barrelled.
\end{remark}

Let $(X,\norm{\cdot},\topo)$ be a Saks space and $(S(t))_{t\geq 0}$ a locally $\topo$-bi-continuous cosine family on $X$.
According to \cite[Definition 4]{CangChenSong2014}, its \emph{bi-generator} is the linear operator defined by 
    \begin{align*}
    D(A_{\norm{\cdot},\topo})&\coloneq\bigl\{x\in X \;|\; \topo\text{-}\lim_{t\to 0\rlim} \tfrac{2}{t^2} (S(t)x-x)\text{ exists and }\sup_{t\in(0,1]}\tfrac{2}{t^2}\norm{S(t)x-x}<\infty \bigr\},\\
    A_{\norm{\cdot},\topo}x &\coloneq \topo\text{-}\lim_{t\to 0\rlim} \tfrac{2}{t^2} (S(t)x-x) \quad(x\in  D(A_{\norm{\cdot},\topo})). 
    \end{align*}
Due to \prettyref{prop:tau_gamma_strong}, $(S(t))_{t\geq 0}$ is also $\gamma$-strongly continuous. Let $A$ with domain $D(A)$ 
denote its generator when considered as a $\gamma$-strongly continuous cosine family. Then we have
    \begin{equation}\label{eq:bi_gen_cs}
     D(A_{\norm{\cdot},\topo})=D(A)\quad\text{and}\quad  A_{\norm{\cdot},\topo}=A
    \end{equation}
by \cite[I.1.10 Proposition]{Cooper1978} as in the case of semigroups.

\begin{proposition}[{cf.~\cite[Theorem, p.~58]{DuanSun2006}}]
\label{prop:generation_cosine_bi-continuous}
Let $(X,\norm{\cdot},\topo)$ be a sequentially complete C-sequential Saks space.
    Let  $(A,D(A))$ be a linear operator on $X$, $\omega\geq 0$ and  $M\geq 1$. Then the following assertions are equivalent.
    \begin{enumerate}[label=\upshape(\alph*), leftmargin=*]
        \item\label{it:Hil_Yos_cos_bicont_1} $A$ is the bi-generator of a locally $\topo$-bi-continuous and exponentially bounded cosine family $(S(t))_{t\geq 0}$ on $X$ such that
        \[
        \norm{S(t)}_{\calL(X)} \leq M\euler^{\omega t} \quad (t\geq 0).
        \]
        \item\label{it:Hil_Yos_cos_bicont_2}
        $A$ is bi-densely defined, bi-closed, $(\omega^2,\infty)\in\rho_{\norm{\cdot}}(A)$ and 
            \begin{equation}\label{eq:generation_cosine_bi-cont}
            \sup_{\lambda>\omega, m\in\N_0} \norm{\frac{(\lambda-\omega)^{m+1}}{m!} \frac{\d^m}{\d\lambda^m} \lambda R(\lambda^2,A)}_{\calL(X)} \leq M
            \end{equation}
            and the family $\bigl(\frac{(\lambda-\alpha)^{m+1}}{m!} \frac{\d^m}{\d\lambda^m} \lambda R(\lambda^2,A)\bigr)_{\lambda> \alpha,\, m\in\N_0}$ is bi-equicontinuous for all $\alpha>\omega$. 
    \end{enumerate}
\end{proposition}

\begin{proof}
    By \prettyref{defi:seq_compl_C-seq_Saks}, $(X,\gamma)$ is sequentially complete and C-sequential.
    
    ``\ref{it:Hil_Yos_cos_bicont_1}$\Rightarrow$\ref{it:Hil_Yos_cos_bicont_2}''
    %Let $(S(t))_{t\geq 0}$ be the locally $\topo$-bi-continuous cosine family generated by $A$ and $\omega\in\R$ such that $(\euler^{-\omega t} S(t))_{t\geq 0}$ is bounded. 
    Note that $(S(t))_{t\geq 0}$ is $\gamma$-strongly continuous and $(S(t))_{t\geq 0}$ is exponentially $\gamma$-equicontinuous 
    by \prettyref{rem:Saks_seq_complete_Banach} \ref{it:Saks_seq:compl_Banach_1} and \prettyref{cor:bi_cont_quasi_equicont_C_seq_sg_cosine} \ref{it:bi_cont_sg_cosine_rom1}. 
    By \prettyref{rem:count_conv} \ref{it:count_conv_1}, \prettyref{cor:type}, \prettyref{rem:Saks_bounded_sets} \ref{it:Saks_bounded_sets_2}, 
    the local $\gamma$-equicontinuity of $(S(t))_{t\geq 0}$ and the inequality $\norm{S(t)}_{\calL(X)} \leq M\euler^{\omega t}$ for all $t\geq 0$, 
    we have that for all $\alpha>\omega$ and $p\in\calN_{\gamma}$ there exists $\widetilde{q}_{\alpha}\in\calN_{\gamma}$ such that 
    \[
    p(\euler^{-\alpha t}S(t)x)\leq M\widetilde{q}_{\alpha}(x) \quad(t\geq 0,\, x\in X).
    \]    
    By \eqref{eq:bi_gen_cs}, \prettyref{prop:seq_densely_defined} and \prettyref{prop:generation_cosine_lcx}, we obtain that 
    $A$ is $\gamma$-closed and sequentially $\gamma$-densely defined, 
    so $A$ is bi-closed and bi-dense by \prettyref{rem:bi_notions_operators}.
    Moreover, for all $\alpha>\omega$ we have $(\alpha^2,\infty)\subseteq \rho_{\gamma}(A)$ and therefore $(\omega^2,\infty)\subseteq\rho_{\gamma}(A)\subseteq\rho_{\norm{\cdot}}(A)$ where the second inclusion follows from \prettyref{rem:gamma_seq_cont_norm_cont}. 
    Further, for all $\alpha>\omega$ we have that $\bigl(\frac{(\lambda-\alpha)^{m+1}}{m!} \frac{\d^m}{\d\lambda^m} \lambda R(\lambda^2,A)\bigr)_{\lambda> \alpha,\, m\in\N_0}$ is $\gamma$-equicontinuous by \prettyref{prop:generation_cosine_lcx}, thus also bi-equicontinuous by \prettyref{prop:bi_equicont_mixed_equicont}\ref{it:bi_equicont_mixed_3}. 
    It remains to show the boundedness of $\bigl(\frac{(\lambda-\omega)^{m+1}}{m!} \frac{\d^m}{\d\lambda^m} \lambda R(\lambda^2,A)\bigr)_{\lambda> \omega,\, m\in\N_0}$ 
    by $M$ in $\norm{\cdot}_{\calL(X)}$. Let $p\in\calN_{\gamma}$. Using \eqref{eq:resol_laplace_cosine}, we obtain as in the proof of \prettyref{thm:Hille_Yosida_sg} \ref{it:Hil_Yos_sg_1}$\Rightarrow$\ref{it:Hil_Yos_sg_3} that
     \[
          p\left(\frac{\d^m}{\d \lambda^m} \lambda R(\lambda^2,A)x\right) 
     \leq M\int_0^{\infty} t^m\euler^{-(\lambda-\alpha) t} \,\d t\, \widetilde{q}_{\alpha}(x)
     \]
     for all $\lambda\in\C_{\alpha}$, $m\in\N_{0}$ and $x\in X$ and
     \begin{equation}\label{eq:resolv_estimate_cosine}
     p\left(\frac{(\lambda-\alpha)^{m+1}}{m!} \frac{\d^{m}}{\d \lambda^{m}} \lambda R(\lambda^2,A)x\right) 
    \leq M\widetilde{q}_{\alpha}(x)
     \end{equation}
     for all $\lambda\in\C_{\alpha}$, $m\in\N_0$ and $x\in X$. From here we get as in the proof of 
     \prettyref{cor:Hille--Yosida_bi-continuous} \ref{it:Hil_Yos_bicont_1}$\Rightarrow$\ref{it:Hil_Yos_bicont_2} that 
     \[
        \sup_{\lambda>\omega, m\in\N_0} \norm{\frac{(\lambda-\omega)^{m+1}}{m!} \frac{\d^m}{\d\lambda^m} \lambda R(\lambda^2,A)}_{\calL(X)} \leq M.
     \]
     
    ``\ref{it:Hil_Yos_cos_bicont_2}$\Rightarrow$\ref{it:Hil_Yos_cos_bicont_1}''
    Since $A$ is bi-densely defined, $A$ is sequentially $\gamma$-densely defined by \prettyref{rem:bi_notions_operators}, 
    so $\gamma$-densely defined, too. For all $\alpha>\omega$ we have that $\bigl(\frac{(\lambda-\alpha)^{m+1}}{m!} \frac{\d^m}{\d\lambda^m} \lambda R(\lambda^2,A)\bigr)_{\lambda> \alpha,\, m\in\N_0}$ is $\gamma$-equicontinuous by \prettyref{prop:bi_equicont_mixed_equicont}\ref{it:bi_equicont_mixed_3}. 
    Thus, \prettyref{prop:generation_cosine_lcx} yields that $A$ generates an exponentially $\gamma$-equicontinuous $\gamma$-strongly continuous cosine family $(S(t))_{t\geq 0}$. By \prettyref{rem:Saks_seq_complete_Banach} \ref{it:Saks_seq:compl_Banach_1} and 
    \prettyref{cor:bi_cont_quasi_equicont_C_seq_sg_cosine} \ref{it:bi_cont_sg_cosine_rom1}, $(S(t))_{t\geq 0}$ is locally $\topo$-bi-continuous and exponentially bounded with bi-generator $A$ by \eqref{eq:bi_gen_cs}. 
    It remains to show that $(\euler^{-\omega t} S(t))_{t\geq 0}$ is bounded by $M$ in $\norm{\cdot}_{\calL(X)}$. 
    From \eqref{eq:generation_cosine_bi-cont} it follows as in the proof of \prettyref{cor:Hille--Yosida_bi-continuous} \ref{it:Hil_Yos_bicont_2}$\Rightarrow$\ref{it:Hil_Yos_bicont_1} that for all $\alpha>\omega$ and all $p\in\calN_{\gamma}$ there exists $q_{p}\in\calN_{\gamma}$ such that
    \[
     p\left(\frac{(\lambda-\alpha)^{m+1}}{m!} \frac{\d^{m}}{\d \lambda^{m}} \lambda R(\lambda^2,A)x\right) 
    \leq Mq_{p}(x) \quad (\lambda>\alpha,\,m\in\N_0 ,\,x\in X)
    \]  
    For $x\in X$ let $f_{x}\colon (\alpha,\infty)\to (X,\gamma)$, $f_x(\lambda)\coloneq \lambda R(\lambda^2,A)x$, and $M_p\coloneq Mq_{p}(x)$. 
    Since $f_x$ is infinitely differentiable, it follows from Post--Widder's theorem \cite[Chapter 1, Theorem 2.1]{XiaoLiang1998} that there is 
    $F_{1,x}\colon [0,\infty)\to X$ such that $F_{1,x}(0)=0$ and 
    \[
    \lambda R(\lambda^2,A)x=f_{x}(\lambda)=\lambda\int_{0}^{\infty}\euler^{-\lambda t}F_{1,x}(t) \,\d t \quad (\lambda>\alpha ),
    \]
    where the integral above is an improper Riemann integral in $(X,\gamma)$, as well as 
    \begin{equation}\label{eq:Post_Widder_Lipschitz}
         p(F_{1,x}(t+h)-F_{1,x}(t))
    \leq M_p \euler^{\alpha t} \max\{\euler^{\alpha h},\,1\}h
    = Mq_{p}(x)\euler^{\alpha (t+h)}h \quad (t, h\geq 0).
    \end{equation}
    For $t\geq 0$ let $\operatorname{Sin}(t)\colon X\to X$, $\operatorname{Sin}(t)x\coloneq \int_{0}^{t}S(r)x\,\d r$, where the integral is a 
    Riemann integral in $(X,\gamma)$. Then $\operatorname{Sin}(\cdot)x\colon [0,\infty)\to (X,\gamma)$ is differentiable with 
    $\frac{\d}{\d t}\operatorname{Sin}(t)x=S(t)x$ for all $t\geq 0$ and $x\in X$ and 
    \[
      \int_{0}^{\infty}\euler^{-\lambda t}\operatorname{Sin}(t)x \,\d t
    = \frac{1}{\lambda}\int_{0}^{\infty}\euler^{-\lambda t}S(t)x \,\d t 
    = R(\lambda^2,A)x \quad (\lambda>\alpha,\,x\in X)
    \]
    by integration by parts. Hence, we obtain
    \[
      \int_{0}^{\infty}\euler^{-\lambda t}\operatorname{Sin}(t)x \,\d t
    = \int_{0}^{\infty}\euler^{-\lambda t}F_{1,x}(t) \,\d t \quad (\lambda>\alpha,\,x\in X),
    \]
    which implies $\operatorname{Sin}(t)x=F_{1,x}(t)$ for all $t\geq 0$ and $x\in X$ by the uniqueness of the Laplace transform, 
    see \cite[Chapter 1, Theorem 1.6]{XiaoLiang1998}. Due to \eqref{eq:Post_Widder_Lipschitz}, we thus have 
    \[
         p\Bigl(\tfrac{1}{h}\bigl(\operatorname{Sin}(t+h)x-\operatorname{Sin}(t)x\bigr)\Bigr)
    \leq Mq_{p}(x)\euler^{\alpha (t+h)} \quad (t\geq 0,\,h>0,\,x\in X)
    \]
    Letting $h\to 0\rlim$ and using $\frac{\d}{\d t}\operatorname{Sin}(t)x=S(t)x$ for all $t\geq 0$ and $x\in X$, we get 
    \[
         p(S(t)x)
    \leq Mq_{p}(x)\euler^{\alpha t} \quad (t\geq 0,\,x\in X)
    \]
    By \prettyref{rem:count_conv} \ref{it:count_conv_1}, \prettyref{cor:type} and 
    \prettyref{rem:Saks_bounded_sets} \ref{it:Saks_bounded_sets_2}, this yields that 
    $\norm{S(t)}_{\calL(X)} \leq M\euler^{\omega t}$ for all $t\geq 0$.
\end{proof}

\begin{remark}
    We note that the technique to derive the norm bound for the family $(S(t))_{t\geq 0}$ in \prettyref{prop:generation_cosine_bi-continuous} by means of the Post--Widder theorem can also be used in case of semigroups as in \prettyref{cor:Hille--Yosida_bi-continuous}. However, since we already have \prettyref{thm:Hille_Yosida_sg} for semigroups at our disposal, our given proof seems to be easier. Of course, this raises the question whether a theorem analogous to \prettyref{thm:Hille_Yosida_sg} for cosine families instead of semigroups can be shown. We postpone this question to future work.
\end{remark}

Comparing \prettyref{prop:generation_cosine_bi-continuous} with \cite[Theorem, p.~58]{DuanSun2006} and \cite[Theorem 4.6]{Budde2024}, 
we remark again that the only additional assumption in \prettyref{prop:generation_cosine_bi-continuous} is that $(X,\norm{\cdot},\topo)$ is C-sequential, 
see \prettyref{rem:additional_assump}. This additional condition is fulfilled for the example in \cite[p.~56]{DuanSun2006} and \cite[Example 1.7]{Budde2024} 
(which is the same example) by \cite[Remark 3.19 (a)]{KruseSchwenninger2022}. 

\begin{remark}
    Let $(X,\norm{\cdot},\topo)$ be a Saks space. 
    We note that in the definition of the domain $D(A_{\norm{\cdot},\topo})$ of the bi-generator $A_{\norm{\cdot},\topo}$ of a locally 
    $\topo$-bi-continuous cosine family $(S(t))_{t\geq 0}$ on $X$ given in \cite[Defi\-nition 4]{DuanSun2006}, \cite[Definition 2.1]{Budde2024} and \cite[Definition 2]{Budde2025}, the condition 
    $\sup_{t\in(0,1]}\frac{2}{t^2}\norm{S(t)x-x}<\infty$ is missing. 
    However, this condition is needed (and implicitly used), for example, to prove that 
    $S(t_0)Ax=\topo\text{-}\lim_{t\to 0\rlim} \frac{2}{t^2}S(t_0)(S(t)x-x)$ for all $t_0\geq 0$ and 
    $x\in D(A_{\norm{\cdot},\topo})$ in the proof of \cite[Lemma 2.2]{Budde2024}, 
    which follows from local bi-equicontinuity if $\sup_{t\in(0,1]}\frac{2}{t^2}\norm{S(t)x-x}<\infty$. 
    Such a mistake is also present in the literature on locally $\topo$-bi-continuous semigroups, see \cite[Equation (1.11)]{Kuehnemund2001} and \cite[Equation (9)]{Kuehnemund2003}, and was corrected in \cite[Definition 1.2.6]{Farkas2003} in this case. 
    This correction is needed to prove e.g.~\cite[Proposition 1.16 (a)]{Kuehnemund2001} which is the counterpart of 
    \cite[Lemma 2.2]{Budde2024} for semigroups. 
\end{remark}

\subsection{Approximation Theorems}

We now focus on Trotter--Kato approximation theorems for semigroups and cosine families.
We start with semigroups and then we proceed to cosine families but first we recall the following notions.

\begin{definition}\label{defi:unif_loc_equicont}
Let $X$ be a Hausdorff locally convex space, $I$ a Hausdorff topological space 
and $(S_{n}(t))_{n\in\N, t\in I}$ a family of linear maps $S_n(t)\colon X \to X$. 
Then $(S_{n}(t))_{n\in\N, t\in I}$  is called \emph{(sequentially) uniformly locally equicontinuous} if 
$(S_{n}(t))_{(n,t)\in\N\times M}$ is (sequentially) equicontinuous for all compact sets $M\subseteq I$.
\end{definition}

\prettyref{defi:unif_loc_equicont} generalises \cite[Definition 11]{AlbaneseKuehnemund2002} where $I=[0,\infty)$ and the $S_n(t)$ are 
locally equicontinuous semigroups. 

\begin{definition}\label{defi:unif_loc_bi-equicont}
Let $(X,\norm{\cdot})$ be a normed space, $\topo\subseteq\topo_{\norm{\cdot}}$ a Hausdorff locally convex topology on $X$, 
$I$ a Hausdorff topological space and $(S_{n}(t))_{n\in\N, t\in I}$ a family of linear maps $S_n(t)\colon X \to X$. 
Then $(S_{n}(t))_{n\in\N, t\in I}$ is called \emph{uniformly locally bi-equicontinuous} if 
$(S_{n}(t))_{(n,t)\in\N\times M}$ is bi-equicontinuous for all compact sets $M\subseteq I$.
\end{definition}

\prettyref{defi:unif_loc_bi-equicont} generalises \cite[Definition 2.1 (ii)]{Kuehnemund2001} and 
\cite[Definition 3.1 (2)]{AlbaneseMangino2004} where $I=[0,\infty)$ and the $S_n(t)$ 
are locally $\topo$-bi-continuous semigroups. Further, it corrects \cite[Definition 3 (ii)]{Budde2025} in the case of cosine families.

\begin{remark}\label{rem:unif_loc_bi-equicont_gamma_seq}
    Let $(X,\norm{\cdot})$ be a normed space, $\topo\subseteq\topo_{\norm{\cdot}}$ a Hausdorff locally convex topology on $X$, 
    $I$ a Hausdorff topological space and $(S_{n}(t))_{n\in\N, t\in I}$ a family of linear maps $S_n(t)\colon X \to X$. 
    Due to \cite[I.1.10 Proposition]{Cooper1978}, $(S_{n}(t))_{n\in\N, t\in I}$ is uniformly locally bi-equicontinuous 
    if and only if it is sequentially uniformly locally $\gamma$-equicontinuous.
\end{remark}

We note the following special version of the Trotter--Kato theorem \cite[Theorem 16, Remark 18]{AlbaneseKuehnemund2002}.

\begin{proposition}
\label{prop:Trotter-Kato_sg_lkx}
    Let $(X,\topo)$ be a sequentially complete Hausdorff locally convex space and 
    $(S_n(t))_{t\geq 0}$ a strongly continuous semigroup on $X$ with generator $A_n$ for each $n\in\N$ such that
    $(S_n(t))_{n\in\N, t\geq 0}$ is uniformly locally equicontinuous. Let there be $\alpha\in\R$ such that 
    $(\euler^{-\alpha t}S_n(t))_{t\geq 0}$ is equicontinuous for each $n\in\N$ and let $\lambda>\alpha$.
    Consider the following assertions.
    \begin{enumerate}[label=\upshape(\alph*), leftmargin=*]
        \item\label{it:Trotter-Kato_sg_lkx1} There exists a densely defined linear operator $(A,D(A))$ and a core $D\subseteq \bigcup_{k=1}^{\infty}\bigcap_{n=k}^{\infty} D(A_n)$ for $A$ such that $A_nx\to Ax$ for all $x\in D$ and $(\lambda-A)$ has dense range.
        \item\label{it:Trotter-Kato_sg_lkx2} There exists $R\in\calL(X,\topo)$ such that $R(\lambda,A_n)x\to Rx$ for all $x\in X$ and $R$ is injective and has dense range.
        \item\label{it:Trotter-Kato_sg_lkx3} There exists a strongly continuous semigroup $(S(t))_{t\geq 0}$ on $X$ such that 
        $(\euler^{-\alpha t}S(t))_{t\geq 0}$ is equicontinuous and for all $x\in X$ we have $S_n(\cdot)x\to S(\cdot)x$ 
        uniformly on compact subsets of $[0,\infty)$.
    \end{enumerate}
    Then \ref{it:Trotter-Kato_sg_lkx1}$\Rightarrow$\ref{it:Trotter-Kato_sg_lkx2}$\Leftrightarrow$\ref{it:Trotter-Kato_sg_lkx3}. 
    In particular, if \ref{it:Trotter-Kato_sg_lkx2} holds, then the generator $B$ of $(S(t))_{t\geq 0}$ fulfils $R(\lambda,B) = R$.
    If \ref{it:Trotter-Kato_sg_lkx1} holds, then $A$ is closable and $B=\overline{A}$.
\end{proposition}
\begin{proof}
    Since $\lambda>\alpha$, we have that $\lambda\in\rho_{\topo}(A_{n})$ for all $n\in\N$ and 
    \[
    R(\lambda,A_n)x = \int_0^\infty \euler^{-\lambda t} S_n(t)x\,\d t \quad (n\in\N,\,x\in X)
    \]
    (see the proof of \prettyref{thm:Hille_Yosida_sg}). Hence, our statement follows from the Trotter--Kato theorem 
    \cite[Theorem 16, Remark 18]{AlbaneseKuehnemund2002} with $a=\infty$ where one replaces in the proof of 
    \cite[Theorem 16 (b)$\Rightarrow$(c)]{AlbaneseKuehnemund2002} the Hille--Yosida theorem \cite[Theorem 5]{AlbaneseKuehnemund2002} 
    by \prettyref{prop:Hille--Yosida_lcx} to obtain that $(\euler^{-\alpha t}S(t))_{t\geq 0}$ is equicontinuous.
\end{proof}

\begin{remark}\label{rem:core_cond}
    Comparing \prettyref{prop:Trotter-Kato_sg_lkx} \ref{it:Trotter-Kato_sg_lkx1} with \cite[Theorem 16 (a)]{AlbaneseKuehnemund2002}, we remark that we corrected an assumption from the latter. 
    In \cite[Theorem 16 (a)]{AlbaneseKuehnemund2002} it is merely assumed that $D$ is a core for $A$ but that does not guarantee that the limit 
    $\lim_{n\to\infty} A_n x$ is well-defined for $x\in D$ since $x$ need not be in $D(A_n)$ for any $n\in\N$. Therefore, we assume that 
    $D\subseteq \bigcup_{k=1}^{\infty}\bigcap_{n=k}^{\infty} D(A_n)$ which implies that for every $x\in D$ there is some $k\in\N$ such that $x\in D(A_n)$ for all $n\geq k$. 
    Adjusting the proof of \cite[Theorem 16 (a)$\Rightarrow$(b)]{AlbaneseKuehnemund2002} in the obvious way, the conclusion of \cite[Theorem 16]{AlbaneseKuehnemund2002} 
    remains valid after this correction. The origin of this correction is \cite[Proposition 4.1]{Konishi1972} which is the counterpart for cosine families.
\end{remark}

We can recover the bi-continuous analogue, see also \cite[Theorem 2.3, Remark 2.4, Theorem 2.6]{Kuehnemund2001}. 

\begin{corollary}[{cf.~\cite[Theorem 3.6]{AlbaneseMangino2004}}]
\label{cor:Trotter-Kato_sg_bicont}
    Let $(X,\norm{\cdot},\topo)$ be a sequentially complete C-sequential Saks space and
    $(S_n(t))_{t\geq 0}$ a locally $\topo$-bi-continuous semigroup on $X$ with bi-generator $A_n$ for each $n\in\N$ such that $(S_n(t))_{n\in\N, t\geq 0}$ is uniformly locally bi-equicontinuous. Let $\omega\in\R$ an $M\geq 1$ such that 
    $\norm{S_n(t)}_{\calL(X)}\leq M \euler^{\omega t}$ for all $t\geq 0$ and $n\in\N$ and let $\lambda>\omega$.
    Consider the following assertions.
    \begin{enumerate}[label=\upshape(\alph*), leftmargin=*]
        \item\label{it:Trotter-Kato_sg_bi1} There exists a bi-densely defined linear operator $(A,D(A))$ and a bi-core $D\subseteq \bigcup_{k=1}^{\infty}\bigcap_{n=k}^{\infty} D(A_n)$ for $A$ such that $(A_nx)_{n\in\N}$ is $\norm{\cdot}$-bounded and 
        $A_nx\to Ax$ w.r.t.~$\topo$ for all $x\in D$ and $(\lambda-A)$ has bi-dense range.
        \item\label{it:Trotter-Kato_sg_bi1_1} There exists a $\gamma$-densely defined linear operator $(A,D(A))$ and a 
        $\gamma$-core $D\subseteq \bigcup_{k=1}^{\infty}\bigcap_{n=k}^{\infty} D(A_n)$ for $A$ such that $A_nx\to Ax$  w.r.t.~$\gamma$ 
        for all $x\in D$ and $(\lambda-A)$ has $\gamma$-dense range.
        \item\label{it:Trotter-Kato_sg_bi2} There exists $R\in\calL(X)$ such that $R(\lambda,A_n)x\to Rx$ w.r.t.~$\topo$ for all $x\in X$ and $R$ is injective 
        and has bi-dense range.
        \item\label{it:Trotter-Kato_sg_bi2_1} There exists $R\in\calL(X,\gamma)$ such that $R(\lambda,A_n)x\to Rx$ w.r.t.~$\gamma$ for all $x\in X$ and $R$ is injective 
        and has $\gamma$-dense range.
        \item\label{it:Trotter-Kato_sg_bi3} There exists a locally $\topo$-bi-continuous semigroup $(S(t))_{t\geq 0}$ on $X$ 
        such that $\norm{S(t)}_{\calL(X)}\leq M \euler^{\omega t}$ for all $t\geq 0$, and for all $x\in X$ we have $S_n(\cdot)x\to S(\cdot)x$ w.r.t.~$\topo$ uniformly on compact subsets of $[0,\infty)$.
    \end{enumerate}
    Then \ref{it:Trotter-Kato_sg_bi1}$\Rightarrow$\ref{it:Trotter-Kato_sg_bi1_1}$\Rightarrow$\ref{it:Trotter-Kato_sg_bi2}$\Leftrightarrow$\ref{it:Trotter-Kato_sg_bi2_1}$\Leftrightarrow$\ref{it:Trotter-Kato_sg_bi3}
    In particular, if \ref{it:Trotter-Kato_sg_bi2} holds, then the bi-generator $B$ of $(S(t))_{t\geq 0}$ fulfils $R(\lambda,B) = R$. 
    If \ref{it:Trotter-Kato_sg_bi1_1} holds, then $A$ is $\gamma$-closable and $B=\overline{A}^{\gamma}$. 
    If \ref{it:Trotter-Kato_sg_bi1} holds, then $B=\overline{A}^{\operatorname{bi}}$.
\end{corollary}

\begin{proof}
    We start with some observations.
    Again, we note that $(X,\gamma)$ is sequentially complete and C-sequential by \prettyref{defi:seq_compl_C-seq_Saks}.
    Due to \prettyref{rem:Saks_seq_complete_Banach} \ref{it:Saks_seq:compl_Banach_1} and \prettyref{cor:bi_cont_quasi_equicont_C_seq_sg_cosine} \ref{it:bi_cont_sg_cosine_rom1}, $(S_n(t))_{t\geq 0}$ is $\gamma$-strongly continuous and exponentially $\gamma$-equicontinuous and its bi-generator 
    $A_n$ coincides with its generator w.r.t.~$\gamma$ for all $n\in\N$. 
    Furthermore, $(S_n(t))_{n\in\N, t\geq 0}$ is uniformly locally $\gamma$-equicontinuous 
    by \prettyref{prop:bi_equicont_mixed_equicont} \ref{it:bi_equicont_mixed_3} since it is uniformly locally bi-equicontinuous and $(X,\gamma)$ C-sequential. 
    By \prettyref{rem:count_conv} \ref{it:count_conv_1}, \prettyref{cor:type}, \prettyref{rem:Saks_bounded_sets} \ref{it:Saks_bounded_sets_2}, the local $\gamma$-equicontinuity of $(S_n(t))_{t\geq 0}$ 
    and the inequality $\norm{S_n(t)}_{\calL(X)} \leq M\euler^{\omega t}$ for all $t\geq 0$ and $n\in\N$, we have that  
    for all $n\in\N$, $\alpha>\omega$ and $p\in\calN_{\gamma}$ there exists $q\in\calN_{\gamma}$ such that 
    \[
    p(\euler^{-\alpha t}S_n(t)x)\leq Mq(x) \quad(t\geq 0,\, x\in X).
    \]
    By (the proof of) \prettyref{thm:Hille_Yosida_sg}, this implies that $\lambda\in\rho_{\gamma}(A_n)\subseteq\rho_{\norm{\cdot}}(A_n)$ for all $n\in\N$ and 
    \[
    R(\lambda,A_n)x = \int_0^\infty \euler^{-\lambda t} S_n(t)x\,\d t \quad (n\in\N,\,x\in X)
    \]
    where the integral is an improper Riemann integral in $(X,\gamma)$. 
    Furthermore, for all $n\in\N$, $\alpha\in\R$ such that $\lambda>\alpha>\omega$ and $p\in\calN_{\gamma}$ there exists $\widetilde{q}\in\calN_{\gamma}$ such 
    that 
    \begin{equation}\label{eq:Hille-Yosida_sg_estimate_1}
    p(R(\lambda,A_n)x) \leq \frac{M}{|\lambda-\alpha|}\widetilde{q}(x)\quad (x\in X)
    \end{equation}
    and 
    \begin{equation}\label{eq:Hille-Yosida_sg_estimate_2}
    \norm{R(\lambda,A_n) x} \leq \frac{M}{|\lambda-\omega|}\norm{x} \quad (x\in X)
    \end{equation}
    by \prettyref{rem:count_conv} \ref{it:count_conv_1}, \prettyref{rem:Saks_bounded_sets} \ref{it:Saks_bounded_sets_2}, \prettyref{thm:Hille_Yosida_sg} and \prettyref{cor:Hille--Yosida_bi-continuous} with $m\coloneq 1$.
    
    ``\ref{it:Trotter-Kato_sg_bi1}$\Rightarrow$\ref{it:Trotter-Kato_sg_bi1_1}''
     Due to \prettyref{rem:bi_notions_operators}, the bi-densely defined linear operator $(A,D(A))$ is sequentially $\gamma$-dense, so $\gamma$-dense, 
     the bi-core $D$ is a sequential $\gamma$-core, so a $\gamma$-core, and  the bi-dense range of $(\lambda-A)$ is sequentially $\gamma$-dense, so $\gamma$-dense. 
     Moreover, that $(A_nx)_{n\in\N}$ is $\norm{\cdot}$-bounded and $A_nx\to Ax$ w.r.t.~$\topo$ for all $x\in D$ is equivalent to $A_nx\to Ax$ w.r.t.~$\gamma$ 
     for all $x\in D$ by \cite[I.1.10 Proposition]{Cooper1978}.
     
     ``\ref{it:Trotter-Kato_sg_bi1_1}$\Rightarrow$\ref{it:Trotter-Kato_sg_bi2_1}'' 
     This follows directly from \prettyref{prop:Trotter-Kato_sg_lkx} \ref{it:Trotter-Kato_sg_lkx1}$\Rightarrow$\ref{it:Trotter-Kato_sg_lkx2} 
     and our observations above since there is some $\alpha\in\R$ such that $\lambda>\alpha>\omega$.

    ``\ref{it:Trotter-Kato_sg_bi2_1}$\Rightarrow$\ref{it:Trotter-Kato_sg_bi3}''
    By our observations above and \prettyref{prop:Trotter-Kato_sg_lkx} \ref{it:Trotter-Kato_sg_lkx2}$\Rightarrow$\ref{it:Trotter-Kato_sg_lkx3}, 
    there exists a $\gamma$-strongly continuous semigroup $(S(t))_{t\geq 0}$ on $X$ such that  
    $(\euler^{-\alpha t}S(t))_{t\geq 0}$ is $\gamma$-equicontinuous for all $\alpha\in\R$ such that $\lambda>\alpha>\omega$, and for all $x\in X$ 
    we have that $S_n(\cdot)x\to S(\cdot)x$ w.r.t.~$\gamma$ uniformly on compact subsets of $[0,\infty)$. Since $\topo\subseteq\gamma$, 
    this implies that $S_n(\cdot)x\to S(\cdot)x$ w.r.t.~$\topo$ uniformly on compact subsets of $[0,\infty)$.    
    Let $t_0\geq 0$, $\varepsilon >0$, $p\in\calN_{\gamma}$ and $x\in X$. Since $S_n(\cdot)x\to S(\cdot)x$ w.r.t.~$\gamma$ uniformly on compact subsets of 
    $[0,\infty)$, there is $N\in\N$ such that for all $t\in [0,t_0]$ 
    \[
         p(\euler^{-\omega t}S(t)x)
    \leq p(\euler^{-\omega t}S_N(t)x)+p(\euler^{-\omega t}(S(t)-S_N(t))x)
    \leq \norm{\euler^{-\omega t}S_N(t)x}+\euler^{|\omega|t_0}\varepsilon
    \leq M\norm{x}+\euler^{|\omega|t_0}\varepsilon .
    \]
    Since $\varepsilon >0$ is arbitrary, this implies that 
    \[
    p(\euler^{-\omega t}S(t)x)\leq M\norm{x} \quad (t\in [0,t_0],\, x\in X)
    \]
    for all $p\in\calN_{\gamma}$ and $t_0\geq 0$. By \prettyref{rem:Saks_bounded_sets}, we obtain that 
    $\norm{S(t)}_{\calL(X)}\leq M \euler^{\omega t}$ for all $t\geq 0$. 
    Moreover, $(S(t))_{t\geq 0}$ is locally $\topo$-bi-continuous by \prettyref{cor:bi_cont_quasi_equicont_C_seq_sg_cosine} \ref{it:bi_cont_sg_cosine_rom1}. 

    ``\ref{it:Trotter-Kato_sg_bi3}$\Rightarrow$\ref{it:Trotter-Kato_sg_bi2}'' 
    Due to \prettyref{rem:Saks_seq_complete_Banach} \ref{it:Saks_seq:compl_Banach_1} and \prettyref{cor:bi_cont_quasi_equicont_C_seq_sg_cosine} \ref{it:bi_cont_sg_cosine_rom1}, $(S(t))_{t\geq 0}$ is $\gamma$-strongly continuous and exponentially $\gamma$-equicontinuous and its bi-generator 
    $B$ coincides with its generator w.r.t.~$\gamma$. 
    By \prettyref{rem:count_conv} \ref{it:count_conv_1}, \prettyref{cor:type}, \prettyref{rem:Saks_bounded_sets} \ref{it:Saks_bounded_sets_2}, the local $\gamma$-equicontinuity of $(S(t))_{t\geq 0}$ 
    and the inequality $\norm{S(t)}_{\calL(X)} \leq M\euler^{\omega t}$ for all $t\geq 0$, we have that $\lambda\in\rho_{\gamma}(B)$ 
    by \prettyref{thm:Hille_Yosida_sg}. We define $R\coloneq R(\lambda,B)\in\calL(X,\gamma)$ and note that $R\in\calL(X)$ by 
    \prettyref{rem:gamma_seq_cont_norm_cont}, $R$ is injective, and $R(X)=D(B)$ which is sequentially $\gamma$-dense in $X$ by 
    \prettyref{prop:seq_densely_defined}. In particular, $R(X)$ is sequentially bi-dense by \prettyref{rem:bi_notions_operators}. 
    Looking into the proof of \prettyref{prop:Trotter-Kato_sg_lkx} \ref{it:Trotter-Kato_sg_lkx3}$\Rightarrow$\ref{it:Trotter-Kato_sg_lkx2}, 
    see \cite[Theorem 14 (b)$\Rightarrow$(a)]{AlbaneseKuehnemund2002}, it holds that $R(\lambda,A_n)x\to Rx$ w.r.t.~$\gamma$ for all $x\in X$, so also 
    w.r.t.~the weaker topology $\topo$. 
        
    ``\ref{it:Trotter-Kato_sg_bi2}$\Rightarrow$\ref{it:Trotter-Kato_sg_bi2_1}''  
    First, we remark that $R\in\calL(X)$ has bi-dense range, so sequentially $\gamma$-dense range by \prettyref{rem:bi_notions_operators}, 
    and thus $\gamma$-dense range. Due to \eqref{eq:Hille-Yosida_sg_estimate_2}, 
    the sequence $(R(\lambda,A_n) x)_{n\in\N}$ is $\norm{\cdot}$-bounded for each $x\in X$. 
    In combination with $R(\lambda,A_n)x\to Rx$ w.r.t.~$\topo$ for all $x\in X$ this yields that $R(\lambda,A_n)x\to Rx$ w.r.t.~$\gamma$ for all $x\in X$ 
    by \cite[I.1.10 Proposition]{Cooper1978}. It remains to show that $R\in\calL(X,\gamma)$. 
    Let $\varepsilon >0$, $p\in\calN_{\gamma}$, $x\in X$ and $\alpha\in\R$ such that $\lambda>\alpha>\omega$. 
    Since $R(\lambda,A_n)x\to Rx$ w.r.t.~$\gamma$ for all $x\in X$ and using \eqref{eq:Hille-Yosida_sg_estimate_1}, 
    there are $N\in\N$ and $\widetilde{q}\in\calN_{\gamma}$ such that
    \[
         p(Rx)
    \leq p(R(\lambda,A_N)x)+p(Rx-R(\lambda,A_N)x)
    \leq p(R(\lambda,A_N)x)+\varepsilon
    \leq \frac{M}{|\lambda-\alpha|}\widetilde{q}(x)+\varepsilon .
    \]
    Since $\varepsilon >0$ is arbitrary, this implies that $R\in\calL(X,\gamma)$.

    From the addenda we only need to show that $B=\overline{A}^{\operatorname{bi}}$ if \ref{it:Trotter-Kato_sg_bi1} holds. 
    The others follow from \prettyref{prop:Trotter-Kato_sg_lkx}. So, let \ref{it:Trotter-Kato_sg_bi1} hold. Then \ref{it:Trotter-Kato_sg_bi1_1} holds 
    and $A$ is $\gamma$-closable with $B=\overline{A}^{\gamma}$. In particular, $A$ is sequentially $\gamma$-closable, so bi-closable 
    by \prettyref{rem:bi_notions_operators}, $D(\overline{A}^{\operatorname{bi}})\subseteq D(B)$ and $\overline{A}^{\operatorname{bi}}=B$ on 
    $D(\overline{A}^{\operatorname{bi}})$. The converse inclusion follows analogously to \cite[p.~45]{Kuehnemund2001}. 
\end{proof}

Comparing \prettyref{cor:Trotter-Kato_sg_bicont} with \cite[Theorem 3.6]{AlbaneseMangino2004}, we note that the only additional assumption 
in \prettyref{cor:Trotter-Kato_sg_bicont} is that $(X,\norm{\cdot},\topo)$ is C-sequential, see \prettyref{rem:additional_assump}, which also allows 
us to add parts \ref{it:Trotter-Kato_sg_bi1_1} and \ref{it:Trotter-Kato_sg_bi2_1}. In the examples of \cite[5.~Examples]{AlbaneseMangino2004} the spaces 
$(X,\norm{\cdot},\topo)$ are C-sequential by \cite[Remark 3.19 (a)]{KruseSchwenninger2022}. 
Moreover, as in \prettyref{rem:core_cond} a correction for the assumptions on the bi-core $D$ in \cite[Theorem 3.6 (a)]{AlbaneseMangino2004} 
is needed, which we did in part \ref{it:Trotter-Kato_sg_bi1} of \prettyref{cor:Trotter-Kato_sg_bicont}.

Similarly to the semigroups, we can recover the Trotter--Kato approximation theorem for locally bi-continuous cosine families from its 
version on sequentially complete Hausdorff locally convex spaces. Such a result can be found in \cite[Theorems C, C', D, Proposition 4.1]{Konishi1972} by Konishi. However, Konishi works with generalised resolvents in order to cope with the fact the cosine families may only be locally and not exponentially equicontinuous and thus the Laplace 
transform may only be performed in a distributional way. Under the additional assumption of exponential equicontinuity this issue can be overcome.

\begin{proposition}[{cf.~\cite[Theorems C', D, Proposition 4.1]{Konishi1972}}]
\label{prop:Trotter-Kato_cosine_lkx}
    Let $(X,\topo)$ be a sequentially complete Hausdorff locally convex space and
    $(S_n(t))_{t\geq 0}$ a strongly continuous cosine family on $X$ with generator $A_n$ for each $n\in\N$ 
    such that $(S_n(t))_{n\in\N, t\geq 0}$ is uniformly locally equicontinuous. Let there be $\alpha\geq 0$ such that $(\euler^{-\alpha t}S_n(t))_{t\geq 0}$ is equicontinuous for each $n\in\N$ and let $\lambda >\alpha$.
    Consider the following assertions.
    \begin{enumerate}[label=\upshape(\alph*), leftmargin=*]
        \item\label{it:Trotter-Kato_cosine_lkx1} There exists a densely defined linear operator $(A,D(A))$ and a core $D\subseteq \bigcup_{k=1}^{\infty}\bigcap_{n=k}^{\infty} D(A_n)$ for $A$ such that $A_nx\to Ax$ for all $x\in D$ and $(\lambda^2-A)$ has dense range.
        \item\label{it:Trotter-Kato_cosine_lkx2} There exists $R\in\calL(X,\topo)$ such that $R(\lambda^2,A_n)x\to Rx$ for all $x\in X$ and $R$ is injective and has dense range.
        \item\label{it:Trotter-Kato_cosine_lkx3} There exists a strongly continuous cosine family $(S(t))_{t\geq 0}$ on $X$ such that $(\euler^{-\alpha t}S(t))_{t\geq 0}$ is equicontinuous and for all $x\in X$ we have $S_n(\cdot)x\to S(\cdot)x$ uniformly on compact subsets of $[0,\infty)$.
    \end{enumerate}
    Then \ref{it:Trotter-Kato_cosine_lkx1}$\Rightarrow$\ref{it:Trotter-Kato_cosine_lkx2}$\Leftrightarrow$\ref{it:Trotter-Kato_cosine_lkx3}. 
    In particular, if \ref{it:Trotter-Kato_cosine_lkx2} holds, then the generator $B$ of $(S(t))_{t\geq 0}$ fulfils $R(\lambda^2,B) = R$.
    If \ref{it:Trotter-Kato_cosine_lkx1} holds, then $A$ is closable and $B=\overline{A}$.
\end{proposition}

Using this result, we can recover and extend the bi-continuous analogue in \cite[Theorem 1]{Budde2025}, which itself is a special case of 
\cite[Theorem 3.1]{LiSongZhao2010a}.

\begin{corollary}
\label{cor:Trotter-Kato_cosine_bicont}
    Let $(X,\norm{\cdot},\topo)$ be a sequentially complete C-sequential Saks space and
    $(S_n(t))_{t\geq 0}$ a locally $\topo$-bi-continuous cosine family on $X$ with bi-generator $A_n$ for each $n\in\N$ such that $(S_n(t))_{n\in\N, t\geq 0}$ is uniformly locally bi-equicontinuous. Let $\omega\geq 0$ and $M\geq 1$ such that 
    $\norm{S_n(t)}_{\calL(X)}\leq M \euler^{\omega t}$ for all $t\geq 0$ and $n\in\N$ and let $\lambda>\omega$.
    Consider the following assertions.
    \begin{enumerate}[label=\upshape(\alph*), leftmargin=*]
        \item\label{it:Trotter-Kato_cosine_bi1} There exists a bi-densely defined linear operator $(A,D(A))$ and a bi-core $D\subseteq \bigcup_{k=1}^{\infty}\bigcap_{n=k}^{\infty} D(A_n)$ for $A$ such that $(A_nx)_{n\in\N}$ is $\norm{\cdot}$-bounded 
        and $A_nx\to Ax$ w.r.t.~$\topo$ for all $x\in D$ and $(\lambda^2-A)$ has bi-dense range.
        \item\label{it:Trotter-Kato_cosine_bi1_1} There exists a $\gamma$-densely defined linear operator $(A,D(A))$ and a 
        $\gamma$-core $D\subseteq \bigcup_{k=1}^{\infty}\bigcap_{n=k}^{\infty} D(A_n)$ for $A$ such that $A_nx\to Ax$  w.r.t.~$\gamma$ 
        for all $x\in D$ and $(\lambda^2-A)$ has $\gamma$-dense range.
        \item\label{it:Trotter-Kato_cosine_bi2} There exists $R\in\calL(X)$ such that $R(\lambda^2,A_n)x\to Rx$ w.r.t.~$\topo$ for all $x\in X$ and $R$ is injective and has bi-dense range.
        \item\label{it:Trotter-Kato_cosine_bi2_1} There exists $R\in\calL(X,\gamma)$ such that $R(\lambda^2,A_n)x\to Rx$ w.r.t.~$\gamma$ for all $x\in X$ and $R$ is injective and has $\gamma$-dense range.        
        \item\label{it:Trotter-Kato_cosine_bi3} There exists a locally $\topo$-bi-continuous cosine family $(S(t))_{t\geq 0}$ on $X$ 
        such that $\norm{S(t)}_{\calL(X)}\leq M \euler^{\omega t}$ for all $t\geq 0$, and for all $x\in X$ we have $S_n(\cdot)x\to S(\cdot)x$ w.r.t.~$\topo$ uniformly on compact subsets of $[0,\infty)$.
    \end{enumerate}
    Then \ref{it:Trotter-Kato_cosine_bi1}$\Rightarrow$\ref{it:Trotter-Kato_cosine_bi1_1}$\Rightarrow$\ref{it:Trotter-Kato_cosine_bi2}$\Leftrightarrow$\ref{it:Trotter-Kato_cosine_bi2_1}$\Leftrightarrow$\ref{it:Trotter-Kato_cosine_bi3}. 
    In particular, if \ref{it:Trotter-Kato_cosine_bi2} holds, then the bi-generator $B$ of $(S(t))_{t\geq 0}$ fulfils $R(\lambda^2,B) = R$ 
    If \ref{it:Trotter-Kato_cosine_bi1_1} holds, then $A$ is $\gamma$-closable and $B=\overline{A}^{\gamma}$. 
    If \ref{it:Trotter-Kato_cosine_bi1} holds, then $B=\overline{A}^{\operatorname{bi}}$.
\end{corollary}

\begin{proof}
    We start with some observations. Again, we note that $(X,\gamma)$ is sequentially complete and C-sequential by \prettyref{defi:seq_compl_C-seq_Saks}. 
    By the same reasoning as in the proof of \prettyref{cor:Trotter-Kato_sg_bicont}, we have that 
    $(S_n(t))_{t\geq 0}$ is $\gamma$-strongly continuous and exponentially $\gamma$-equicontinuous and its bi-generator 
    $A_n$ coincides with its generator w.r.t.~$\gamma$ for all $n\in\N$. 
    Furthermore, $(S_n(t))_{n\in\N, t\geq 0}$ is uniformly locally $\gamma$-equicontinuous and
    for all $n\in\N$, $\alpha>\omega$ and $p\in\calN_{\gamma}$ there exists $q\in\calN_{\gamma}$ such that 
    \[
    p(\euler^{-\alpha t}S_n(t)x)\leq Mq(x) \quad(t\geq 0,\, x\in X).
    \]
    By the proof of \prettyref{cor:bi_cont_quasi_equicont_C_seq_sg_cosine} \ref{it:bi_cont_sg_cosine_rom4} and \eqref{eq:resol_laplace_cosine}, 
    this implies that $\lambda^2\in\rho_{\gamma}(A_n)\subseteq\rho_{\norm{\cdot}}(A_n)$ for all $n\in\N$ and 
    \[
    \lambda R(\lambda^2,A_n)x = \int_0^\infty \euler^{-\lambda t} S_n(t)x\,\d t \quad (n\in\N,\,x\in X)
    \]
    where the integral is an improper Riemann integral in $(X,\gamma)$. 
    Furthermore, for all $n\in\N$, $\alpha\in\R$ such that $\lambda>\alpha>\omega$ and $p\in\calN_{\gamma}$ there exists $\widetilde{q}\in\calN_{\gamma}$ such 
    that 
    \begin{equation}\label{eq:Hille-Yosida_cosine_estimate_1}
    p(R(\lambda^2,A_n)x) \leq \frac{M}{|\lambda-\alpha||\lambda|}\widetilde{q}(x)\quad (x\in X)
    \end{equation}
    and 
    \begin{equation}\label{eq:Hille-Yosida_cosine_estimate_2}
    \norm{R(\lambda^2,A_n) x} \leq \frac{M}{|\lambda-\omega||\lambda|}\norm{x} \quad (x\in X)
    \end{equation}
    by \eqref{eq:resolv_estimate_cosine} and \prettyref{prop:generation_cosine_bi-continuous} with $m\coloneq 0$. 

    The rest of the proof is analogous to the one of \prettyref{cor:Trotter-Kato_sg_bicont}. 
    One just uses \prettyref{prop:Trotter-Kato_cosine_lkx}, \eqref{eq:Hille-Yosida_cosine_estimate_1} and \eqref{eq:Hille-Yosida_cosine_estimate_2} 
    instead of \prettyref{prop:Trotter-Kato_sg_lkx}, \eqref{eq:Hille-Yosida_sg_estimate_1} and \eqref{eq:Hille-Yosida_sg_estimate_2}, respectively. 
    Further, in the proof of ``\ref{it:Trotter-Kato_cosine_bi3}$\Rightarrow$\ref{it:Trotter-Kato_cosine_bi2}'' the fact that $\lambda^2\in\rho_{\gamma}(B)$ 
    follows from \prettyref{prop:generation_cosine_lcx} instead of \prettyref{thm:Hille_Yosida_sg}.
\end{proof}

Comparing \prettyref{cor:Trotter-Kato_cosine_bicont} with \cite[Theorem 1]{Budde2025}, we note again that the only additional assumption 
in \prettyref{cor:Trotter-Kato_sg_bicont} is that $(X,\norm{\cdot},\topo)$ is C-sequential, see \prettyref{rem:additional_assump}, 
which also allows us to add parts \ref{it:Trotter-Kato_cosine_bi1_1} and \ref{it:Trotter-Kato_cosine_bi2_1}. In the example of \cite[Section 4]{Budde2025} the space 
$(X,\norm{\cdot},\topo)$ is C-sequential by \cite[Remark 3.19 (a)]{KruseSchwenninger2022}. 
Moreover, in \cite[Theorem 1]{Budde2025} part \ref{it:Trotter-Kato_cosine_bi1} is not present and it is a-priori assumed that there is already 
a locally $\topo$-bi-continuous cosine family $(S(t))_{t\geq 0}$ with bi-generator $B$ and $R=R(\lambda^2,B)$.

\subsection{The Weierstra{\ss} formula}

Similarly as in \cite[Eq.~(3.102)]{ArendtBHN2011} we can deduce a Weierstra{\ss} formula for $2n$-times integrated locally $\topo$-bi-continuous and exponentially bounded $C$-cosine families. 

Let $n\in\N_0$, $(X,\norm{\cdot},\topo)$ be a sequentially complete Saks space, $C\in\calL(X,\gamma)$ and $(S(t))_{t\geq 0}$ 
an $n$-times integrated $\gamma$-strongly continuous exponentially $\gamma$-equicontinuous semigroup or cosine family on 
$X$ which is \emph{non-degenerate}, i.e.~if $x\in X$ is such that $S(t)x=0$ for all $t>0$, then $x=0$. 
Since $(S(t))_{t\geq 0}$ is exponentially $\gamma$-equicontinuous, there is $\omega\geq 0$ such that 
$(\euler^{-\omega t}S(t))_{t\geq 0}$ is $\gamma$-equicontinuous, and we set 
\[
R(\lambda)x\coloneq\int_{0}^{\infty}\lambda^n \euler^{-\lambda t}S(t)x\,\d t \quad (\lambda>\omega,\, x\in X)
\]
where the integral is an improper Riemann integral in the sequentially complete space $(X,\gamma)$.
$(S(t))_{t\geq 0}$ being $n$-times integrated and non-degenerate implies that $C$ is injective, and the uniqueness of the Laplace 
transform implies that $R(\lambda)$ is injective for $\lambda>\omega$. The closed linear operator $A$ defined by 
\begin{align*}
D(A)&\coloneq \{x\in X\;|\;Cx\in R(\lambda)(X)\},\\
Ax&\coloneq 
\begin{cases}
(\lambda-R(\lambda)^{-1}C)x & \quad (x\in D(A))\text{ if }(S(t))_{t\geq 0} \text{ is a semigroup},\\
\lambda(\lambda-R(\lambda)^{-1}C)x & \quad (x\in D(A))\text{ if }(S(t))_{t\geq 0} \text{ is a cosine family},
\end{cases}
\end{align*}
is called the \emph{generator} of $(S(t))_{t\geq 0}$ and independent of $\lambda>\omega$, see \cite[p.~30]{LiShaw1993}. If $n=0$ and $C=\id$, then this definition of the generator is equivalent to the standard definition given above \prettyref{cor:bi_cont_quasi_equicont_C_seq_sg_cosine}, see \cite[p.~30]{LiShaw1993}. 
Furthermore, it also coincides with the definition of the bi-generator of a non-degenerate 
$n$-times integrated locally $\topo$-bi-continuous exponentially bounded semigroup or cosine family on 
$X$ given in \cite[Definition 2]{ChangLiu2012} and \cite[Definition 2.4]{LiSongZhao2010a}, respectively, by \cite[Proposition 2.5 (b)]{KruseSchwenninger2024}, the fact that $\topo\subseteq\gamma$ and, in 
the case of cosine families, by \cite[Proposition 1.6]{ShawLi1995}.

\begin{proposition}
    Let $(X,\norm{\cdot},\topo)$ be a sequentially complete C-sequential Saks space, $C\in\calL(X,\gamma)$,
    $(C(t))_{t\geq 0}$ a non-degenerate $2n$-times integrated locally $\topo$-bi-continuous and exponentially bounded $C$-cosine family on $X$ for $n\in\N_0$ with bi-generator $A$. Define
    \[
    S(t)x\coloneq  \frac{(2n)!}{2^{2n-1} \cdot n!\cdot \Gamma(n+\frac{1}{2})} \int_0^\infty \euler^{-s^2} C(2t^{1/2}s)x\,\d s 
    \quad (t\geq 0,\, x\in X)
    \]
    where the integral above is an improper Riemann integral in $(X,\gamma)$.
    Then $(S(t))_{t\geq 0}$ is a non-degenerate $n$-times integrated locally $\topo$-bi-continuous and exponentially bounded $C$-semigroup on $X$ generated by $A$.
\end{proposition}

\begin{proof}
    By \prettyref{rem:standard_Fs_fine} and \prettyref{thm:bi_cont_quasi_equicont_C_seq} \ref{it:bi_quasi_equicont_rom1}, 
    $(C(t))_{t\geq 0}$ is $\gamma$-strongly continuous and $\calF_{\exp}$-$\gamma$-equicontinu\-ous. 
    By \cite[Theorem 1]{LiShaw1993}, $(S(t))_{t\geq 0}$ defines a non-degenerate $n$-times integrated $\gamma$-strongly continuous and $\calF_{\exp}$-$\gamma$-equi\-contin\-uous $C$-semigroup generated 
    by $A$.
    Another application of \prettyref{thm:bi_cont_quasi_equicont_C_seq} \ref{it:bi_quasi_equicont_rom1} yields that $(S(t))_{t\geq 0}$ is locally $\topo$-bi-continuous and exponentially bounded.
\end{proof}

\subsection{D'Alembert formula}

We now derive a D'Alembert formula for $n$-times integrated locally $\topo$-bi-continuous and exponentially bounded $C$-cosine families, which gives rise to construct such examples; for particular instances see \cite[p.~56]{DuanSun2006} and \cite[Example 1.7]{Budde2024} 
(which is the same example). They follow from \cite[Examples 6 (b)]{Kuehnemund2003} and \cite[Remark 3.19 (a)]{KruseSchwenninger2022} 
by the proposition below.

\begin{proposition}
    Let $(X,\norm{\cdot},\topo)$ be a sequentially complete C-sequential Saks space, $C\in\calL(X,\gamma)$,
    $(S_+(t))_{t\geq 0}$ and ($S_-(t))_{t\geq 0}$ non-degenerate $n$-times integrated locally $\topo$-bi-continuous and exponentially bounded $C$-semigroups on $X$ for $n\in\N_0$ 
    with bi-generators $A$ and $-A$, respectively.
    Define
    \[
    S(t)x\coloneq \frac{1}{2}(S_+(t)x + S_-(t)x) \quad (t\geq 0,\, x\in X).
    \]
    Then $(S(t))_{t\geq 0}$ is a non-degenerate $n$-times integrated locally $\topo$-bi-continuous and exponentially bounded $C$-cosine family on $X$ generated by $A^2$.
\end{proposition}

\begin{proof}
    By \prettyref{rem:standard_Fs_fine} and \prettyref{thm:bi_cont_quasi_equicont_C_seq} \ref{it:bi_quasi_equicont_rom1}, $(S_+(t))_{t\geq 0}$ and $(S_-(t))_{t\geq 0}$ are $\gamma$-strongly continuous and $\calF_{\exp}$-$\gamma$-equicontinuous. 
    By \cite[Theorem 4]{LiShaw1993}, $(S(t))_{t\geq 0}$ defines a non-degenerate $n$-times integrated $\gamma$-strongly continuous and $\calF_{\exp}$-$\gamma$-equicontinuous $C$-cosine family generated by $A^2$.
    Another application of \prettyref{thm:bi_cont_quasi_equicont_C_seq} \ref{it:bi_quasi_equicont_rom1} 
    yields that $(S(t))_{t\geq 0}$ is locally $\topo$-bi-continuous and exponentially bounded.
\end{proof}

\printbibliography

\end{document}